\documentclass{article}

\usepackage{lineno,hyperref}
\usepackage{amsmath}
\usepackage{amssymb}
\usepackage{amsthm}
\modulolinenumbers[5]
\usepackage{enumitem}
\usepackage[all]{xy}
\SelectTips{cm}{12}

\newtheorem{lemma}{Lemma}
\newtheorem*{lemma*}{Lemma}
\newtheorem{prop}{Proposition}
\newtheorem{theorem}{Theorem}
\newtheorem*{theorem*}{Theorem}

\theoremstyle{definition}

\newcommand\dotminus{\mathbin{\dot{-}}}

\newcommand\dotoplus{\mathbin{\dot{\oplus}}}

\DeclareMathOperator\elem{E}

\DeclareMathOperator\stlin{St}

\DeclareMathOperator\Cent{C}

\DeclareMathOperator\Ker{Ker}

\newcommand\e{{\mathrm e}}
\newcommand\R{{\mathcal R}}

\newcommand\eps{\varepsilon}
\newcommand\leqt{\trianglelefteq}

\DeclareMathOperator\Aut{Aut}
\DeclareMathOperator\Image{Im}
\DeclareMathOperator\Spec{Spec}

\newcommand\fp{{\mathrm{fp}}}

\DeclareMathOperator\ev{ev}
\DeclareMathOperator\colim{colim}
\DeclareMathOperator\schur{M}

\newcommand{\up}[2]{{^{#1}\!{#2}}}
\newcommand{\interval}[2]{\,{]}{#1}, {#2}{[}\,}

\newcommand{\Set}{\mathbf{Set}}
\newcommand{\Group}{\mathbf{Grp}}

\newcommand{\Ring}{\mathbf{Ring}}

\DeclareMathOperator\Pro{Pro}
\DeclareMathOperator\Ind{Ind}
\DeclareMathOperator\Ex{Ex}

\newcommand{\Cat}{\mathbf{Cat}}

\title{
	Locally isotropic Steinberg groups II. \\
	Schur multipliers
}

\author{
	Egor Voronetsky\thanks{
		Research is supported by the Russian Science Foundation grant 19-71-30002.
	} \\
	Chebyshev Laboratory, \\
	St. Petersburg State University, \\
	14th Line V.O., 29B, \\
	Saint Petersburg 199178 Russia \\
}

\begin{document}
\maketitle

\begin{abstract}
	We compute Schur multipliers of locally isotropic Steinberg groups and of all root graded Steinberg groups with root systems of rank at least
	\( 3 \). As an application, we show that locally isotropic Steinberg groups are well defined as abstract groups.
\end{abstract}


\section{Introduction}

Let
\( G \) be a sufficiently isotropic reductive group. Its elementary subgroup
\( \elem_G \leqt G \) is normal and perfect, so it is natural to study its universal central extension. As an approximation, it is relatively easy to define the Steinberg group
\( \stlin_G \to \elem_G \). This group is a central perfect extension with simple presentation, but not necessarily centrally closed.

For Chevalley groups the Schur multipliers of
\( \stlin_G \) are found by van der Kallen and Stein \cite{schur-mult} (see also the earlier work \cite{centr-closed} of Stein). They also computed the answer for the
\( \mathsf A_\ell \)-graded case, i.e. for linear Steinberg groups over arbitrary unital associative rings, even though such groups are not necessarily central extensions of the elementary groups. It turns out that
\( \stlin_G \) is centrally closed unless the root system is small and the base ring has small residue fields. Such Steinberg groups are actually central extensions of the elementary groups, see \cite{f4-central} and further references therein.

The only other known results for sufficiently general rings (i.e. not local ones) are that under some assumptions odd unitary Steinberg groups are centrally closed \cite{schur-u-odd, schur-u-even} by Lavrenov and Steinberg groups over Jordan pairs are centrally closed \cite{loos-neher} by Loos and Neher. Note that our result does not cover all odd unitary groups because in general they have only
\( \mathsf D_\ell \)-Weyl elements, but
\( \mathsf B_\ell \)-root subgroups.

We started to study locally isotropic Steinberg groups in \cite{iso-st-k2}. These groups are defined not only for Chevalley groups and some their globally isotropic variations such as general linear groups over Azumaya algebras, but also for groups without globally defined proper parabolic subgroups at all. We proved that
\( \stlin_G \to G \) is a crossed module, in particular, the Steinberg group is a central extension of its image.

For example, let
\( P \) be a finite projective module over a unital commutative ring
\( K \) of rank at least
\( 4 \) at every maximal ideal. Then group
\( \Aut_K(P) \) is locally isotropic (of local isotropic rank at least
\( 3 \)), but in general it has no globally defined proper parabolic subgroups. It is easy to construct explicit such projective modules at least without unimodular vectors, e.g. if
\[
	K
	=
	\mathbb R[x_0^2 + \ldots + x_n^2]
	/
	(x_0^2 + \ldots + x_n^2 - 1)
\]
is the ring of rational functions on the unit sphere and
\[
	P
	=
	\bigl\{
		\vec v \in K^{n + 1}
		\mid
		\sum_{i = 0}^n v_i x_i = 0
	\bigr\}
\]
is the tangent bundle, then it is well known that
\( P \) has no unimodular vectors for even
\( n \). On the other hand, by Serre's theorem \cite[\S 8]{bass} if the rank of
\( P \) at every maximal ideal is more than the Bass--Serre dimension
\( \delta(K) \), then
\( P \) has a unimodular vector, so in this setting our construction is interesting only for high-dimensional rings.

In this paper we compute the Schur multiplier of
\( \stlin_G \) for any locally isotropic reductive group
\( G \) (theorem \ref{schur-loc}). Again it turns out that the Steinberg group is centrally closed unless the rank of
\( G \) is small and the base ring has small residue fields. So our result generalizes not only the case of Chevalley groups, but also the classical computation of Schur multipliers of finite simple groups of Lie type \cite{griess, steinberg-1, steinberg-2} (for rank at least
\( 3 \)).

Moreover, we compute in theorem \ref{schur-graded} Schur multipliers of all abstract Steinberg groups graded by root systems of rank at least
\( 3 \). Such groups have nice algebraic description \cite{h-graded, root-graded, wiedemann}, and we use this setting to almost uniformly treat all possible Tits indices of reductive groups.

Finally, in theorem \ref{compactness} we apply the knowledge of Schur multiplier to deduce that locally isotropic Steinberg groups actually exist as ordinary groups.

\section{Preliminaries}

\subsection{Locally isotropic reductive groups and root graded groups}

In this paper we deal with both crystallographic and ``spherical'' root systems. Recall that a \textit{spherical root system}
\(
	\Phi
	\subseteq
	\mathbb R^\ell \setminus \{ 0 \}
\) is a finite spanning subset closed under reflections through the orthogonal complements to the roots and such that the only multiples of a root
\( \alpha \) belonging to
\( \Phi \) are
\( \pm \alpha \). We usually consider spherical root systems up to the projection on the unit sphere.

On the other hand, in a \textit{crystallographic root system} we remove the condition on multiples, but require that
\(
	2 \frac{(\alpha, \beta)}{(\alpha, \alpha)}
	\in
	\mathbb Z
\) for all roots
\( \alpha, \beta \in \Phi \), so the roots span a lattice closed under the action of the Weyl group
\( \mathrm W(\Phi) \) (generated by the reflections). In particular, for every root
\( \alpha \) the only its multiples from
\( \Phi \) may be
\( \pm \alpha \),
\( \pm \frac 1 2 \alpha \), or
\( \pm 2 \alpha \). Every crystallographic root system determines a corresponding spherical root system, e.g. by the projection on the unit sphere.

Every root system has unique decomposition into disjoint union of pairwise orthogonal indecomposable root subsystems (its \textit{components}). Up to isomorphism (and the projection on the unit sphere) the only indecomposable spherical root systems are
\( \mathsf A_\ell \) (%
\( \ell \geq 1 \)),
\( \mathsf B_\ell \) (%
\( \ell \geq 2 \)),
\( \mathsf D_\ell \) (%
\( \ell \geq 4 \)),
\( \mathsf E_\ell \) (%
\( \ell \in \{ 6, 7, 8 \} \)),
\( \mathsf F_4 \),
\( \mathsf G_2 \),
\( \mathsf H_\ell \) (%
\( \ell \in \{ 2, 3, 4 \} \)), and
\( \mathsf I_2^m \) (%
\( m \geq 7 \)). Their \textit{rank}, i.e. the dimension of the span, is the lower index. The only indecomposable crystallographic root systems (up to isomorphism and scaling) are
\( \mathsf A_\ell \) (%
\( \ell \geq 1 \)),
\( \mathsf B_\ell \) (%
\( \ell \geq 2 \)),
\( \mathsf C_\ell \) (%
\( \ell \geq 3 \)),
\( \mathsf{BC}_\ell \) (%
\( \ell \geq 1 \)),
\( \mathsf D_\ell \) (%
\( \ell \geq 4 \)),
\( \mathsf E_\ell \) (%
\( \ell \in \{ 6, 7, 8 \} \)),
\( \mathsf F_4 \), and
\( \mathsf G_2 \). After projection on the unit sphere all of them become indecomposable spherical root systems with the same notation except
\(
	\mathsf C_\ell,
	\mathrm{BC}_\ell
	\mapsto
	\mathsf B_\ell
\), where
\( \mathsf B_1 = \mathsf A_1 \).

A subset
\( \Sigma \subseteq \Phi \) of a spherical root system is called \textit{unipotent} if it is an intersection of
\( \Phi \) with a convex cone not containing a line, so in particular it does not contain opposite roots. A root
\( \alpha \in \Sigma \) in a unipotent subset is called \textit{extreme} if the ray
\( \mathbb R_{\geq 0} \alpha \) is extreme in the convex cone spanned by
\( \Sigma \), i.e. if
\( \alpha \) does not lie in any open angle spanned by two linearly independent points of
\( \mathrm{conv}(\Sigma) \). For example, if
\( \alpha \) and
\( \beta \) are linearly independent roots, then the \textit{interval}
\[
	\interval \alpha \beta
	=
	\Phi
	\cap
	(
		\mathbb R_{> 0} \alpha
		+ \mathbb R_{> 0} \beta
	)
\]
is unipotent. Note that the usual definition of unipotent subsets of crystallographic root systems (i.e. that
\( (\Sigma + \Sigma) \cap \Phi \subseteq \Sigma \) and
\( \Sigma \cap (-\Sigma) = \varnothing \)) is strictly weaker than our ``convexity'' condition, e.g. a pair of orthogonal long roots in
\( \mathsf B_2 \) is not unipotent by our definition.

Recall the notation from \cite[\S 2.2]{iso-st-k2}. An \textit{isotropic pinning} of reductive group scheme
\( G \) over a unital commutative ring
\( K \) consists of a split torus
\( T = \mathbb G^\ell_{\mathrm m} \leq G \), a crystallographic root system
\( \Phi \leq \mathbb Z^\ell \) in its group of constant characters (with respect to some dot product, not necessarily the standard one), and a base
\( \Delta \subseteq \Phi \) such that
\( \Phi \) is the set of non-zero weights of the Lie algebra
\( \mathfrak g \) of
\( G \) with respect to
\( T \) (or
\( K = 0 \)), all root subspaces are free
\( K \)-modules, and there are
\( \alpha \)-Weyl elements in
\( G(K) \) for all
\( \alpha \in \Phi \). If
\( K \) is local (or an LG-ring with connected spectrum \cite{gille}), then
\( G \) has a maximal isotropic pinning and all such isotropic pinnings are conjugate.

The \textit{rank} of an isotropic pinning is the smallest rank of components of
\( \Phi \) unless
\( T \) commutes with some non-abelian simple factor-group in a geometric fiber, in this exceptional case the rank is zero. The \textit{isotropic rank} of a reductive group scheme
\( G \) over a local ring is the rank of its maximal isotropic pinnings, and the \textit{local isotropic rank} of a reductive group scheme
\( G \) over arbitrary unital commutative ring is the minimum of the isotropic ranks of the localizations
\( G_{\mathfrak m} \) for all maximal ideals
\( \mathfrak m \).

An abstract group
\( G \) is called \textit{root graded} by a spherical root system
\( \Phi \) if there are subgroups
\( G_\alpha \leq G \) for
\( \alpha \in \Phi \) such that
\begin{itemize}

	\item
	\(
		[G_\alpha, G_\beta]
		\leq
		\bigl\langle
			G_\gamma
			\mid
			\gamma \in \interval \alpha \beta
		\bigr\rangle
	\) for linearly independent
	\( \alpha \) and
	\( \beta \);

	\item
	\(
		G_\alpha
		\cap
		\bigl\langle
			G_\beta
			\mid
			\beta \in \Sigma \setminus \{ \alpha \}
		\bigr\rangle
		=
		1
	\) for all unipotent subsets
	\( \Sigma \subseteq \Phi \) with an extreme root
	\( \alpha \);

	\item
	for every root
	\( \alpha \) there exists an element
	\( n \in G_\alpha G_{-\alpha} G_\alpha \) with the property
	\(
		\up n {G_\beta}
		=
		G_{
			\beta
			-
			2
			\frac{(\alpha, \beta)}{(\alpha, \alpha)}
			\alpha
		}
	\) for all roots
	\( \beta \) (it is called an \textit{%
		\( \alpha \)-Weyl element%
	}).

\end{itemize}

For example, if
\( G \) is a reductive group scheme with an isotropic pinning, then the point group
\( G(K) \) is root graded by the corresponding spherical root system. In
\cite{h-graded, wiedemann} a stronger variant of the second condition is used,
\[
	G_\alpha
	\cap
	\bigl\langle G_\beta \mid \beta \in \Pi \bigr\rangle
	=
	1
\]
for all subsystems of positive roots
\( \Pi \subseteq \Phi \) and all
\( \alpha \in \Phi \setminus \Pi \), so we formally need this condition in theorem
\ref{schur-graded} for the
\( \mathsf H_\ell \)-graded case.

\subsection{%
	Unital
	\( \Phi \)-rings%
}

We need the explicit descriptions from \cite{root-graded, wiedemann} of root graded groups up to isogeny for all indecomposable spherical root systems of rank
\( \geq 3 \) excluding
\( \mathsf H_3 \) and
\( \mathsf H_4 \) (for these exceptional cases see \cite{h-graded}, they reduce to the cases
\( \mathsf D_6 \) and
\( \mathsf E_8 \) respectively). Here ``up to isogeny'' means that we actually describe only the isomorphism classes of the subgroups
\( G_\alpha \), the maps
\( G_\alpha \times G_\beta \to G_\gamma \) for
\( \gamma \in \interval \alpha \beta \) from the commutator relations, and the action of a distinguished family of Weyl elements on these subgroups. For all root systems there is a one-to-one correspondence between such isogeny classes and algebras
\( A \) from a certain variety (two-sorted in the doubly laced cases
\( \mathsf B_\ell \) and
\( \mathsf F_4 \)), namely, root subgroups from each orbit under the action of the Weyl group are isomorphic to the corresponding part of
\( A \). We call these algebras \textit{%
	unital
	\( \Phi \)-rings%
}, in \cite{root-graded} they are called
\( (\Phi, \Phi) \)-rings. For some root systems the varieties of such algebras coincide, there are only
\( 5 \) distinct varieties in total (see below). It is convenient to choose the canonical group in the isogeny class corresponding to a unital
\( \Phi \)-ring
\( A \), namely, the \textit{Steinberg group}
\( \stlin(\Phi, A) \) is the group generated by the subgroups
\( G_\alpha \) and the only additional relations taken from the commutator formula.

Unital
\( \mathsf A_\ell \)-rings are just unital associative rings. For every such ring
\( R \) the Steinberg group
\( \stlin(\mathsf A_\ell, R) \) is generated by
\( x_{ij}(p) \) for distinct
\( i, j \in \{ 1, \ldots, \ell + 1 \} \) and
\( p \in R \). The relations are
\begin{align*}
	x_{ij}(p)\, x_{ij}(q) = x_{ij}(p + q),
	\quad
	[x_{ij}(p), x_{jk}(q)] = x_{ik}(p q)
	\text{ for }
	i \neq k,
	\quad
	[x_{ij}(p), x_{kl}(q)] = 1
	\text{ for }
	j \neq k
	\text{ and }
	i \neq l,
\end{align*}
and the root subgroup with the root
\( \e_i - \e_j \) is just
\( x_{ij}(R) \).

Unital
\( \mathsf D_\ell \)- and
\( \mathsf E_\ell \)-rings are unital commutative rings. For every such ring
\( K \) the Steinberg group
\( \stlin(\Phi, K) \) is the usual Steinberg group constructed by the split reductive group scheme over
\( K \). Namely, it is generated by
\( x_\alpha(p) \) for
\( \alpha \in \Phi \),
\( p \in K \) with the only relations
\begin{align*}
	x_\alpha(p) x_\alpha(q)
	&=
	x_\alpha(p + q),
	\\
	[x_\alpha(p), x_\beta(q)]
	&=
	x_{\alpha + \beta}(N_{\alpha \beta} p q)
	\text{ for }
	\alpha + \beta \in \Phi,
	\\
	[x_\alpha(p), x_\beta(q)]
	&=
	1
	\text{ for }
	\alpha + \beta \notin \Phi \cup \{ 0 \},
\end{align*}
where
\( N_{\alpha \beta} \in \{ \pm 1 \} \) are the structure constants. The root subgroup with the root
\( \alpha \) is precisely
\( x_\alpha(K) \).

Before proceeding with the doubly laced cases, let us recall some non-associative ring theory necessary for unital
\( \mathsf B_3 \)- and
\( \mathsf F_4 \)-rings. Let
\( R \) be an arbitrary ring, i.e. an abelian group with biadditive multiplication. The \textit{associator} of elements
\( p, q, r \in R \) is
\( [p, q, r] = (p q) r - p (q r) \). The \textit{nucleus} of
\( R \) is the set
\[
	\mathrm N(R)
	=
	\{
		\nu \in R
		\mid
		[\nu, R, R] = [R, \nu, R] = [R, R, \nu] = 0
	\},
\]
and the \textit{center} of
\( R \) is
\[
	\mathrm C(R)
	=
	\{
		\zeta \in \mathrm N(R)
		\mid
		\zeta r = r \zeta
		\text{ for all }
		r \in R
	\}
\]
We say that an element of
\( R \) is \textit{nuclear} if it lies in the nucleus, i.e. if it associates with everything, and similar for \textit{centrality}. A ring
\( R \) is called \textit{alternative} if it satisfies the identities
\( [p, p, q] = [p, q, q] = 0 \), so the associator is skew-symmetric. Associative rings and composition algebras over unital commutative rings are alternative.

Every unital
\( \mathsf B_\ell \)-ring
\( (R, \Delta) \) consists of \cite[\S 7]{root-graded}
\begin{itemize}

	\item
	an alternative (associative if
	\( \ell \geq 4 \)) unital ring
	\( R \) with a distinguished nuclear invertible element
	\( \lambda \in R \) such that
	\[
		\lambda^{\pm 1} [p, q, r]
		=
		[\lambda^{\pm 1} p, q, r]
		=
		[p \lambda^{\pm 1}, q, r]
		=
		[p, q, r] \lambda^{\pm 1}
		=
		-[p, q, r];
	\]

	\item
	an anti-endomorphism
	\( (-)^* \colon R \to R \), i.e.
	\( (p + q)^* = p^* + q^* \),
	\( 1^* = 1 \), and
	\( (p q)^* = q^* p^* \), such that
	\begin{align*}
		\lambda^* &= \lambda^{-1},
		&
		p^{**} &= \lambda p \lambda^{-1},
		&
		[p^*, q, r] &= [p, q, r]^* = -[p, q, r];
	\end{align*}

	\item
	a group
	\( \Delta \) with the group operation
	\( \dotplus \);

	\item
	an additive map
	\( \phi \colon R \to \Delta \) with central image satisfying the identities
	\[
		\phi(p + p^* \lambda) = \dot 0,
		\quad
		\phi((p q) r) = \phi(p (q r));
	\]

	\item
	a biadditive map
	\(
		\langle {-}, {=} \rangle
		\colon
		\Delta \times \Delta
		\to
		R
	\) with nuclear image satisfying the identities
	\[
		\langle v, u \rangle
		=
		\langle u, v \rangle^* \lambda,
		\quad
		\langle u, \phi(p) \rangle = 0,
		\quad
		u \dotplus v
		=
		v
		\dotplus
		u
		\dotminus
		\phi(\langle u, v \rangle);
	\]

	\item
	a map
	\( \rho \colon \Delta \to R \) with nuclear image satisfying the identities
	\[
		\rho(u \dotplus v)
		=
		\rho(u) - \langle u, v \rangle + \rho(v),
		\quad
		\rho(\phi(p)) = p - p^* \lambda,
		\quad
		0
		=
		\rho(u)
		+
		\langle u, u \rangle
		+
		\rho(u)^* \lambda;
	\]

	\item
	a distinguished element
	\( \iota \in \Delta \) such that
	\( \rho(\iota) = 1 \);

	\item
	a map
	\(
		\Delta \times R \to \Delta,\,
		(u, p) \mapsto u \cdot p
	\) additive on the first argument and such that
	\begin{align*}
		u \cdot (p + q)
		&=
		u \cdot p
		\dotplus
		\phi(q^* \rho(u) p)
		\dotplus
		u \cdot q,
		&
		\langle u, v \cdot p \rangle
		&=
		\langle u, v \rangle p,
		\\
		u \cdot 1 &= u,
		&
		\phi(p) \cdot q &= \phi(q^* p q),
		\\
		(u \cdot p) \cdot q &= u \cdot p q,
		&
		\rho(u \cdot p) &= p^* \rho(u) p.
	\end{align*}

\end{itemize}

In unital
\( \mathsf B_\ell \)-rings the operations satisfy additional identities
\begin{align*}
	[p^*, p, q] &= 0,
	&
	\langle \iota, \iota \rangle &= -1 - \lambda,
	\\
	\lambda^{**} &= \lambda,
	&
	u \cdot 0 &= \dot 0,
	\\
	\langle \phi(p), u \rangle &= 0,
	&
	u \dotplus u \cdot (-1) &= \phi(\rho(u)),
	\\
	\rho(\dot 0) &= 0,
	&
	u \cdot (p q) r &= u \cdot p (q r),
	\\
	\rho(\dotminus u) &= \rho(u)^* \lambda,
	&
	\langle u \cdot p, v \rangle
	&=
	p^* \langle u, v \rangle.
\end{align*}

The initial unital
\( \mathsf B_\ell \)-ring is
\[
	R
	=
	\mathbb Z[\lambda, \lambda^{-1}],
	\quad
	\Delta
	=
	\phi(\mathbb Z[\lambda^{-1}])
	\dotoplus
	\bigoplus_{m \in \mathbb Z}^\cdot
		\iota \cdot \mathbb Z \lambda^m
\]
independent of
\( \ell \), see \cite[example 2]{root-graded} for details. Non-associative
\( \mathsf B_3 \)-algebras arise e.g. in twisted forms of the Chevalley group scheme
\( \mathsf G(\mathsf E_7, {-}) \), see \cite[example 3]{root-graded}.

The following lemma implies the two remaining axioms
\[
	p (q^* \rho(u) q) = (p q^*) \rho(u) q,
	\quad
	(p^* \rho(u) p) q = p^* \rho(u) (p q)
\]
from \cite[\S 7]{root-graded}.
\begin{lemma}
	\label{inv-artin}
	Let
	\( (R, \Delta) \) be a unital
	\( \mathsf B_3 \)-ring and
	\( x, y \in R \) be two elements. Then the subring
	\( R' \) generated by
	\( x \),
	\( y \),
	\( x^* \),
	\( y^* \), and some nuclear elements is associative. Moreover, if
	\( z \in R \) is another element and
	\( [x, y, z] = 0 \), then the subring
	\( R'' \) generated by
	\( R' \),
	\( z \),
	\( z^* \) is also associative. For any
	\( * \)-invariant subring
	\(
		\rho(\Delta)
		\subseteq
		\widetilde R
		\subseteq
		R
	\) (e.g.
	\( \widetilde R = R' \) or
	\( \widetilde R = R'' \) if we take at least all
	\( \rho(u) \) and
	\( \lambda^{\pm 1} \) as generators) the pair
	\( (\widetilde R, \Delta) \) is a unital
	\( \mathsf B_3 \)-subring of
	\( (R, \Delta) \).
\end{lemma}
\begin{proof}
	Since
	\( x - x^* \lambda = \rho(\phi(x)) \) and
	\( y - y^* \lambda = \rho(\phi(y)) \), we may assume that the ring
	\( R' \) is generated by
	\( x \),
	\( y \), and nuclear elements. It suffices to check that
	\( (a b) c = a (b c) \), where
	\( a \),
	\( b \),
	\( c \) are products of
	\( x \),
	\( y \), and nuclear elements. We prove this by induction on the sum of lengths of their decompositions. If one of
	\( a \),
	\( b \), or
	\( c \) is trivial (i.e.
	\( 1 \)), then there is nothing to prove, otherwise the decompositions of
	\( a \),
	\( b \), and
	\( c \) are independent on the brackets by the induction hypothesis. If
	\( a = \nu a' \) begins with a nuclear factor
	\( \nu \), then
	\[
		(a b) c
		=
		((\nu a') b) c
		=
		(\nu (a' b)) c
		=
		\nu ((a' b) c)
		=
		\nu (a' (b c))
		=
		(\nu a') (b c)
		=
		a (b c),
	\]
	and similarly if
	\( c \) ends by a nuclear factor. But the associator is skew-symmetric, so the elements
	\( a \),
	\( b \),
	\( c \) associate if at least one of them begins or ends by a nuclear factor.

	If
	\( a = a' \nu a'' \) has a nuclear factor in the middle, then
	\[
		(a b) c
		=
		(a' \nu a'' b) c
		=
		(a' \nu) (a'' b c)
		=
		(a' \nu a'') (b c)
		=
		a (b c)
	\]
	by the induction hypothesis and the already proved case. So the only remaining case is where all elements
	\( a \),
	\( b \),
	\( c \) are products of
	\( x \) and
	\( y \). This is precisely Artin's theorem \cite[theorem 2.3.2]{zhevlakov}.

	The second claim may be proved in the same way using that a
	\( 3 \)-generated subring of an alternative ring is associative if and only if the generators associate \cite[theorem I.2]{bruck}. The third claim is obvious because
	\( \lambda = \rho(\dotminus \iota) \),
	\( \lambda^{-1} = \lambda^* \), and
	\(
		\langle u, v \rangle
		=
		\rho(u) + \rho(v) - \rho(u \dotplus v)
	\).
\end{proof}

In order to describe the Steinberg group
\( \stlin(\mathsf B_\ell, R, \Delta) \) let
\( \lambda_i = \lambda \) for
\( -\ell \leq i \leq -1 \) and
\( \lambda_i = 1 \) for
\( 1 \leq i \leq \ell \), so always
\( \lambda_i \in \{ 1, \lambda \} \) are invertible nuclear elements and
\( \lambda_i \lambda_{-i} = \lambda \) (in \cite{root-graded} we used the alternative notation
\( \lambda_i = \lambda^{(1 - \eps_i) / 2} \), where
\( \eps_i = 1 \) for
\( i > 0 \) and
\( \eps_i = -1 \) for
\( i < 0 \)). The Steinberg group is generated by
\( x_{ij}(p) \) for
\( 1 \leq |i|, |j| \leq \ell \),
\( i \neq \pm j \),
\( p \in R \) (with the corresponding root
\( \e_j - \e_i \)) and by
\( x_i(u) \) for
\( 1 \leq |i| \leq \ell \),
\( u \in \Delta \) (with the corresponding root
\( \e_i \)), here
\( \e_{-i} = -\e_i \). The relations are
\begin{align*}
	x_{ij}(p)
	&=
	x_{-j, -i}(-\lambda_j^* p^* \lambda_i),
	&&
	\\
	x_{ij}(p)\, x_{ij}(q) &= x_{ij}(p + q),
	&
	x_i(u)\, x_i(v) &= x_i(u \dotplus v),
	\\
	[x_{-i, j}(p), x_{ji}(q)]
	&=
	x_i(\phi(\lambda_i p q)),
	&
	[x_i(u), x_j(v)]
	&=
	x_{-i, j}(-\lambda_i^* \langle u, v \rangle)
	\text{ for }
	i \neq \pm j,
	\\
	[x_{ij}(p), x_{kl}(q)] &= 1
	\text{ for }
	\{ -i, j \} \cap \{ k, -l \} = \varnothing,
	&
	[x_i(u), x_{jk}(p)] &= 1
	\text{ for }
	i \notin \{ j, -k \},
	\\
	[x_{ij}(p), x_{jk}(q)] &= x_{ik}(p q)
	\text{ for }
	i \neq \pm k,
	&
	[x_i(u), x_{ij}(p)]
	&=
	x_{-i, j}(\lambda_i^* \rho(u) p)\,
	x_j(\dotminus u \cdot (-p)).
\end{align*}

Finally, a unital
\( \mathsf F_4 \)-ring
\( (R, S) \) consists of \cite[\S 8]{root-graded}
\begin{itemize}

	\item
	unital alternative rings
	\( R \) and
	\( S \);

	\item
	involutions
	\( p \mapsto p^* \) on both
	\( R \) and
	\( S \) (i.e. anti-endomorphisms such that
	\( p^{**} = p \)) with the properties
	\[
		[p^*, q, r] = -[p, q, r],
		\quad
		[p, q, r]^* = -[p, q, r];
	\]

	\item
	additive homomorphisms
	\( \phi \colon R \to \mathrm C(S) \) and
	\( \phi \colon S \to \mathrm C(R) \) such that
	\begin{align*}
		\phi(p q) &= \phi(q p),
		&
		\phi(p^*) &= \phi(p)^* = \phi(p),
		\\
		\phi((p q) r) &= \phi(p (q r)),
		&
		\phi(\phi(p)) &= 0;
	\end{align*}

	\item
	multiplicative homomorphisms
	\( \rho \colon R \to \mathrm C(S) \) and
	\( \rho \colon S \to \mathrm C(R) \) (i.e.
	\( \rho(1) = 1 \) and
	\( \rho(u v) = \rho(u) \rho(v) \)) such that
	\begin{align*}
		u \phi(p) &= \phi(\rho(u) p),
		&&
		\\
		\rho(u^*) &= \rho(u)^* = \rho(u),
		&
		\rho(u \dotplus v)
		&=
		\rho(u) + \phi(u v^*) + \rho(v),
		\\
		\rho(\rho(u)) &= u u^*,
		&
		u + u^* &= \rho(\phi(u)) + \phi(\rho(u)).
	\end{align*}

\end{itemize}
This variety is symmetric because
\( \mathsf F_4 \) has an outer automorphism. The initial
\( \mathsf F_4 \)-ring is
\( R = S = \mathbb Z[\lambda] / (\lambda^2 - 1) \), where
\[
	\phi(p + q \lambda)
	=
	(p - q) + (p - q) \lambda,
	\quad
	\rho(u + v \lambda)
	=
	\textstyle\frac {(u - v)^2 + u + v} 2
	+
	\frac {(u - v)^2 - u - v} 2 \lambda,
	\quad
	\rho(-1) = \lambda.
\]
Every
\( \mathsf F_4 \)-ring
\( (R, S) \) may be considered as an
\( \mathsf B_3 \)-ring with
\( \iota = 1_S \),
\( u \cdot p = \rho(p) u \), and
\( \langle u, v \rangle = -\phi(u v^*) \).

We consider
\( \mathsf F_4 \) as a crystallographic root system in order to refer to is roots as ``long'' and ''short'' ones. The Steinberg group is generated by
\( x_\alpha(p) \), where
\( p \in R \) for long
\( \alpha \) and
\( p \in S \) for short
\( \alpha \). The relations are
\begin{align*}
	x_\alpha(p)\, x_\alpha(q) &= x_\alpha(p + q),
	\\
	[x_\alpha(p), x_\beta(q)]
	&=
	\prod_{\gamma \in \interval \alpha \beta}
		x_\gamma(C_{\alpha \beta}^\gamma(p, q))
	\text{ for linearly independent }
	\alpha
	\text{ and }
	\beta,
\end{align*}
where
\( C_{\alpha \beta}^\gamma(p, q) \) are certain ``structure terms'' from the following table.

\begin{center}
	\begin{tabular}{|c|c|}
		\hline
		
		the set of possible
		\( C_{\alpha \beta}^\gamma \)
		&
		conditions
		\\ \hline \hline
		
		\( \pm \phi(p q) \),
		\( \pm \phi(p^* q) \)
		&
		\( \angle(\alpha, \beta) = \frac \pi 2 \),
		\( |\alpha| = |\beta| \)
		\\ \hline
		
		\( \upsilon \widehat p \widehat q \),
		\( \upsilon \widehat q \widehat p \),
		where
		\( \upsilon \in \Upsilon \),
		\( \widehat p \in \{ p, p^* \} \),
		\( \widehat q \in \{ q, q^* \} \)
		&
		\( \angle(\alpha, \beta) = 2 \pi / 3 \),
		\( |\alpha| = |\beta| \)
		\\ \hline
		
		\( \upsilon \rho(q) p \),
		\( \upsilon \rho(q) p^* \),
		where
		\( \upsilon \in \Upsilon \)
		&
		\( \angle(\alpha, \beta) = 3 \pi / 4 \),
		\( |\alpha| = |\gamma| \neq |\beta| \)
		\\ \hline
		
	\end{tabular}
\end{center}

Here
\( \Upsilon = \{ \pm 1, \pm \lambda \} \) is the group of invertible elements of both components of the free unital
\( \mathsf F_4 \)-ring and
\( \lambda = \rho(-1) \). Similarly to the cases of
\( \mathsf D_\ell \) and
\( \mathsf E_\ell \) (where the signs of
\( x_\alpha \) may be changed), the structure terms depend on the choice of parametrizations
\( x_\alpha \), each of them may be precomposed with a term from the set
\(
	\{
		\upsilon p,
		\upsilon p^*
		\mid
		\upsilon \in \Upsilon
	\}
\). Non-associative
\( \mathsf F_4 \)-rings arise in descriptions of some twisted forms of
\( \mathrm G(\mathsf E_8, {-}) \), see \cite[example 6]{root-graded}.

\subsection{Colocalization and infinitary pretopoi}

Since our ultimate goal is to calculate Schur multipliers of locally isotropic Steinberg groups, let us briefly recall their construction from \cite{iso-st-k2}. Let
\( G \) be a reductive group scheme of local isotropic rank
\( \geq 3 \) over a unital commutative ring
\( K \).

By definition and a direct limit argument, there is an open covering
\( \Spec(K) = \bigcup_i \mathcal D(s_i) \) such that every group scheme
\( G_{s_i} \) over
\( K_{s_i} \) has an isotropic pinning of rank
\( \geq 3 \). Clearly,
\( G(K_{s_i}) = G_{s_i}(K_{s_i}) \) is root graded and thus has a corresponding Steinberg group
\( \stlin_G(K_{s_i}) \). Unfortunately, Steinberg groups do not form a sheaf in Zariski topology and we cannot reconstruct a group
\( \stlin_G(K) \) by these local parts.

On the other hand, let us consider the formal direct limit
\[
	K^{(s_i^\infty)}
	=
	\varprojlim\nolimits^{\Pro}(
		\ldots
		\xrightarrow{s_i} K
		\xrightarrow{s_i}
	)
\]
in the pro-completion
\( \Pro(\Set) \) of the category of sets, the \textit{colocalization} of
\( K \) at
\( s_i \). Such object is in a way dual to the localization
\( K_{s_i} \), in particular, it is a non-unital
\( K_{s_i} \)-algebra and together they form a cosheaf in the Zariski topology.

Since the Steinberg group is well-defined over non-unital rings (we shortly generalize this observation to all
\( \Phi \)-rings), there is a group object
\( \stlin_G(K^{(s_i^\infty)}) \) given by explicit generators and relations. But the category
\( \Pro(\Set) \) is not suitable for this construction, so we embed it into a larger category
\( \Ex(\Ind(\Pro(\Set))) \) by ind-completion and exact completion. To handle functorial properties and to apply the generic element method, it is also useful to replace
\( \Set \) by the category
\[
	\mathbf P_K = \Cat(\Ring_K^\fp, \Set)
\]
of presheaves on the category of all finitely presented unital commutative
\( K \)-algebras. The base ring object in this category is the presheaf
\( \R \colon R \mapsto R \), i.e. sheaf of points of the affine line. The colocal Steinberg groups
\( \stlin_G(\R^{(s_i^\infty)}) \) are thus constructed in
\[
	\mathbf U_K = \Ex(\Ind(\Pro(\mathbf P_K))).
\]

The category
\( \mathbf U_K \) is convenient because it is an \textit{infinitary pretopos} just as
\( \Set \). We refer to \cite[A1.4, D1.3]{elephant} for a thorough treatment of the necessary categorical logic. A quick summary and future references relating infinitary pretopoi and ind- and pro-completions may be found in our previous paper \cite[\S 3]{iso-st-k2}. Every infinitary pretopos
\( \mathbf C \) has finite limits and all small colimits, every morphism
\( f \) is a composition of an epimorphism and a monomorphism (in a unique way up to a unique isomorphism), for every object
\( X \) the class of its subobjects is a bounded lattice, and
\( X \) has factor-objects by all equivalence relations.

If we have a collection of morphisms and subobjects in an infinitary pretopos, then new subobjects may be defined using first order formulae of geometric logic. We write
\( [\![ \varphi(\vec x) ]\!] \) for the subobject corresponding to the formula
\( \varphi \) and similarly
\( [\![ t(\vec x) ]\!] \) for the morphism corresponding to the term
\( t \). Inclusions between definable subobjects correspond to sequents and may be proved using natural deduction. For example, every variety of algebras determines the category of such algebras in an infinitary pretopos
\( \mathbf C \). This category has all small colimits (and finite limits), in particular, the initial object. Lemma \ref{inv-artin} holds for all unital
\( \mathsf B_3 \)-ring objects
\( (R, \Delta) \) in
\( \mathbf C \) with literally the same proof translated into the geometric logic. Since the geometric logic lacks universal quantifiers, we cannot define the subobjects
\( \mathrm N(R) \) and
\( \mathrm C(R) \) by geometric first order formulae (at least in the obvious way), but we still can express that a given subobject
\( X \subseteq R \) is nuclear or central.

Another useful technique to prove various results is to apply Yoneda lemma and reduce the problem to
\( \Set \), but only if the problem is expressible in Cartesian logic. For example, our future applications of lemma \ref{inv-artin} to unital
\( \mathsf B_3 \)-ring objects in categories can be easily obtained in this way because we need the associativity identities between pairs of morphisms
\( p, q \colon T \to R \) for some domain
\( T \) and several other nuclear morphisms
\( \nu_i \colon T \to R \). The categorical version of the lemma instead gives the associativity of the subring generated by a family of global elements, i.e. the particular case of the terminal object
\( T = \{ * \} \). On the other hand, the categorical version may be applied to the slice category
\( \mathbf C / T \) (and slice categories of an infinitary pretopos are also infinitary pretopoi) to get the same general associativity result.

Now let us return to Steinberg groups. Recall that a crossed module of group objects in an infinitary pretopos
\( \mathbf C \) is a homomorphism
\( \delta \colon X \to G \) together with an action of
\( G \) on
\( X \) (i.e. a morphism
\( \up{({-})}{({=})} \colon G \times X \to X \) satisfying the obvious identities) such that
\( \delta(\up g x) = g \delta(x) g^{-1} \) and
\( x y x^{-1} = \up {\delta(x)} y \) for
\( x, y \colon X \) and
\( g \colon G \). These equalities actually define certain subobjects of the domains
\( G \times X \) and
\( X \times X \), and the axioms are that the subobjects coincides with the whole domains (i.e. the equalities are ``identities'').

It turns out that Steinberg groups
\( \stlin_G(\R^{(s_i^\infty)}) \) are crossed modules over
\( G(\R) \) in a canonical way and in the category of crossed modules they form a cosheaf in Zariski topology. Thus
\( \stlin_G(\R) \) may be defined by the cosheaf property as a certain finite colimit. Note that
\( \stlin_G(\R^{(s_i^\infty)}) \) definitely do not form a cosheaf of group objects because the Steinberg group over a product of rings is the product of the Steinberg groups over the factors, not the free product.

In general we do not know yet whether
\( \stlin_G(\R) \) lies in the full subcategory
\( \Ex(\Ind(\mathbf P_K)) \subseteq \mathbf U_K \) (but see theorem \ref{compactness} below). Suppose that this holds, e.g. that
\( G \) has an isotropic pinning of rank
\( \geq 3 \). Since
\( \mathbf P_K \) is already an infinitary pretopos, there is a canonical ``evaluation'' functor
\(
	\ev_*
	\colon
	\Ex(\Ind(\mathbf P_K))
	\to
	\mathbf P_K
\). The value of the presheaf
\( \ev_*(\stlin_G(\R)) \) on
\( K \) is
\( \stlin_G(K) = \ev_K(\stlin_G(\R)) \), by definition this is the set-theoretic locally isotropic Steinberg group.

\subsection{%
	Non-unital
	\( \Phi \)-rings
}

Let us generalize Steinberg groups to a non-unital variation of
\( \Phi \)-rings, where
\( \Phi \) is an irreducible spherical root system of rank
\( \geq 3 \), but not from the
\( \mathsf H_\ell \) series. We say that a \textit{%
	non-unital
	\( \Phi \)-ring%
} is the kernel of a homomorphism from a unital
\( \Phi \)-ring
\( A \) to the initial
\( \Phi \)-ring
\( A_0 \) in the corresponding variety. Since any
\( \Phi \)-ring is a group or a pair of groups under addition (though the addition may be non-commutative in the case
\( \Phi = \mathsf B_\ell \)), such kernels are well-defined and they are also groups or pairs of groups. Note that the homomorphism
\( A \to A_0 \) is surjective and has a unique section, so
\( A \) is completely determined by the kernel and the ``action'' of
\( A_0 \) on the kernel. In all cases non-unital
\( \Phi \)-rings form a variety of algebras with finitely many operations and axioms, as a category it is equivalent to the slice category of unital
\( \Phi \)-rings over the initial one.

Namely, non-unital
\( \mathsf A_\ell \)-rings are associative non-unital rings, non-unital
\( \mathsf D_\ell \) and
\( \mathsf E_\ell \)-rings are commutative non-unital rings.

A non-unital
\( \mathsf B_\ell \)-ring
\( (R, \Delta) \)
consists of two groups with the same operations and axioms as in the unital case, but excluding the elements
\( 1 \),
\( \lambda \),
\( \lambda^{-1} \),
\( \iota \) and the axioms involving them. Instead this variety has the additional operations
\begin{align*}
	R &\to R,\,
	p \mapsto p \lambda^{\pm 1},
	&
	\Delta &\to \Delta,\,
	u \mapsto u \cdot \lambda^{\pm 1},
	\\
	R &\to R,\,
	p \mapsto \lambda^{\pm 1} p,
	&
	R &\to \Delta,\,
	p \mapsto \iota \cdot p,
	\\
	\Delta &\to R,\,
	u \mapsto \langle \iota, u \rangle
\end{align*}
satisfying certain axioms.

Finally, a non-unital
\( \mathsf F_4 \)-ring
\( (R, S) \) consists of two alternative non-unital rings with involutions and maps
\( \phi \) and
\( \rho \) between them satisfying the axioms from the unital case. It also has four additional operations
\[
	R \to R,\, p \mapsto p \lambda^{\pm 1}
	\quad
	S \to S,\, u \mapsto u \lambda^{\pm 1}
\]
satisfying some axioms, where
\( \lambda = \rho(-1) \) are the two distinguished elements from the initial unital
\( \mathsf F_4 \)-ring.

For every such root system
\( \Phi \) and non-unital
\( \Phi \)-ring object
\( A \) in an infinitary pretopos
\( \mathbf C \) we may construct the \textit{unrelativized} Steinberg group
\( \stlin(\Phi, A) \) by the usual generators and relations. But without further assumptions this group object does not have usual properties expected from a Steinberg group, e.g. it is abelian if
\( \Phi = \mathsf A_\ell \) and the ring object
\( A \) has zero multiplication.

We say that a non-unital
\( \Phi \)-ring object
\( A \) in an infinitary pretopos
\( \mathbf C \) is \textit{weakly unital} if there is a \textit{weak unit}
\( \mathcal E \subseteq R \) (where
\( R = A \) in the simply laced cases and
\( R \) is the first component in the doubly laced cases) such that

\begin{itemize}
	
	\item
	\( \mathcal E \) is a multiplicative sub-semigroup, i.e.
	\( \mathcal E \mathcal E \subseteq \mathcal E \);

	\item
	\( \mathcal E \) is central and hermitian (i.e.
	\( \eps^* = \eps \) for
	\( \eps \colon \mathcal E \) in the doubly laced cases);
	
	\item
	the first order formula
	\(
		\exists
			\zeta, \eps', \eta' \colon \mathcal E
		\enskip
			\eps = \eps' \zeta
			\wedge
			\eta = \eta' \zeta
	\) holds for
	\( \eps, \eta \colon \mathcal E \) (informally, in the semigroup
	\( \mathcal E \) every two elements have a common divisor);
	
	\item
	the multiplication morphisms
	\begin{align*}
		[\![ \eta \eps ]\!]
		&\colon
		\mathcal E \times \mathcal E \to \mathcal E,
		&
		[\![ u \cdot \eps ]\!]
		&\colon
		\Delta \times \mathcal E \to \Delta
		\text{ for }
		\Phi = \mathsf B_\ell,
		\\
		[\![ p \eps ]\!]
		&\colon
		R \times \mathcal E \to R,
		&
		[\![ s \rho(\eps) ]\!]
		&\colon
		S \times \mathcal E \to S
		\text{ for }
		\Phi = \mathsf F_4
	\end{align*}
	are epimorphisms;
	
	\item
	if
	\( \mu \cdot \eps \) is a multiplication morphism from the previous condition (for
	\( \mu \colon R \),
	\( \mu \colon \Delta \), or
	\( \mu \colon S \)) and
	\( \dot 0 \) is the zero element of the domain of
	\( \mu \), then the first order sequent
	\[
		\mu \cdot \eps = \dot 0
		\vdash
		\exists \eps', \eps'' \colon \mathcal E \enskip
			\eps = \eps' \eps''
			\wedge
			\mu \cdot \eps' = \dot 0
	\]
	holds for the variables
	\( \mu \) and
	\( \eps \colon \mathcal E \).
	
\end{itemize}

\begin{lemma}
	\label{weak-unit}
	For every positive integer
	\( n \) the monomial morphisms
	\begin{align*}
		[\![ \eta \eps^n ]\!]
		&\colon
		\mathcal E \times \mathcal E \to \mathcal E,
		&
		[\![ u \cdot \eps^n ]\!]
		&\colon
		\Delta \times \mathcal E \to \Delta
		\text{ for }
		\Phi = \mathsf B_\ell,
		\\
		[\![ p \eps^n ]\!]
		&\colon
		R \times \mathcal E \to R,
		&
		[\![ s \rho(\eps^n) ]\!]
		&\colon
		S \times \mathcal E \to S
		\text{ for }
		\Phi = \mathsf F_4
	\end{align*}
	are epimorphisms. Moreover, if
	\( \mu \cdot \eps^n \) is one of these morphisms with
	\( \mu \colon R \),
	\( \mu \colon \Delta \), or
	\( \mu \colon S \) and
	\( \dot 0 \) is the zero element of the domain of
	\( \mu \), then the first order sequent
	\[
		\mu \cdot \eps^n = \dot 0
		\vdash
			\exists \eps', \eps'' \colon \mathcal E \enskip
				\eps = \eps' \eps''
				\wedge
				\mu \cdot (\eps')^n = \dot 0
	\]
	holds for the variables
	\( \mu \) and
	\( \eps \colon \mathcal E \).
\end{lemma}
\begin{proof}
	We give informal arguments, they can be straightforwardly translated to a formal proof in geometric logic. For the first claim take
	\( \mu \) with one of the domains
	\( \mathcal E \),
	\( R \),
	\( \Delta \), or
	\( S \). It is of the form
	\( \mu = \mu' \cdot \prod_{i = 1}^n \eps_i \) for some
	\( \eps_i \colon \mathcal E \). Now all
	\( \eps_i \) have a common divisor, i.e.
	\( \eps_i = \eps'_i \eta \) for some
	\( \eps'_i, \eta \colon \mathcal E \), so
	\(
		\mu
		=
		\bigl( \mu' \cdot \prod_{i = 1}^n \eps'_i \bigr)
		\cdot
		\eta^n
	\).
	
	To prove the second claim suppose that
	\( \mu \cdot \eps^n = \dot 0 \). Then there are decompositions
	\( \eps = \eps'_i \eta_i \) for
	\( 1 \leq i \leq n \) such that
	\( \mu \cdot \prod_{i = 1}^n \eps'_i = \dot 0 \). Choose a common divisor of
	\( \eta_i \), i.e.
	\( \eta_i = \eta'_i \zeta \) for
	\( \eta'_i, \zeta \colon \mathcal E \). Since
	\( (\eps'_i \eta'_i - \eps'_j \eta'_j) \zeta = 0 \) for all
	\( i \neq j \), we can finally choose a decomposition
	\( \zeta = \zeta' \zeta'' \) for
	\( \zeta', \zeta'' \colon \mathcal E \) such that
	\( \eps'_i \eta'_i \zeta' = \eps' \) is independent of
	\( i \). In other words,
	\( \eps = \eps' \zeta'' \) and
	\( \mu \cdot (\eps')^n = \dot 0 \).
\end{proof}

For example, if
\( A \) is unital, then we can take
\( \mathcal E = \{ 1 \} \). For the colocalized Steinberg groups
\( \stlin_G(\R^{(s^\infty)}) \) in the locally isotropic case we take
\( \mathcal E = \R^{(\infty, s)} \), see \cite[lemma 5]{iso-st-k2} and its proof. If
\( \mathcal E \) is a weak unit for a non-unital
\( \Phi \)-ring object
\( A \), then the image of
\( [\![ \eps^n ]\!] \colon \mathcal E \to \mathcal E \)
is also a weak unit for every positive integer
\( n \) by lemma \ref{weak-unit}. Note that for
\( \Phi = \mathsf F_4 \) our definition is not symmetric with respect to exchanging
\( R \) and
\( S \), though if
\( \mathcal E \subseteq R \) is a weak unit for
\( (R, S) \), then
\( \rho(\mathcal E) \subseteq S \) is a weak unit for
\( (S, R) \) again by lemma \ref{weak-unit}.

Recall that if
\( A \) is a non-unital
\( \Phi \)-ring in an infinitary pretopos
\( \mathbf C \) and
\( X \in \mathbf C \) is any object, then
\( A_X = A \times X \) is a non-unital
\( \Phi \)-ring in the infinitary pretopos
\( \mathbf C / X \) (this holds for all algebraic structures). If
\( \mathcal E \) is a weak unit of
\( A \), then
\( \mathcal E_X \) is clearly a weak unit of
\( A_X \).

\begin{lemma}
	\label{wunit-artin}
	Let
	\( (R, \Delta) \) be a non-unital
	\( \mathsf B_3 \)-ring object in an infinitary pretopos
	\( \mathbf C \) with a weak unit
	\( \mathcal E \) and
	\( p, q, r \colon X \to R \) be morphisms such that
	\( \bigl[ p(x), q(x), r(x) \bigr] = 0 \). Then there exist an epimorphism
	\( X' \to X \) and an associative subring
	\( R' \subseteq R_{X'} \) such that
	\( (R', \Delta_{X'}) \) is still a non-unital
	\( \mathsf B_3 \)-ring object with a weak unit
	\( \mathcal E_{X'} \) inside
	\( \mathbf C / X' \) and the morphisms
	\( p|_{X'} \),
	\( q|_{X'} \),
	\( r|_{X'} \) factor through
	\( R' \).
\end{lemma}
\begin{proof}
	Replacing
	\( X \) by a suitable cover we may assume that
	\( p(x) = p'(x)\, \eps(x) \),
	\( q(x) = q'(x)\, \eps(x) \), and
	\( r(x) = r'(x)\, \eps(x) \) for some
	\( \eps \colon X \to \mathcal E \). Moreover, without loss of generality
	\( \bigl[ p'(x), q'(x), r'(x) \bigr] = 0 \). Passing to the slice category we may further assume that
	\( X = \{ * \} \) is the terminal object. Let
	\( (R', \Delta) \subseteq (R, \Delta) \) be the non-unital
	\( \mathsf B_3 \)-subring generated by
	\( \mathcal E \),
	\( p' \mathcal E \),
	\( q' \mathcal E \),
	\( r' \mathcal E \), and
	\( \Delta \). It is easy to see that
	\( \mathcal E \) is a weak unit for
	\( (R', \Delta) \) and
	\( R' \) is associative by lemma \ref{inv-artin}.
\end{proof}

\subsection{Universal central extensions}

All results in this subsection are well known for the category of sets, see \cite[\S 6.9]{weibel} and \cite[\S 2]{brown-loday}. We still give the proofs so he reader can check that all arguments hold in any infinitary pretopos.

Let
\( H \) and
\( G \) be group objects in an infinitary pretopos
\( \mathbf C \). A homomorphism
\( f \colon H \to G \) is called a \textit{central extension} if it is an epimorphism (in the category
\( \mathsf C \)) with central kernel. If
\( H \) is perfect (so
\( G \) is also perfect), then
\( f \) is called a \textit{perfect central extension}. Any central extension
\( f \) is a crossed module in a unique way and there exists a unique morphism
\( \langle {-}, {=} \rangle \colon G \times G \to H \) such that
\( f\bigl( \langle x, y \rangle \bigr) = [x, y] \) and
\( [x, y] = \bigl\langle f(x), f(y) \bigr\rangle \). This morphism
\( \langle {-}, {=} \rangle \) satisfies the obvious identities
\[
	\langle x, x \rangle = 1,
	\quad
	\langle x, 1 \rangle
	=
	\langle 1, x \rangle
	=
	1,
	\quad
	\langle x, y \rangle
	=
	\langle y, x \rangle^{-1},
\]
the \textit{crossed pairing identities}
\[
	\tag{XP}
	\label{XP}
	\langle x y, z \rangle
	=
	\up x {\langle y, z \rangle}\,
	\langle x, z \rangle,
	\quad
	\langle x, y z \rangle
	=
	\langle x, y \rangle\,
	\up y {\langle x, z \rangle},
\]
and the \textit{Hall--Witt identity}
\[
	\tag{HW}
	\label{HW}
	\up y {
		\bigl\langle x, [y^{-1}, z] \bigr\rangle
	}\,
	\up z {
		\bigl\langle y, [z^{-1}, x] \bigr\rangle
	}\,
	\up x {
		\bigl\langle z, [x^{-1}, y] \bigr\rangle
	}
	=
	1.
\]
Also,
\[
	\up g h
	=
	\bigl\langle g, f(h) \bigr\rangle\, h,
	\quad
	\up g {\langle x, y \rangle}
	=
	\langle \up g x, \up g y \rangle
\]
for
\( g \colon G \) and
\( x, y, h \colon H \).

\begin{lemma}
	\label{c-ext-perf}
	Let
	\( f \colon G \to H \) and
	\( g \colon H \to K \) be homomorphisms between perfect group objects. If two of
	\( f \),
	\( g \),
	\( g \circ f \) are central extensions, then the third one is also a central extension.

	If
	\( f, f' \colon G \to H\) and
	\( g \colon H \to K \) are perfect central extensions such that
	\( g \circ f = g \circ f' \), then
	\( f = f' \).
\end{lemma}
\begin{proof}
	Let
	\( f \colon G \to H \) and
	\( g \colon H \to K \) be perfect central extensions. If
	\( g(f(x)) = 1 \) for
	\( x \colon G \), then
	\( f(x) \) is central and
	\[
		\up x {\langle y, z \rangle}
		=
		\up{f(x)}{\langle y, z \rangle}
		=
		\bigl\langle
			\up {f(x)} y,
			\up {f(x)} z
		\bigr\rangle
		=
		1
	\]
	for
	\( y, z \colon H \). Since
	\( G \) is generated by the image of
	\( \langle {-}, {=} \rangle \colon H \times H \to G \),
	\( x \) is central.
	
	Now suppose that
	\( g \) and
	\( g \circ f \) are central extensions and take
	\( x, y \colon H \). Choose
	\( \widehat x, \widehat y \colon G \) such that
	\( g(x) = g(f(\widehat x)) \) and
	\( g(y) = g(f(\widehat y)) \), then
	\( f(\widehat x) x^{-1} \) and
	\( f(\widehat y) y^{-1} \) are central. Now
	\[
		[x, y]
		=
		\bigl[ f(\widehat x), f(\widehat y) \bigr]
		=
		f\bigl( [\widehat x, \widehat y] \bigr)
	\]
	i.e.
	\( f \) is an epimorphism.
	
	Next assume that
	\( f \) and
	\( g \circ f \) are central extensions. If
	\( g(x) = 1 \) for
	\( x \colon H \), then choose
	\( \widehat x \colon G \) such that
	\( x = f(\widehat x)\), so
	\( g(f(\widehat x)) = 1 \) and
	\( \widehat x \) is central. Since
	\( f \) is an epimorphism,
	\( f(\widehat x) \) is also central.
	
	To prove the last claim take
	\( x, y \colon G \). We have
	\[
		f\bigl( [x, y] \bigr)
		=
		\bigl\langle g(f(x)), g(f(y)) \bigr\rangle_g
		=
		\bigl\langle g(f'(x)), g(f'(y)) \bigr\rangle_g
		=
		f'\bigl( [x, y] \bigr),
	\]
	so
	\( f = f' \).
\end{proof}

A central extension
\( \kappa \colon \widetilde G \to G \) is called \textit{universal} if for every central extension
\( \delta \colon X \to G \) there is a unique homomorphism
\( f \colon \widetilde G \to X \) such that
\( \kappa = \delta \circ f \) (if all group objects are perfect, then this
\( f \) is also a central extension by lemma \ref{c-ext-perf}). For given
\( G \) such extension is unique up to a unique isomorphism if only it exists. The kernel
\( \schur(G) \) of
\( \kappa \) is called the \textit{Schur multiplier} of
\( G \). It is well known that
\( \schur(G) \cong \mathrm H_2(G) \) is the second integral homology group for perfect groups in the category of sets.

\begin{lemma}
	\label{c-ext-univ}
	A group object
	\( G \) has a universal central extension if and only if it is perfect. In the perfect case the universal central extension is also perfect and it can be constructed by the generators
	\( \langle x, y \rangle \) for
	\( x, y \colon G \) with the relations
	\begin{align*}
		\langle x y, z \rangle
		&=
		\langle \up x y, \up x z \rangle\,
		\langle x, z \rangle,
		&
		\up{\langle x, y \rangle}{\langle z, w \rangle}
		&=
		\bigl\langle
			\up {[x, y]} z,
			\up {[x, y]} w
		\bigr\rangle,
		\\
		\langle x, y z \rangle
		&=
		\langle x, y \rangle\,
		\langle \up y x, \up y z \rangle,
		&
		\bigl[
			\langle x, y \rangle,
			\langle z, w \rangle
		\bigr]
		&=
		\bigl\langle [x, y], [z, w] \bigr\rangle.
	\end{align*}
	The action of
	\( G \) on
	\( \widetilde G \) is given by
	\(
		\up x {\langle y, z \rangle}
		=
		\langle \up x y, \up x z \rangle
	\) and
	\(
		\kappa\bigl( \langle x, y \rangle \bigr)
		=
		[x, y]
	\).
\end{lemma}
\begin{proof}
	If
	\( \kappa \colon \widetilde G \to G \) is a universal central extension, then there are two homomorphisms
	\(
		\widetilde G
		\to
		\widetilde G / [\widetilde G, \widetilde G]
		\times
		G
	\) over
	\( G \), namely,
	\( f(x) = (1, \kappa(x)) \) and
	\( g(x) = \bigl( [x], \kappa(x) \bigr) \). By the universal property
	\( f = g \), i.e.
	\( \widetilde G \) and
	\( G \) are perfect.
	
	Now suppose that
	\( G \) is perfect and construct the group object
	\( \widetilde G \) by generators and relations as in the statement. The action of
	\( G \) on
	\( \widetilde G \) is well-defined by e.g. \cite[lemma 4]{iso-st-k2}. It is easy to see that
	\( \kappa \colon \widetilde G \to G \) is a crossed module and
	\( \widetilde G \) is perfect. If
	\( \delta \colon H \to G \) is a central extension and
	\( \kappa = \delta \circ f \), then
	\[
		f\bigl( [x, y] \bigr)
		=
		\bigl[ f(x), f(y) \bigr]
		=
		\bigl\langle
			\kappa(x),
			\kappa(y)
		\bigr\rangle_\delta,
	\]
	i.e.
	\(
		f\bigl( \langle x, y \rangle \bigr)
		=
		\langle x, y \rangle_\delta
	\) is unique and clearly well defined.
\end{proof}

Actually, the right two relations from lemma \ref{c-ext-univ} follow from the left ones, see \cite[proposition 2.3]{brown-loday}. The universal central extension of perfect group objects is functorial in a unique way. By lemma \ref{c-ext-perf} every universal central extension
\( \widetilde G \)
is \textit{centrally closed}, i.e. all central extensions of
\( \widetilde G \) split.

\begin{lemma}
	\label{c-ext-xmod}
	If
	\( X \to G \) is perfect crossed module over a group object
	\( G \), then the universal central extension
	\( \widetilde X \) of
	\( X \) is also a crossed module over
	\( G \) in a unique way. If
	\( X = \colim_{i \in \mathbf I} X_i \) is a colimit of perfect crossed modules by a finite diagram
	\( \mathbf I \), then
	\( X \) itself is perfect and
	\(
		\widetilde X
		\cong
		\colim_{i \in \mathbf I}
			\widetilde{X_i}
	\).
\end{lemma}
\begin{proof}
	The first part is easy, the action of
	\( G \) on
	\( \widetilde X \) is given by
	\(
		\up g {\langle x, y \rangle}
		=
		\langle \up g x, \up g y \rangle
	\) and the uniqueness follows from e.g. \cite[lemma 9(3)]{iso-st-k2}.

	To prove the second claim note that the colimit of perfect crossed modules is perfect. Since
	\( \colim_{i \in \mathbf I} \widetilde X_i \) is a crossed module over
	\( G \), it is a perfect central extension of
	\( X = \colim_{i \in \mathbf I} X_i \). Now take the universal central extension
	\( \widetilde X \to X \). For every
	\( i \) there is the evident homomorphism
	\[
		\widetilde X_i
		\to
		\widetilde X \times_X X_i
		\to
		\widetilde X
	\]
	of crossed modules over
	\( G \) lifting the homomorphism
	\( X_i \to X \), it is unique with these properties by lemma \ref{c-ext-perf}. By uniqueness these homomorphisms induce a homomorphism
	\(
		\colim_{i \in \mathbf I} \widetilde X_i
		\to
		\widetilde X
	\) of crossed modules over
	\( G \) making a commutative triangle together with
	\( X \). Lemma \ref{c-ext-perf} implies that it is a central extension, i.e. an isomorphism.
\end{proof}

\subsection{%
	Some varieties of non-unital
	\( \Phi \)-rings%
}

Let us briefly discuss certain varieties naturally appearing in the descriptions of
\( \schur(\stlin(\Phi, A)) \) for weakly unital
\( \Phi \)-rings
\( A \) below. Some of them already appear in \cite{schur-mult} in the unital case. We denote by
\( \langle X \rangle \) the ideal of a
\( \Phi \)-ring generated by
\( X \).

Recall that a non-unital associative ring
\( R \) (or a ring object in an infinitary pretopos) is called \textit{Boolean} if it satisfies the axioms
\[
	2 p = 0,
	\quad
	p q = q p,
	\quad
	p^2 = p.
\]
Finitely generated Boolean rings are just finite products of copies of the field
\( \mathbb F_2 \), and arbitrary Boolean rings in
\( \Set \) are subdirect products of copies of
\( \mathbb F_2 \). Every Boolean ring is weakly unital (even in an infinitary pretopos) because
\( p = p (p + q + p q) \) and
\( q = q (p + q + p q) \) for all
\( p, q \colon R \). For any associative ring
\( R \) let
\[
	R_2
	=
	R / \langle 2 p, p q - q p, p^2 - p \rangle
\]
be the universal Boolean ring generated by
\( R \).

Next consider the variety of associative rings with the additional axioms
\[
	2 p = 0,
	\quad
	p q = q p,
	\quad
	(p^2 - p) (q^2 - q) = 0.
\]
If
\( R \) is an algebra from this variety (even in an infinitary pretopos), then
\( R = I_\eps \rtimes R_2 \), where
\( I_\eps = \langle p^2 - p \rangle \leqt R \) is an ideal. Namely, the projection
\( R \to I_\eps \) is given by the polynomial
\( p^2 - p \), and the inclusion
\( R_2 \to R \) is given by the polynomial
\( p^2 \). Conversely, every semi-direct product of a Boolean ring and its (non-unital) module lies in the variety. From this description we see that indecomposable finitely generated rings from the variety are just
\( V \rtimes \mathbb F_2 \) for a finite vector space
\( V_{\mathbb F_2} \) and the cyclic additive group
\( \mathrm C_2 \) with zero multiplication, so all rings are subdirect products of copies of
\( \mathbb F_2 \),
\( \mathrm C_2 \), and
\( \mathbb F_2[\eps] / (\eps^2) \). A ring
\( R \) from the variety (even in an infinitary pretopos) is weakly unital if and only if
\( I_\eps R_2 = I_\eps \), in this case
\( \mathcal E = R_2 \) is always a valid weak unit by the same argument as for Boolean rings. For any associative ring
\( R \) let
\[
	R_{2 \eps}
	=
	R
	/
	\bigl\langle
		2 p,
		p q - q p,
		(p^2 - p) (q^2 - q)
	\bigr\rangle,
\]
and if
\( R \) is weakly unital, then
\( R_{2 \eps} \) is also weakly unital. If
\( R \) is an involution ring (as the first component of a
\( \mathsf B_\ell \)- or an
\( \mathsf F_4 \)-ring), then we redefine
\( R_{2 \eps} \) as
\[
	R
	/
	\bigl\langle
		2 p,
		p q - q p,
		(p^2 - p) (q^2 - q),
		p - p^*
	\bigr\rangle.
\]

We also need several varieties of
\( \mathsf B_3 \)-rings. The first one is given by the axioms
\[
	(p q) r = p (q r),
	\quad
	p q = q p,
	\quad
	\lambda^{\pm 1} p = p,
	\quad
	p^* = p,
	\quad
	3 p = 0,
	\quad
	p^3 = p,
	\quad
	u = \iota \cdot \langle \iota, u \rangle.
\]
These axioms have easy corollaries
\[
	\phi(p) = \dot 0,
	\quad
	u \cdot 3 = \dot 0,
	\quad
	u \cdot \lambda^{\pm 1} = u,
	\quad
	\rho(u) = \langle \iota, u \rangle^2,
	\quad
	\langle u, v \rangle
	=
	\langle \iota, u \rangle \langle \iota, v \rangle.
\]
Clearly, the operations
\(
	\Delta \to R,\, u \mapsto \langle \iota, u \rangle
\) and
\(
	R \to \Delta,\, p \mapsto \iota \cdot p
\) are inverses of each other, so we may identify
\( \Delta \) with
\( R \). Under this identification finitely generated algebras from the variety are finite powers of
\( \mathbb F_3 \) and all algebras are subdirect products of copies of
\( \mathbb F_3 \). Every
\( \mathsf B_3 \)-ring from this variety (even in an infinitary pretopos) is weakly unital with
\(
	\mathcal E
	=
	\Image\bigl( [\![ p^2 ]\!] \colon R \to R \bigr)
	=
	[\![ p^2 = p ]\!]
\) because
\( p = p (p^2 + q^2 - p^2 q^2) \) and
\( q = q (p^2 + q^2 - p^2 q^2) \) for
\( p, q \colon R \). For any
\( \mathsf B_3 \)-ring
\( (R, \Delta) \) let
\( (R_3, \Delta_3) \) be its universal factor-algebra from this variety.

Next there is the variety with the axioms
\[
	(p q) r = p (q r),
	\quad
	p q = q p,
	\quad
	2 p = 0,
	\quad
	\lambda^{\pm 1} p = p,
	\quad
	(p p^*)^2 = p p^*,
	\quad
	u \cdot 2 = \dot 0,
	\quad
	u \cdot \lambda^{\pm 1} = u,
	\quad
	\langle u, u \rangle = 0.
\]
An algebra from this variety (even in an infinitary pretopos) is weakly unital if and only if
\( R = R\, \Image\bigl( [\![ p p^* ]\!] \bigr) \) and
\( \Delta = \phi(R) \dotplus \Delta \cdot R \), under these conditions the object
\(
	\mathcal E
	=
	\Image\bigl( [\![ p p^* ]\!] \bigr)
	=
	[\![ p = p^2 = p^* ]\!]
\) is always a weak unit and a Boolean subring. For arbitrary
\( \mathsf B_3 \)-ring
\( (R, \Delta) \) let
\( (R_{2 *}, \Delta_{2 *}) \) be the universal factor-algebra from this variety, so if
\( (R, \Delta) \) is weakly unital, then
\( (R_{2 *}, \Delta_{2 *}) \) is also weakly unital. The remaining varieties of
\( \mathsf B_3 \)-rings are its subvarieties.

The third variety of
\( \mathsf B_3 \)-rings is determined by the additional axioms (to the previous variety)
\[
	p^* = p^2,
	\quad
	u = \iota \cdot \rho(u).
\]
In this variety
\[
	\langle u, v \rangle = 0,
	\quad
	p^4 = p,
\]
and there is an isomorphism between
\( \Delta \) and
\(
	[\![ p = p^* ]\!]
	=
	\Image\bigl( [\![ p p^* ]\!] \bigr)
	\subseteq
	R
\) given by
\( [\![ \iota \cdot p ]\!] \) and
\( \rho \). Finitely generated algebras from this variety are products of
\( (\mathbb F_4, \mathbb F_2) \) and
\( (\mathbb F_2, \mathbb F_2) \), these two algebras are the only subdirectly irreducible ones. It is easy to see that every algebra from this variety (even in an infinitary pretopos) is weakly unital with
\( \mathcal E = [\![ p = p^* ]\!] \), such
\( \mathcal E \) is a Boolean ring. For arbitrary
\( \mathsf B_3 \)-ring
\( (R, \Delta) \) let
\( (R_4, \Delta_4) \) be the universal factor-algebra from this variety. In applications we also need the ideal
\( I_4 \leq R_4 \) generated by
\( [\![ p - p^* ]\!] \), i.e.
\( R_4 / I_4 \) is the universal Boolean factor-ring. Note that
\( I_4 = \Image\bigl( [\![ p (q - q^*) ]\!] \bigr)\) because
\[
	p (q + q^*) + r (s + s^*)
	=
	\bigl( p (q + q^*) + r (s + s^*) \bigr)
	\bigl( q + q^* + s + s^* + (q + q^*) (s + s^*) \bigr).
\]

The fourth variety of
\( \mathsf B_3 \)-rings has the additional axioms
\[
	p = p^*,
	\quad
	p^2 = p,
	\quad
	\langle \iota, u \rangle \rho(u) = 0,
	\quad
	u \cdot \langle \iota, u \rangle
	\dotplus
	\iota \cdot \rho(u)
	=
	u.
\]
It follows that
\[
	\phi(p) = \dot 0,
	\quad
	u \cdot \langle \iota, v \rangle
	\dotplus
	v \cdot \langle \iota, u \rangle
	=
	\iota \cdot \langle u, v \rangle.
\]
Finitely generated algebras are products of
\( (\mathbb F_2, \mathbb F_2) \) and
\( (\mathbb F_2, \mathbb F_2 \times \mathbb F_2) \) with
\( \rho(x, y) = x y \), and these two algebras are the only subdirectly irreducible ones. Every algebra from this variety in an infinitary pretopos is weakly unital with
\( \mathcal E = R \). For arbitrary
\( \mathsf B_3 \)-ring
\( (R, \Delta) \) let
\( (R_{2 \mathrm b}, \Delta_{2 \mathrm b}) \) be the universal factor-algebra from this variety.

The last variety of
\( \mathsf B_3 \)-rings is given by the axioms
\begin{align*}
	p &= p^*,
	&
	(p^2 - p) (q^2 - q) &= 0,
	&
	u \cdot \langle \iota, v \rangle
	\dotplus
	v \cdot \langle \iota, u \rangle
	&=
	\iota \cdot \langle u, v \rangle,
	\\
	\phi(p) &= \dot 0,
	&
	\langle \iota, u \rangle \rho(u)
	&=
	\rho(u)^2 - \rho(u),
	&
	\langle u, v \rangle^2
	+
	\langle u, v \rangle^2 \langle \iota, u \rangle^2
	+
	\langle \iota, v \rangle^2 \rho(u)^2
	&=
	0.
\end{align*}
An algebra from this variety is weakly unital if and only if
\( R = R R \) and
\( \Delta = \Delta \cdot R \) (even in an infinitary pretopos), under these conditions we can take
\( \mathcal E = [\![ p = p^2 ]\!] \). For any
\( \mathsf B_3 \)-ring
\( (R, \Delta) \) let
\( (R_{2 \eps \delta}, \Delta_{2 \eps \delta}) \) be the universal factor-algebra from this variety. In applications we also need the ideal
\(
	I_{2 \eps \delta}
	=
	\langle p^2 - p \rangle
	\leqt
	R_{2 \eps \delta}
\).

Finally, we need the variety of
\( \mathsf F_4 \)-rings with the axioms
\begin{align*}
	(p q) r &= p (q r),
	&
	p q &= q p,
	&
	2 p &= 0,
	&
	\lambda p &= p,
	&
	p^* &= p^2,
	\\
	(u v) w &= u (v w),
	&
	u v &= v u,
	&
	2 u &= 0,
	&
	\lambda u &= u,
	&
	u^* &= u^2
\end{align*}
for
\( p, q, r \colon R \) and
\( u, v, w \colon S \). All algebras from this variety are weakly unital with
\( \mathcal E = \rho(S) \), this is a Boolean subring (even in infinitary pretopoi). Finitely generated algebras are products of
\( (\mathbb F_2, \mathbb F_2) \),
\( (\mathbb F_2, \mathbb F_4) \) (with
\( \phi(p) = 0 \),
\( \rho(p) = p \),
\( \phi(u) = u + u^* \),
\( \rho(u) = u u^* \)), and
\( (\mathbb F_4, \mathbb F_2) \) (with
\( \phi(p) = p + p^* \),
\( \rho(p) = p p^* \),
\( \phi(u) = 0 \),
\( \rho(u) = u \)), and these algebras are the only subdirecly irreducible ones. The universal factor-ring from this variety of an arbitrary
\( \mathsf F_4 \)-ring
\( (R, S) \) is denoted by
\( (R_{44}, S_{44}) \). In applications we also need the ideal
\(
	I_4
	=
	\Image\bigl( [\![ p (q - q^*) ]\!] \bigr)
	\leqt
	R_{44}
\).

\section{Upper bounds on Schur multipliers}

In this section we fix an infinitary pretopos
\( \mathbf C \), an indecomposable spherical root system
\( \Phi \) of rank
\( \geq 3 \) not of type
\( \mathsf H_\ell \), and a weakly unital
\( \Phi \)-ring object
\( A \) in
\( \mathbf C \). We are going to find generators of the Schur multiplier
\( \schur(\stlin(\Phi, A)) \) and at least some relations between them. The idea is to write down all instances of (\ref{XP}) and (\ref{HW}) applied to root subgroups (such that the convex hull of roots does not contain
\( 0 \)), and then derive all consequences.

\subsection{Simply laced cases}

Let us begin with the case
\( \Phi = \mathsf A_3 \). Take a weakly unital associative ring object
\( R \) in
\( \mathbf C \). Below we assume that all index variable
\( i, j, k, \ldots \) are distinct.

\begin{lemma}
	\label{a3-1}
	\(
		\bigl\langle x_{ij}(p), x_{ij}(q) \bigr\rangle
		=
		1
	\),
	\(
		\bigl\langle x_{ij}(p), x_{ik}(q) \bigr\rangle
		=
		\bigl\langle x_{ik}(p), x_{jk}(q) \bigr\rangle
		=
		1
	\).
\end{lemma}
\begin{proof}
	The identity (\ref{HW}) applied to
	\( x_{ij}(p) \),
	\( x_{ik}(-q) \),
	\( x_{kj}(r) \) implies
	\(
		\bigl\langle
			x_{ij}(p),
			x_{ij}(q r)
		\bigr\rangle
		=
		1
	\), so
	\(
		\bigl\langle x_{ij}(p), x_{ij}(q) \bigr\rangle
		=
		1
	\) by weak unitality. From the same identity for
	\( x_{ij}(p) \),
	\( x_{il}(-q) \),
	\( x_{lk}(r) \) and weak unitality we get
	\(
		\bigl\langle x_{ij}(p), x_{ik}(q) \bigr\rangle
		=
		1
	\), and similarly (or just applying the outer automorphism of
	\( \Phi \)) for
	\(
		\bigl\langle x_{ik}(p), x_{jk}(q) \bigr\rangle
		=
		1
	\).
\end{proof}

The next lemma (together with weak unitality) actually says that every
\( c_{ij \mid kl} \) factors through the initial commutative
\( \mathbb F_2 \)-algebra
\( R / \langle 2 p, p q - q p \rangle \) with a homomorphism from
\( R \).

\begin{lemma}
	\label{a3-2}
	For all permutations
	\( i j k l \in \mathrm S_4 \) there exists a unique homomorphism
	\(
		c_{ij \mid kl}
		\colon
		R
		\to
		\schur(\stlin(\mathsf A_3, R))
	\) such that
	\begin{align*}
		\bigl\langle x_{ij}(p), x_{kl}(q) \bigr\rangle
		&=
		c_{ij \mid kl}(p q),
		&
		c_{ij \mid kl}(p a b q)
		&=
		c_{ij \mid kl}(p b a q),
		\\
		c_{ij \mid kl}(2 p) &= 1,
		&
		c_{ij \mid kl}(p)
		=
		c_{il \mid kj}(p)
		&=
		c_{kl \mid ij}(p)
		=
		c_{kj \mid il}(p).
	\end{align*}
\end{lemma}
\begin{proof}
	Let
	\(
		c_{ij \mid kl}(p, q)
		=
		\bigl\langle x_{ij}(p), x_{kl}(q) \bigr\rangle
	\), this is a biadditive morphism
	\(
		R \times R
		\to
		\schur(\stlin(\mathsf A_3, R))
	\) by (\ref{XP}) satisfying
	\( c_{ij \mid kl}(p, q) = c_{kl \mid ij}(-q, p) \). The identity (\ref{HW}) applied to
	\( x_{ki}(p) \),
	\( x_{il}(b) \),
	\( x_{ij}(a) \) means that
	\(
		c_{ij \mid kl}(a, p b)
		=
		c_{il \mid kj}(b, p a)
	\), so
	\[
		\tag{
			\( \diamondsuit \)
		}
		\label{a3-comm}
		c_{ij \mid kl}(a, p b c)
		=
		c_{il \mid kj}(c, p b a)
		=
		c_{ij \mid kl}(b a, p c).
	\]
	Taking
	\( a, p \colon \mathcal E \) and using weak unitality, we obtain a unique homomorphism
	\(
		c_{ij \mid kl}
		\colon
		R
		\to
		\schur(\stlin(\mathsf A_3, R))
	\) such that
	\(c_{ij \mid kl}(p, q) = c_{ij \mid kl}(p q) \). Namely, first take
	\( b \colon \mathcal E \) in (\ref{a3-comm}) and construct the morphism such that the defining identity holds for
	\( p \colon \mathcal E \), and then use (\ref{a3-comm}) with arbitrary
	\( b \) to conclude that the defining identity always holds. To sum up, there are homomorphisms
	\(
		c_{ij \mid kl}
		\colon
		R
		\to
		\schur(\stlin(\mathsf A_3, R))
	\) such that
	\( c_{ij \mid kl}(p q) = c_{kl \mid ij}(-q p) \),
	\(
		c_{ij \mid kl}(a p b)
		=
		c_{il \mid kj}(b p a)
	\), and
	\(
		c_{ij \mid kl}(a b q)
		=
		c_{ij \mid kl}(b a q)
	\) by (\ref{a3-comm}) and weak unitality. It follows that
	\[
		c_{ij \mid kl}(p a b q)
		=
		c_{ij \mid kl}(b p a q)
		=
		c_{ij \mid kl}(p b a q).
	\]

	It remains to prove that
	\( c_{ij \mid kl}(2 p) = 1 \). But we may also apply (\ref{HW}) to
	\( x_{jl}(p) \),
	\( x_{kj}(b) \),
	\( x_{ij}(a) \) and get
	\(
		c_{ij \mid kl}(a, b p)
		=
		c_{kj \mid il}(b, a p)
	\), so
	\[
		c_{ij \mid kl}(p)
		=
		c_{kl \mid ij}(-p)
		=
		c_{kj \mid il}(-p)
		=
		c_{ij \mid kl}(-p). \qedhere
	\]
\end{proof}

Now let
\[
	y_{ijk}(p, q)
	=
	\bigl\langle x_{ij}(p), x_{jk}(q) \bigr\rangle,
\]
so
\( \kappa(y_{ijk}(p, q)) = x_{ik}(p q) \).

\begin{lemma}
	\label{a3-3}
	The morphisms
	\( y_{ijk} \) are biadditive and satisfy the identity
	\[
		y_{ikl}(p q, r)
		=
		y_{ijl}(p, q r)\,
		c_{ik \mid jl}(p q^2 r).
	\]
\end{lemma}
\begin{proof}
	The first claim follows from (\ref{XP}), the second one --- from (\ref{HW}) applied to
	\( x_{ij}(p) \),
	\( x_{jk}(-q) \),
	\( x_{kl}(r) \) and lemma \ref{a3-2}.
\end{proof}

\begin{prop}
	\label{a3-schur}
	Let
	\( R \) be a weakly unital associative ring object in an infinitary pretopos
	\( \mathbf C \). Then the homomorphisms
	\( c_{ij \mid kl} = c \) coincide and factor through
	\( R_{2 \eps} \). The Schur multiplier
	\( \schur(\stlin(\mathsf A_3, R)) \) is the image of
	\( c \).
\end{prop}
\begin{proof}
	We use the identities from lemma \ref{a3-2} without explicit references. By lemma \ref{a3-3} we have
	\begin{align*}
		y_{ikl}(p, q r s t)\,
		c_{ij \mid kl}(p q^2 r^2 s^2 t)
		&=
		y_{ijl}(p q r s, t)
		=
		y_{ikl}(p q r, s t)\,
		c_{ij \mid kl}(p q r s^2 t) \\
		&=
		y_{ijl}(p q, r s t)\,
		c_{ik \mid jl}(p q r^2 s t)\,
		c_{ij \mid kl}(p q r s^2 t) \\
		&=
		y_{ikl}(p, q r s t)\,
		c_{ij \mid kl}(p q^2 r s t)\,
		c_{ik \mid jl}(p q r^2 s t)\,
		c_{ij \mid kl}(p q r s^2 t)
	\end{align*}
	It follows (using weak unitality) that
	\[
		\tag{
			\( \spadesuit \)
		}
		\label{a3-tw-sym}
		c_{ik \mid jl}(p q r^2 s)\,
		=
		c_{ij \mid kl}(p q r s (q r s + q + s)).
	\]
	Since the left hand side depends only on
	\( q s \), we have
	\[
		c_{ij \mid kl}(p q r u s (q r u s + q u + s))
		=
		c_{ij \mid kl}(p q r u s (q r u s + q + u s)),
	\]
	and this together with weak unitality implies
	\[
		c_{ij \mid kl}\bigl(
			p (q^2 s - q s^2) (u^2 - u)
		\bigr)
		=
		1.
	\]
	Taking
	\( q = s v \) and
	\( s \colon \mathcal E \), we obtain
	\[
		\tag{
			\( \clubsuit \)
		}
		\label{a3-dbl-bool}
		c_{ij \mid kl}\bigl(
			p (u^2 - u) (v^2 - v)
		\bigr)
		=
		1
	\]
	by weak unitality. Now use (\ref{a3-dbl-bool}) to remove double squares in (\ref{a3-tw-sym}). After simplification using weak unitality, we get
	\[
		c_{ik \mid jl}(p) = c_{ij \mid kl}(p),
	\]
	so
	\( c = c_{ij \mid kl} \) is independent on the indices. It remains to prove that the Schur multiplier is the image of
	\( c \). Indeed, by lemma \ref{a3-3} we have
	\[
		y_{ikl}(p q r, s)
		=
		y_{ijl}(p q, r s)\, c(p q r^2 s)
		=
		y_{ikl}(p, q r s)\, c(p q r s (q + r)),
	\]
	so
	\[
		y_{ikl}(p q, r)
		=
		y_{ikl}(p, q r)\, c(p (q^2 - q) r)
	\]
	by weak unitality and the identities for
	\( c \). Again using weak unitality (first for
	\( p, q \colon \mathcal E \), then only for
	\( p \colon \mathcal E \)) we obtain a unique homomorphism
	\[
		y_{il}
		\colon
		R
		\to
		\widetilde \stlin(\mathsf A_3, R) / c(R)
	\]
	such that
	\( y_{ikl}(p, q) \bmod{c(R)} = y_{il}(p q) \). Again by weak unitality it follows that also
	\( y_{ijl}(p, q) \bmod{c(R)} = y_{il}(p q) \). Since the homomorphisms
	\( y_{ij} \) evidently satisfy all Steinberg relations, the epimorphism
	\(
		\widetilde \stlin(\mathsf A_3, R) / c(R)
		\to
		\stlin(\mathsf A_3, R)
	\) has a section. But a perfect central extension with a section is trivial by lemma \ref{c-ext-perf}.
\end{proof}

Now let
\[
	\Phi
	=
	\mathsf D_4
	=
	\{ \pm \e_i \pm \e_j \mid 1 \leq i < j \leq 4 \}.
\]
We embed this root system into
\[
	\textstyle \mathsf F_4
	=
	\mathsf D_4
	\sqcup
	\{ \pm \e_i \mid 1 \leq i \leq 4 \}
	\sqcup
	\bigl\{
		\frac 1 2
		(
			\eps_1 \e_1
			+
			\eps_2 \e_2
			+
			\eps_3 \e_3
			+
			\eps_4 \e_4
		)
		\mid
		\eps_i \in \{ -1, 1 \}
	\bigr\}.
\]
Recall that there are
\( 3 \) orbits of short roots of
\( \mathsf F_4 \) under the action of
\( \mathrm W(\mathsf D_4) \), namely,
\[
	\textstyle
	O_0 = \{ \pm \e_i \},
	\quad
	O_{+}
	=
	\bigl\{
		\frac 1 2 \sum_i \eps_i \e_i
		\mid
		\prod_i \eps_i = 1
	\bigr\},
	\quad
	O_{-}
	=
	\bigl\{
		\frac 1 2 \sum_i \eps_i \e_i
		\mid
		\prod_i \eps_i = -1
	\bigr\},
\]
and the outer automorphism group of
\( \mathsf D_4 \) may be identified with the group of permutations of these orbits. The whole automorphism group of
\( \mathsf D_4 \) is just
\( \mathrm W(\mathsf F_4) \).

Take a weakly unital commutative ring object
\( K \) in
\( \mathbf C \). For all orthogonal roots
\( \alpha, \beta \in \mathsf D_4 \) let
\[
	c_{\alpha \beta}(p, q)
	=
	\bigl\langle x_\alpha(p), x_\beta(q) \bigr\rangle.
\]
By proposition \ref{a3-schur} (applied to various root subsystems of type
\( \mathsf A_3 \)) this expression depends only on
\( p q \bmod{2 K} \), in particular, it is independent of the choice of the root homomorphisms in the group scheme
\( \mathbb{PSO}_8 \) over
\( \mathbb Z \). The independence of indices from the same proposition implies that
\[
	c_{\alpha \beta}(p, q)
	=
	c_{(\alpha + \beta) / 2}(p q)
\]
for unique homomorphisms
\[
	c_\gamma
	\colon
	K_{2 \eps} \to \schur(\stlin(\mathsf D_4, K))
\]
for
\( \gamma \in \mathsf F_4 \setminus \mathsf D_4 \) because any two decompositions of
\( \gamma \) lie in a common root subsystem of type
\( \mathsf A_3 \). Moreover, these
\( c_\gamma \) depend only on the orbit of
\( \gamma \) under the action of
\( \mathrm W(\mathsf D_4) \) because any two roots from the same orbit have decompositions in a common root subsystem of type
\( \mathsf A_3 \).

\begin{prop}
	\label{d4-schur}
	Let
	\( K \) be a weakly unital commutative ring object in an infinitary pretopos
	\( \mathbf C \). Then the Schur multiplier
	\( \schur(\stlin(\mathsf D_4, K)) \) is the product of the images of
	\[
		c_0, c_{+}, c_{-}
		\colon
		K_{2 \eps} \to \schur(\stlin(\mathsf D_4, K))
	\]
	constructed above. These homomorphisms also satisfy the relations
	\[
		c_0(p) = c_{+}(p) = c_{-}(p),
		\quad
		c_0(q)\, c_{+}(q)\, c_{-}(q) = 1
	\]
	for
	\( p \colon I_\eps \) and
	\( q \colon K_2 \), where
	\( K_{2 \eps} = I_\eps \rtimes K_2 \) is the canonical decomposition.
\end{prop}
\begin{proof}
	Let
	\[
		y_{\alpha \beta}(p, q)
		=
		\bigl\langle
			x_\alpha(p),
			x_\beta(q)
		\bigr\rangle
	\]
	for all roots
	\( \alpha, \beta \in \mathsf D_4 \) at the angle
	\( 2 \pi / 3 \). By lemma \ref{a3-3} we have
	\[
		y_{\alpha + \beta, \gamma}(p q, r)
		=
		y_{\alpha, \beta + \gamma}(p, q r)
			^{\eps_{\alpha \beta \gamma}}\,
		c_{\alpha + \beta, \beta + \gamma}(p q^2 r)
	\]
	for all bases
	\( (\alpha, \beta, \gamma) \) of root subsystems of type
	\( \mathsf A_3 \) (with
	\( \alpha \perp \gamma \)), where
	\( \eps_{\alpha \beta \gamma} \in \{ -1, 1 \} \) depends only on the structure constants. It follows that
	\begin{align*}
		y_{\e_1 + \e_2, -\e_2 - \e_3}(q r s, p)^{-\eps}\,
		c_0(p q^2 r^2 s)
		&=
		y_{-\e_2 - \e_3, \e_1 + \e_2}(p, q r s)^\eps\,
		c_0(p q^2 r^2 s)
		=
		y_{\e_1 - \e_2, \e_2 - \e_3}(p q r, s) \\
		&=
		y_{\e_1 - \e_4, \e_4 - \e_3}(p q, r s)^\eta\,
		c_{+}(p q r^2 s) \\
		&=
		y_{\e_1 + \e_2, -\e_2 - \e_3}(p, q r s)^\tau\,
		c_{+}(p q^2 r s)\,
		c_{-}(p q r^2 s),
	\end{align*}
	where
	\( \eps, \eta, \tau \in \{ -1, 1 \} \) depend only on the choice of structure constants. Taking
	\( \mathbf C = \Set \),
	\( K = \mathbb Z \), and evaluating the images of both sides in the Chevalley group we see that
	\( \eps = -\tau \), i.e.
	\[
		\tag{
			\( \heartsuit \)
		}
		\label{d4-schur-sym}
		y_{\e_1 + \e_2, -\e_2 - \e_3}(q r s, p)
		=
		y_{\e_1 + \e_2, -\e_2 - \e_3}(p, q r s)\,
		c_0(p q^2 r^2 s)\,
		c_{+}(p q r^2 s)\,
		c_{-}(p q^2 r s).
	\]
	Note that most terms depend only on
	\( q r \), so
	\[
		c_{+}(p (q u) r^2 s)\, c_{-}(p (q u)^2 r s)
		=
		c_{+}(p q (u r)^2 s)\, c_{-}(p q^2 (u r) s).
	\]
	Substituting
	\( q = r \) and applying weak unitality and the triality symmetry we get the first relation
	\[
		c_0(p (q^2 - q))
		=
		c_{+}(p (q^2 - q))
		=
		c_{-}(p (q^2 - q)).
	\]
	
	Now the relation (\ref{d4-schur-sym}) can be written as
	\[
		y_{\e_1 + \e_2, -\e_2 - \e_3}(q, p)
		=
		y_{\e_1 + \e_2, -\e_2 - \e_3}(p, q)\,
		c_0(p q)\,
		c_{+}(p q)\,
		c_{-}(p q),
	\]
	and taking
	\( p = q \) we obtain the second relation
	\[
		c_0(p^2)\, c_{+}(p^2)\, c_{-}(p^2) = 1.
	\]
	
	It remains to show that these morphisms generate the Schur multiplier. Let
	\(
		C
		=
		\bigl\langle
			c_0(K),
			c_{+}(K),
			c_{-}(K)
		\bigr\rangle
	\). We have well-defined homomorphisms
	\(
		y_\alpha
		\colon
		K
		\to
		\widetilde \stlin(\mathsf D_4, K) / C
	\) such that
	\(
		y_{\beta, \alpha - \beta}(p, q)
		=
		y_\alpha(p q)
	\) by proposition \ref{a3-schur} applied to various root subsystems of type
	\( \mathsf A_3 \) because any two decompositions of
	\( \alpha \) into a sum of two roots lie in a common root subsystem of type
	\( \mathsf A_3 \). Since they satisfy all Steinberg relations, we conclude as in proposition \ref{a3-schur} using lemma \ref{c-ext-perf}.
\end{proof}

Finally, we deal with the remaining simply laced root systems, i.e.
\( \mathsf A_\ell \),
\( \mathsf D_\ell \), and
\( \mathsf E_\ell \).
\begin{prop}
	\label{ade-schur}
	Let
	\( \Phi \) be an irreducible spherical simply laced root system of rank
	\( \geq 4 \) excluding
	\( \mathsf D_4 \) and
	\( R \) be a weakly unital
	\(\Phi\)-ring in
	an infinitary pretopos
	\( \mathbf C \). Then the Steinberg group
	\( \stlin(\Phi, R) \) is centrally closed.
\end{prop}
\begin{proof}
	The Schur multiplier is generated by
	\(
		c_{\alpha \beta}(p, q)
		=
		\bigl\langle
			x_\alpha(p),
			x_\beta(q)
		\bigr\rangle
	\) for all
	\( \alpha \perp \beta \) by the same argument as in the proof of proposition \ref{d4-schur}. In order to show that
	\( c_{\alpha \beta}(p, q) = 1 \) it suffices to consider the case of basic roots
	\( \alpha \) and
	\( \beta \), i.e. corresponding to nodes of the Dynkin diagram not connected by an edge. If
	\( \beta \) has a neighbor
	\( \gamma \) in the Dynkin diagram not connected with
	\( \alpha \), then by (\ref{HW}) applied to
	\( x_\alpha(p) \),
	\( x_{\beta + \gamma}(-q) \),
	\( x_{-\gamma}(r) \) we get
	\[
		c_{\alpha \beta}(
			p,
			N_{\beta + \gamma, -\gamma} q r
		)
		=
		1
	\]
	so
	\( c_{\alpha \beta}(p, q) = 1 \) by weak unitality. The same argument works if
	\( \alpha \) has a neighbor not connected with
	\( \beta \). The only remaining case is where
	\( \alpha \) and
	\( \beta \) are leafs of the Dynkin diagram of type
	\( \mathsf D_\ell \) connected to the same node
	\( \delta \). Let
	\( \gamma \) be the third node connected to
	\( \delta \). We have
	\[
		c_{\alpha \beta}(p)
		=
		c_{\alpha \beta}(p^2 - p)\,
		c_{\alpha \beta}(p^2)
		=
		c_{\alpha \gamma}(p^2 - p)\,
		c_{\alpha \gamma}(p^2)\,
		c_{\beta \gamma}(p^2)
		=
		1
	\]
	by proposition \ref{d4-schur} and the already proved cases.
\end{proof}

\subsection{%
	The case of
	\( \mathsf B_3 \)%
}

Let
\( \Phi = \mathsf B_3 \) and
\( (R, \Delta) \) be a weakly unital
\( \mathsf B_3 \)-ring in
\( \mathbf C \). Below we assume that all indices have distinct absolute values from
\( 1 \) to
\( \ell \) unless stated otherwise.

The next lemma shows that if the root subgroups with the roots
\( \alpha \) and
\( \beta \) commute in the Steinberg group, then their preimages commute in the universal central extension unless
\( \alpha \perp \beta \) have distinct lengths. In particular, the preimage of the root subgroup
\( G_\alpha \) in the universal central extension is abelian if
\( \alpha \) is long (i.e.
\( \alpha = \e_i - \e_j \)) and two-step nilpotent if
\( \alpha \) is short (i.e.
\( \alpha = \e_i \)). The next lemma will be used implicitly in calculations.

\begin{lemma}
	\label{b3-1}
	We have
	\begin{align*}
		\bigl\langle x_{ij}(p), x_j(u) \bigr\rangle
		&=
		1,
		&
		\bigl\langle x_{ij}(p), x_{ik}(q) \bigr\rangle
		&=
		1,
		\\
		\bigl\langle x_{ij}(p), x_{ij}(q) \bigr\rangle
		&=
		1,
		&
		\bigl\langle x_{ik}(p), x_{jk}(q) \bigr\rangle
		&=
		1,
		\\
		\bigl\langle x_i(u), x_i(\phi(p)) \bigr\rangle
		&=
		1,
		&
		\bigl\langle
			x_i(\phi(p)),
			x_{jk}(q)
		\bigr\rangle
		&=
		1,
		\\
		\bigl\langle x_i(u), x_j(\phi(p)) \bigr\rangle
		&=
		1.
		&&
	\end{align*}
\end{lemma}
\begin{proof}
	We combine weak unitality and (\ref{HW}) applied to the triples
	\begin{align*}
		&
		x_{-k, j}(p), x_{i, -k}(q), x_j(u);
		&
		&
		x_k(\dotminus \iota \cdot r \lambda_k^*),
		x_{k, -i}(\lambda_k q),
		x_{ij}(p);
		\\
		&
		x_{ij}(p), x_{kj}(q), x_{ik}(r);
		&
		&
		x_{-j}(\dotminus \iota \cdot r \lambda_{-j}^*),
		x_{-j, k}(\lambda_{-j} q),
		x_{ik}(p);
		\\
		&
		x_{-i, j}(\lambda_i^* p), x_i(u), x_{ji}(q);
		&
		&
		x_{-i, j}(-\lambda_i^* p),
		x_{ji}(q),
		x_{jk}(r);
		\\
		&
		x_i(u), x_{kj}(p), x_{-j, k}(\lambda_j^* q).
		&&
	\end{align*}
	For the identities from the right column use the previous ones.
\end{proof}

\begin{lemma}
	\label{b3-2}
	For all indices
	\( i \),
	\( j \),
	\( k \) there is a unique homomorphism
	\(
		c_{i \mid jk}
		\colon
		\Delta / \phi(R)
		\to
		\schur(\stlin(\mathsf B_3, R, \Delta))
	\) such that
	\[
		\bigl\langle x_i(u), x_{jk}(p) \bigr\rangle
		=
		c_{i \mid jk}(u \cdot p).
	\]
	Moreover,
	\begin{align*}
		c_{i \mid jk}(u \cdot p q r)
		&=
		c_{i \mid jk}(u \cdot q p r),
		\\
		c_{i \mid jk}(u \cdot \lambda^{\pm 1})
		&=
		c_{i \mid jk}(u),
		\\
		c_{i \mid jk}(u \cdot p^*)
		&=
		c_{i \mid jk}(u \cdot p),
		\\
		c_{i \mid jk}(u)
		=
		c_{k \mid ji}(\dotminus u)
		&=
		c_{i \mid -k, -j}(\dotminus u)
		=
		c_{-j \mid -i, k}(\dotminus u),
		\\
		c_{i \mid jk}\bigl(
			u \cdot \langle v, w \rangle
		\bigr)\,
		c_{i \mid jk}\bigl(
			v \cdot \langle w, u \rangle
		\bigr)\,
		c_{i \mid jk}\bigl(
			w \cdot \langle u, v \rangle
		\bigr)
		&=
		1,
		\\
		c_{i \mid jk}(u)^2
		&=
		c_{i \mid jk}\bigl(
			\iota \cdot \langle \iota, u \rangle
		\bigr)^2,
		\\
		c_{i \mid jk}(u)^6 &= 1.
	\end{align*}
\end{lemma}
\begin{proof}
	Let
	\[
		c_{i \mid jk}(u, p)
		=
		\bigl\langle x_i(u), x_{jk}(p) \bigr\rangle
		\colon
		\schur(\stlin(\mathsf B_3, R, \Delta)),
	\]
	it is biadditive by (\ref{XP}),
	\( c_{i \mid jk}(\phi(p), q) = 1 \) by lemma \ref{b3-1}, and
	\(
		c_{i \mid jk}(u, p)
		=
		c_{i \mid -k, -j}(
			u,
			\lambda_k^* p^* \lambda_j
		)^{-1}
	\). Applying (\ref{HW}) to
	\( x_{jk}(q) \),
	\( x_k(u) \),
	\( x_{ki}(-p) \) we get
	\[
		c_{i \mid jk}(u \cdot p, q)
		=
		c_{k \mid ji}(u, q p)^{-1}.
	\]
	Then
	\[
		c_{i \mid jk}(u, r (p q))
		=
		c_{k \mid ji}(u \cdot p q, r)^{-1}
		=
		c_{i \mid jk}(u \cdot p, r q)
	\]
	(recall that
	\( R \) is non-associative in general), so by weak unitality
	\(
		c_{i \mid jk}(u \cdot p, \eps)
		=
		c_{i \mid jk}(u, p \eps)
	\) for
	\( \eps \colon \mathcal E \). Again by weak unitality (applied for
	\( p \colon \mathcal E \) and then for general
	\( p \)) it follows that there exists a unique homomorphism
	\[
		c_{i \mid jk}
		\colon
		\Delta / \phi(R)
		\to
		\schur(\stlin(\mathsf B_3, R, \Delta))
	\]
	such that
	\(
		c_{i \mid jk}(u, p)
		=
		c_{i \mid jk}(u \cdot p)
	\), it clearly satisfies the commutativity identity. The homomorphisms
	\( c_{i \mid jk}(u) \) satisfy the symmetry relations
	\[
		c_{i \mid jk}(u \cdot p)
		=
		c_{i \mid -k, -j}(
			u \cdot \lambda_k^* p^* \lambda_j
		)^{-1},
		\quad
		c_{i \mid jk}(u)
		=
		c_{k \mid ji}(u)^{-1}
	\]
	by weak unitality, so we may change the indices under a suitable action of
	\( \mathrm S_3 \). The left one implies that
	\[
		c_{i \mid jk}(
			u \cdot \lambda_{-j}^* (p q)^* \lambda_{-k}
		)
		=
		c_{i \mid -k, -j}(\dotminus u \cdot p q)
		=
		c_{i \mid jk}(
			u \cdot p \lambda_{-j}^* q^* \lambda_{-k}
		),
	\]
	i.e.
	\(
		c_{i \mid jk}(u \cdot p)
		=
		c_{i \mid jk}(u \cdot p^*)
	\)
	and, in particular,
	\(
		c_{i \mid jk}(u \cdot \lambda^{\pm 2})
		=
		c_{i \mid jk}(u)
	\) (using weak unitality). Moreover,
	\begin{align*}
		c_{i \mid jk}(u)
		&=
		c_{i \mid -k, -j}(
			u \cdot \lambda_j \lambda_k
		)^{-1}
		=
		c_{-j \mid -k, i}(u \cdot \lambda_j \lambda_k)
		=
		c_{-j \mid -i, k}(
			u \cdot \lambda_i \lambda_j \lambda
		)^{-1} \\
		&=
		c_{k \mid -i, -j}(
			u \cdot \lambda_i \lambda_j \lambda
		)
		=
		c_{k \mid ji}(u \cdot \lambda^3)^{-1}
		=
		c_{i \mid jk}(u \cdot \lambda^3),
	\end{align*}
	so
	\(
		c_{i \mid jk}(u \cdot \lambda)
		=
		c_{i \mid jk}(u \cdot \lambda^3)
		=
		c_{i \mid jk}(u)
	\).

	Finally, let us apply (\ref{HW}) to
	\( x_i(u) \),
	\( x_{-j}(v) \),
	\( x_k(w) \). After simplification we get
	\[
		c_{i \mid jk}\bigl(
			u \cdot \langle v, w \rangle
		\bigr)\,
		c_{i \mid jk}\bigl(
			v \cdot \langle w, u \rangle
		\bigr)\,
		c_{i \mid jk}\bigl(
			w \cdot \langle u, v \rangle
		\bigr)
		=
		1.
	\]
	Substituting
	\( v = w = \iota \cdot \eps \) for
	\( \eps \colon \mathcal E \), we get
	\(
		c_{i \mid jk}(u)^2
		=
		c_{i \mid jk}\bigl(
			\iota \cdot \langle \iota, u \rangle
		\bigr)^2
	\), so taking
	\( u = \iota \cdot p \) we obtain
	\( c_{i \mid jk}(\iota \cdot p)^6 = 1 \). It follows that
	\(
		c_{i \mid jk}(u)^6
		=
		c_{i \mid jk}\bigl(
			\iota \cdot \langle \iota, u \rangle
		\bigr)^6
		=
		1
	\).
\end{proof}

Since
\( c_{i \mid jk}(u)^6 = 1 \), it is useful to decompose
\( c_{i \mid jk}(u) \) into its
\( 2 \)- and
\( 3 \)-torsion components, i.e.
\[
	c_{i \mid jk}(u)_2 = c_{i \mid jk}(u)^3,
	\quad
	c_{i \mid jk}(u)_3 = c_{i \mid jk}(u)^4,
	\quad
	c_{i \mid jk}(u)_2^2 = c_{i \mid jk}(u)_3^3 = 1,
	\quad
	c_{i \mid jk}(u)
	=
	c_{i \mid jk}(u)_2\, c_{i \mid jk}(u)_3.
\]

\begin{lemma}
	\label{b3-3}
	For any index
	\( i \) there is a unique homomorphism
	\(
		y_{-i, i}
		\colon
		R
		\to
		\widetilde \stlin(\mathsf B_3, R, \Delta)
	\) such that
	\begin{align*}
		\kappa(y_{-i, i}(p)) &= x_i(\phi(p)),
		&
		y_{-i, i}((p q) r) &= y_{-i, i}(p (q r)),
		\\
		\bigl\langle
			x_{-i, j}(p),
			x_{ji}(q)
		\bigr\rangle
		&=
		y_{-i, i}(\lambda_i p q),
		&
		y_{-i, i}(p + p^* \lambda) &= 1,
		\\
		\bigl\langle x_i(u), x_i(v) \bigr\rangle
		&=
		y_{-i, i}\bigl(-\langle u, v \rangle\bigr).
		&&
	\end{align*}
\end{lemma}
\begin{proof}
	Let
	\[
		y_{-i, j, i}(p, q)
		=
		\bigl\langle
			x_{-i, j}(p),
			x_{ji}(q)
		\bigr\rangle,
	\]
	so
	\(
		\kappa(y_{-i, j, i}(p, q))
		=
		x_i(\phi(\lambda_i p q))
	\). Applying (\ref{HW}) to
	\( x_{-i, k}(p) \),
	\( x_{kj}(-q) \),
	\( x_{ji}(r) \) we get
	\[
		y_{-i, j, i}(p q, r)
		=
		y_{-i, k, i}(p, q r),
	\]
	so
	\(
		y_{-i, j, i}(\eps \eps' p, q)
		=
		y_{-i, k, i}(\eps \eps', p q)
		=
		y_{-i, j, i}(\eps, \eps' p q)
	\) for
	\( \eps, \eps' \colon \mathcal E \). From this identity (applied to
	\( \eps' p = p' \colon \mathcal E \) and to the general
	\( p' \)) using weak unitality we get a unique homomorphism
	\(
		y_{-i, i}
		\colon
		R
		\to
		\widetilde \stlin(\mathsf B_3, R, \Delta)
	\) such that
	\(
		y_{-i, j, i}(p, q)
		=
		y_{-i, i}(\lambda_i p q)
	\). Again by weak unitality,
	\(
		y_{-i, k, i}(p, q)
		=
		y_{-i, i}(\lambda_i p q)
	\), so
	\( y_{-i, i} \) is independent on the choices of
	\( j \) and
	\( k \). Moreover, it is ``associative''. The skew-symmetry relation follows from weak unitality and
	\begin{align*}
		y_{-i, i}(p q)
		&=
		y_{-i, j, i}(\lambda_i^* p, q)
		=
		\bigl\langle
			x_{-i, j}(\lambda_i^* p),
			x_{ji}(q)
		\bigr\rangle
		=
		\bigl\langle
			x_{-j, i}(-\lambda_j^* p^* \lambda),
			x_{-i, -j}(-\lambda_i^* q^* \lambda_j)
		\bigr\rangle \\
		&=
		y_{-i, -j, i}(
			\lambda_i^* q^* \lambda_j,
			\lambda_j^* p^* \lambda
		)^{-1}
		=
		y_{-i, i}(-q^* p^* \lambda)
	\end{align*}

	Finally, applying (\ref{HW}) to
	\( x_i(u) \),
	\( x_j(v) \),
	\( x_{ji}(-p) \) we get
	\[
		\bigl\langle x_i(u), x_i(v \cdot p) \bigr\rangle
		=
		y_{-i, i}\bigl( -\langle u, v \rangle p \bigr)
	\]
	and conclude by weak unitality.
\end{proof}

Now we introduce the remaining pairings of root elements with non-opposite roots,
\begin{align*}
	y_{ijk}(p, q)
	&=
	\bigl\langle x_{ij}(p), x_{jk}(q) \bigr\rangle,
	&
	\kappa(y_{ijk}(p, q))
	&=
	x_{ik}(p q),
	\\
	y_{i0j}(u, v)
	&=
	\bigl\langle x_{-i}(u), x_j(v) \bigr\rangle,
	&
	\kappa(y_{i0j}(u, v))
	&=
	x_{ij}\bigl(
		-\lambda_{-i}^* \langle u, v \rangle
	\bigr),
	\\
	z_{ij}(u, p)
	&=
	\bigl\langle x_i(u), x_{ij}(p) \bigr\rangle,
	&
	\kappa(z_{ij}(u, p))
	&=
	x_{-i, j}(\lambda_i^* \rho(u) p)\,
	x_j(\dotminus u \cdot (-p)).
\end{align*}

\begin{lemma}
	\label{b3-4}
	The morphisms
	\( y_{ijk} \) and
	\( y_{i0j} \) are biadditive. Also,
	\begin{align*}
		y_{ijk}(p, q)
		&=
		y_{-k, -j, -i}(
			-\lambda_k^* q^* \lambda_j,
			\lambda_j^* p^* \lambda_i
		),
		&
		z_{ij}(u \dotplus v, p)
		&=
		z_{ij}(v, p)\,
		y_{-i, 0, j}(u, \dotminus v \cdot (-p))\,
		z_{ij}(u, p),
		\\
		y_{i0j}(u, v) &= y_{-j, 0, -i}(v, u)^{-1},
		&
		z_{ij}(u, p + q)
		&=
		z_{ij}(u, p)\,
		y_{-j, j}(-p^* \rho(u) q)\,
		z_{ij}(u, q),
		\\
		1 &= y_{i0j}(u, \phi(p)).
		&&
	\end{align*}
	Moreover,
	\begin{align*}
		y_{ikj}\bigl( \langle u, v \rangle, p \bigr)\,
		y_{i0j}(u \cdot \lambda_{-i}^*, v \cdot p)
		&=
		c_{k \mid ij}\bigl(
			v \cdot \langle u, v \rangle p
		\bigr)\,
		c_{k \mid ij}(u \cdot \rho(v) p)^{-1},
		\\
		y_{ijk}(p \lambda_j^* \rho(u), q)
		&=
		y_{i, -j, k}(p, \lambda_j^* \rho(u) q)\,
		c_{k \mid ij}(u \cdot \rho(u) p q)\,
		c_{k \mid i, -j}(u \cdot p q)^{-2},
		\\
		z_{-i, j}(
			\phi(\lambda_{-i} p q \lambda_i),
			\lambda_i^* r
		)
		&=
		y_{-j, j}\bigl(
			(r^* \lambda p) (q r)
		\bigr)^{-1}\,
		y_{ikj}(p, q r)\,
		y_{i, -k, j}(
			\lambda^* q^* \lambda_k,
			\lambda_k^* p^* r
		)^{-1},
		\\
		z_{jk}(u \cdot p, q)\,
		y_{-j, i, k}(\lambda_j^* p^* \rho(u), p q)^{-1}
		&=
		z_{ik}(u, p q)\,
		y_{-i, j, k}(\lambda_i^* \rho(u) p, q)^{-1}\,
		c_{k \mid ij}(u \cdot p^2 q)^{-1}\,
		c_{k \mid -i, j}(u \cdot \rho(u) p^2 q)^2.
	\end{align*}
\end{lemma}
\begin{proof}
	The first part easily follows from (\ref{XP}) and lemma \ref{b3-3}. The last four identities follow from (\ref{HW}) applied to the triples
	\begin{align*}
		&
		x_k(v),\,
		x_{kj}(-p)\,
		x_{-i}(u \cdot \lambda_{-i}^*);
		&
		&
		x_{ik}(p),\,
		x_{k, -i}(-q \lambda_i),\,
		x_{-ij}(\lambda_i^* r);
		\\
		&
		x_j(u),\,
		x_{i, -j}(p),\,
		x_{jk}(q);
		&
		&
		x_i(u),
		x_{ij}(-p),
		x_{jk}(q).
	\end{align*}
	During the calculations we may also use (\ref{XP}) to expand expressions such as
	\( \up{x_{jk}(q)}{z_{ij}(u, p)} \). For the last identity we first obtain
	\begin{align*}
		y_{-j, -i, k}&(
			\lambda_j^* p^* \lambda_i, 
			\lambda_i^* \rho(u) p q
		)\,
		z_{ik}(u, p q)\,
		c_{k \mid ij}(u \cdot p^2 q) \\
		&=
		y_{-i, 0, k}(u, u \cdot p q)\,
		z_{jk}(u \cdot p, q)\,
		y_{-i, j, k}(
			-\lambda_i^* \rho(u)^* \lambda p,
			q
		)\,
		c_{k \mid -i, j}(u \cdot \rho(u) p^2 q)^2
	\end{align*}
	and then apply previously obtained identities and lemma \ref{b3-2} for further simplification.
\end{proof}

\begin{lemma}
	\label{b3-5}
	Suppose that
	\( R \) is associative. Then for all indices
	\( i \),
	\( j \),
	\( j' \),
	\( k \) with
	\( |j| = |j'| \) there exists a unique biadditive morphism
	\[
		b_{ijj'k}
		\colon
		R \times R
		\to
		\schur(\stlin(\mathsf B_3, R, \Delta))
	\]
	such that
	\(
		y_{ij'k}(p q, r)
		=
		y_{ijk}(p, q r)\, b_{ijj'k}(q, p^* r)
	\), it satisfies the following identities.
	\begin{align*}
		b_{ijj'k}(p, q r s t) &= b_{ijj'k}(p, q s r t),
		\\
		b_{ijj'k}(p, q)^2 &= 1,
		\\
		b_{ijj''k}(p q, r)
		&=
		b_{ijj'k}(p, q r)\,
		b_{ij'j''k}(q, r p^*),
		\\
		b_{ijj'k}(p, q)
		&=
		b_{-k, -j', -j, -i}(
			\lambda_{j'}^* p^* \lambda_j,
			\lambda_i
			\lambda_j^*
			\lambda_{j'}^*
			\lambda_k
			q
		),
		\\
		b_{ijj'k}\bigl(
			p,
			\langle u, v \rangle
		\bigr)
		&=
		c_{k \mid ij}\bigl(
			u \cdot \langle u, v \rangle p
		\bigr)_2\,
		c_{k \mid ij}(v \cdot \rho(u) p)_2\,
		c_{k \mid ij'}\bigl(
			u \cdot \langle u, v \rangle p^2
		\bigr)_2\,
		c_{k \mid ij'}(v \cdot \rho(u) p^2)_2,
		\\
		b_{i, -j, j, k}(\lambda_j^* \rho(u), p)
		&=
		c_{k \mid ij}(u \cdot \rho(u) p)_2,
		\\
		b_{ijj'k}\bigl(
			\langle u, v \rangle,
			p
		\bigr)
		&=
		c_{k \mid i, -j}\bigl(
			u \cdot \langle u, v \rangle p
		\bigr)_2\,
		c_{k \mid i, -j}(v \cdot \rho(u) p)_2\,
		c_{k \mid ij'}\bigl(
			v \cdot \langle u, v \rangle p
		\bigr)_2\,
		c_{k \mid ij'}(u \cdot \rho(v) p)_2,
		\\
		b_{ijj'k}(p, q)
		&=
		b_{i, -j', -j, k}(
			\lambda_{j'}^* p^* \lambda_j,
			\lambda \lambda_j^* \lambda_{j'}^* q
		),
		\\
		b_{-i, j', j, k}&(q r, \lambda_i p s)\,
		b_{-i, -j', j, k}(
			\lambda_{-j'} p r,
			\lambda_i \lambda_{j'}^* q s
		)
		=
		b_{-j, i', i, k}(q, \lambda_j p r^2 s)\,
		b_{-j, -i', i, k}(
			\lambda_{-i'} p,
			\lambda_{i'}^* \lambda_j q r^2 s
		),
		\\
		b_{-i, j', j, k}&(
			q r,
			\lambda_i \rho(u)^* p^* s
		)\,
		b_{-i', j', j, k}(
			q r,
			\lambda_{i'} \rho(u)^* p p^* q s
		) \\
		&=
		b_{-j, i, i', k}(
			p q,
			\lambda_j \rho(u)^* p q r^2 s
		)\,
		b_{-j', i, i', k}(
			p q,
			\lambda_{j'} \rho(u)^* p r s
		) \\
		&\cdot
		c_{k \mid ij}(u \cdot p^2 q^2 r^2 s)_2\,
		c_{k \mid ij'}(u \cdot p^2 q r s)_2\,
		c_{k \mid i'j}(u \cdot p q r^2 s)_2\,
		c_{k \mid -i', -j'}(u \cdot p q^2 r s)_2.
	\end{align*}
	Also,
	\begin{align*}
		c_{k \mid ij'}\bigl(
			u \cdot \langle u, v \rangle p^2
		\bigr)_3\,
		c_{k \mid ij}(v \cdot \rho(u) p)_3\,
		&=
		c_{k \mid ij}\bigl(
			u \cdot \langle u, v \rangle p
		\bigr)_3\,
		c_{k \mid ij'}(v \cdot \rho(u) p^2)_3,
		\\
		c_{k \mid ij}(u \cdot \rho(u) p)_3
		&=
		c_{k \mid i, -j}(u \cdot p)_3^{-1},
		\\
		c_{k \mid ij}\bigl(
			u \cdot \langle u, v \rangle p
		\bigr)_3\,
		c_{k \mid ij'}(u \cdot \rho(v) p)_3
		&=
		c_{k \mid ij}(v \cdot \rho(u) p)_3\,
		c_{k \mid ij'}\bigl(
			v \cdot \langle u, v \rangle p
		\bigr)_3,
		\\
		c_{k \mid ij}(u \cdot p^2 q^2 r^2)_3\,
		c_{k \mid -i, -j'}&(
			u \cdot p^2 q r
		)_3\,
		c_{k \mid -i', -j}(
			u \cdot p^3 q^3 r^2
		)_3\,
		c_{k \mid -i', -j'}(u \cdot p q^2 r)_3 \\
		&=
		c_{k \mid -i, -j}(
			u \cdot p^2 q^2 r^2
		)_3\,
		c_{k \mid ij'}(u \cdot p^2 q r)_3\,
		c_{k \mid i'j}(u \cdot p q r^2)_3\,
		c_{k \mid -i', -j'}(
			u \cdot p^3 q^2 r
		)_3.
	\end{align*}
\end{lemma}
\begin{proof}
	Let
	\[
		a_{ijj'k}(p, q, r)
		=
		y_{ij'k}(p q, r)\,
		y_{ijk}(p, q r)^{-1}
	\]
	for
	\( p, q, r \colon R \), it is triadditive. Simplifying
	\[
		z_{-i, k}\bigl(
			\phi(\lambda_{-i} p q r \lambda_i),
			\lambda_i^* s
		\bigr)\,
		y_{-k, k}(s^* \lambda p q r s)
	\]
	in two ways, we get
	\[
		\tag{
			\( \aleph \)
		}
		\label{b3-5-swap}
		a_{ijj'k}(p, q, r s)
		=
		a_{i, -j', -j, k}(
			\lambda^* r^* \lambda_{j'},
			\lambda_{j'}^* q^* \lambda_j,
			\lambda_j^* p^* s
		)^{-1},
	\]
	so
	\[
		a_{ijj'k}(p q, r, s t)
		=
		a_{i, -j', -j, k}(
			\lambda^* s^* \lambda_{j'},
			\lambda_{j'}^* r^* \lambda_j,
			\lambda_j^* q^* p^* t
		)^{-1}
		=
		a_{ijj'k}(q, r, s p^* t).
	\]
	By weak unitality there exists a unique biadditive morphism
	\( b_{ijj'k}(p, q) \) such that
	\[
		a_{ijj'k}(p, q, r) = b_{ijj'k}(q, p^* r),
		\quad
		b_{ijj'k}(p, q r s t)
		=
		b_{ijj'k}(p, q s r t).
	\]
	The identity (\ref{b3-5-swap}) takes the form
	\[
		b_{ijj'k}(p, q)
		=
		b_{i, -j', -j, k}(
			\lambda_{j'}^* p^* \lambda_j,
			\lambda \lambda_{j'}^* \lambda_j^* q
		)^{-1}.
	\]
	Also,
	\[
		\tag{
			\( \beth \)
		}
		\label{b3-5-cocycle}
		b_{ijj''k}(p q, r)
		=
		b_{ijj'k}(p, q r)\,
		b_{ij'j''k}(q, r p^*)
	\]
	by transforming
	\( y_{ij''k}(p q r, s) \) to
	\( y_{ijk}(p, q r s) \) in two ways and applying weak unitality.
	
	By lemmas \ref{b3-2} and \ref{b3-4} we easily have
	\begin{align*}
		a_{ijj'k}(p, q, r)
		&=
		a_{-k, -j', -j, -i}(
			\lambda_k^* r^* \lambda_{j'},
			\lambda_{j'}^* q^* \lambda_j,
			\lambda_j^* p^* \lambda_i
		),
		\\
		a_{ijj'k}\bigl( \langle u, v \rangle, p, q \bigr)
		&=
		c_{k \mid ij}\bigl(
			v \cdot \langle u, v \rangle p q
		\bigr)\,
		c_{k \mid ij}(u \cdot \rho(v) p q)^{-1}\,
		c_{k \mid ij'}\bigl(
			v \cdot \langle u, v \rangle p^2 q
		\bigr)^{-1}\,
		c_{k \mid ij'}(u \cdot \rho(v) p^2 q),
		\\
		a_{i, -j, j, k}(p, \lambda_j^* \rho(u), q)
		&=
		c_{k \mid ij}(u \cdot \rho(u) p q)\,
		c_{k \mid i, -j}(u \cdot p q)^{-2},
		\\
		a_{ijj'k}\bigl(
			p^*,
			\langle u, v \rangle,
			q
		\bigr)
		&=
		c_{k \mid i, -j}\bigl(
			u \cdot \langle u, v \rangle p q
		\bigr)\,
		c_{k \mid i, -j}(v \cdot \rho(u) p q)^{-1}\,
		c_{k \mid ij'}\bigl(
			v \cdot \langle u, v \rangle p q
		\bigr)^{-1}\,
		c_{k \mid ij'}(u \cdot \rho(v) p q),
	\end{align*}
	so
	\(
		b_{ijj'k}(p, q)
	\) satisfy
	\begin{align*}
		b_{ijj'k}(p, q)
		&=
		b_{-k, -j', -j, -i}(
			\lambda_{j'}^* p^* \lambda_j,
			\lambda_i
			\lambda_j^*
			\lambda_{j'}^*
			\lambda_k
			q
		),
		\\
		b_{ijj'k}\bigl( p, \langle u, v \rangle \bigr)
		&=
		c_{k \mid ij}\bigl(
			u \cdot \langle u, v \rangle p
		\bigr)\,
		c_{k \mid ij}(v \cdot \rho(u) p)^{-1}\,
		c_{k \mid ij'}\bigl(
			u \cdot \langle u, v \rangle p^2
		\bigr)^{-1}\,
		c_{k \mid ij'}(v \cdot \rho(u) p^2),
		\\
		\tag{
			\( \daleth \)
		}
		\label{b3-5-rho}
		b_{i, -j, j, k}(\lambda_j^* \rho(u), p)
		&=
		c_{k \mid ij}(u \cdot \rho(u) p)\,
		c_{k \mid i, -j}(u \cdot p)^{-2},
		\\
		b_{ijj'k}\bigl( \langle u, v \rangle, p \bigr)
		&=
		c_{k \mid i, -j}\bigl(
			u \cdot \langle u, v \rangle p
		\bigr)\,
		c_{k \mid i, -j}(v \cdot \rho(u) p)^{-1}\,
		c_{k \mid ij'}\bigl(
			v \cdot \langle u, v \rangle p
		\bigr)^{-1}\,
		c_{k \mid ij'}(u \cdot \rho(v) p).
	\end{align*}
	Note that (\ref{b3-5-rho}) implies the identity
	\(
		b_{i, -j, j, k}(q, p)
		=
		b_{i, -j, j, k}(\lambda_j^* q^* \lambda_{-j}, p)
	\)
	using the substitution
	\( u = \phi(\lambda_j q) \). On the other hand,
	\(
		b_{i, -j, j, k}(p, q)
		=
		b_{i, -j, j, k}(
			\lambda_j^* p^* \lambda_{-j},
			q
		)^{-1}
	\) as a particular case of (\ref{b3-5-swap}). It follows that
	\( b_{i, -j, j, k}(p, q)^2 = 1 \),
	so by (\ref{b3-5-cocycle}) all morphisms
	\(
		b_{i j j' k}
	\) are
	\( 2 \)-torsion.
	
	Simplifying both sides of
	\[
		z_{jk}\bigl( \phi(r^* p q r), s \bigr)\,
		y_{-j, i, k}\bigl(
			\lambda_j^* r^* (q^* p^* \lambda - p q),
			r s
		\bigr)
		=
		z_{ik}(\phi(p q), r s)\,
		y_{-i, j, k}\bigl(
			\lambda_i^* (q^* p^* \lambda - p q) r,
			s
		\bigr)
	\]
	using lemma \ref{b3-4} we obtain
	\[
		b_{-i, j', j, k}(q r, \lambda_i p^* s)\,
		b_{-j, -i', i, k}(
			\lambda_{i'}^* p^* \lambda,
			\lambda_{i'}^* \lambda_j q r^2 s
		)
		=
		b_{-j, i', i, k}(
			q,
			\lambda_j p^* r^2 s
		)\,
		b_{-i, -j', j, k}(
			\lambda_{j'}^* p^* \lambda r,
			\lambda_i \lambda_{j'}^* q s
		).
	\]
	Finally, comparing results of transforms
	\[
		z_{jk}(u \cdot p q r, s)
		=
		z_{ik}(u, p q r s)\, \cdots
		\text{ and }
		z_{jk}(u \cdot p q r, s)
		=
		z_{i'k}(u \cdot p q, r s)\, \cdots
		=
		z_{j'k}(u \cdot p, q r s)\, \cdots
		=
		z_{ik}(u, p q r s)\, \cdots
	\]
	from lemma \ref{b3-4} and using lemma \ref{b3-2} we get
	\begin{align*}
		b_{-i, j', j, k}&(
			q r,
			\lambda_i \rho(u)^* p^* s
		)\,
		b_{-j, i, i', k}(
			p q,
			\lambda_j \rho(u)^* p q r^2 s
		) =
		b_{-i', j', j, k}(
			q r,
			\lambda_{i'} \rho(u)^* p p^* q s
		)\,
		b_{-j', i, i', k}(
			p q,
			\lambda_{j'} \rho(u)^* p r s
		) \\
		&\cdot
		c_{k \mid ij}(u \cdot p^2 q^2 r^2 s)^{-1}\,
		c_{k \mid -i, j}(
			u \cdot \rho(u) p^2 q^2 r^2 s
		)^2\,
		c_{k \mid ij'}(u \cdot p^2 q r s)\,
		c_{k \mid -i, j'}(
			u \cdot \rho(u) p^2 q r s
		)^{-2} \\
		&\cdot
		c_{k \mid i'j}(u \cdot p q r^2 s)\,
		c_{k \mid -i', j}(
			u \cdot \rho(u) p^3 q^3 r^2 s
		)^{-2}\,
		c_{k \mid -i', -j'}(u \cdot p q^2 r s)^{-1}\,
		c_{k \mid -i', j'}(
			u \cdot \rho(u) p^3 q^2 r s
		)^2.
	\end{align*}
	All identities from the statement follow by taking
	\( 2 \)- and
	\( 3 \)-torsion components.
\end{proof}

\begin{lemma}
	\label{b3-6}
	Suppose that
	\( R \) is associative. Then the
	\( 3 \)-torsion homomorphisms
	\(
		c_{i \mid jk}({-})_3
		\colon
		\Delta
		\to
		\schur(\stlin(\mathsf B_3, R, \Delta))
	\) factor through
	\( \Delta_3 \), where
	\( (R_3, \Delta_3) \) is one of the distinguished factor-algebras of
	\( (R, \Delta) \). They satisfy the symmetry relation
	\[
		c_{\sigma(i) \mid \sigma(j) \sigma(k)}(u)_3
		=
		c_{i \mid jk}(u)_3^{\det(\sigma)}
		\text{ for }
		\sigma \in \mathrm W(\mathsf B_3).
	\]
\end{lemma}
\begin{proof}
	Most identities are proved in lemma \ref{b3-2}. The identity
	\(
		c_{i \mid jk}(u \cdot \rho(u) p)_3
		=
		c_{i \mid j, -k}(u \cdot p)_3^{-1}
	\) from lemma \ref{b3-5} implies
	\[
		c_{i \mid jk}(u \cdot p q)_3
		=
		c_{i \mid j, -k}(u \cdot \rho(u) p^3 q)_3^{-1}
		=
		c_{i \mid jk}(u \cdot p^3 q)_3,
		\quad
		c_{i \mid jk}(u \cdot p^3)_3
		=
		c_{i \mid jk}(u \cdot p)_3.
	\]
	The symmetry relation follows from lemma \ref{b3-2} and
	\[
		c_{i \mid jk}\bigl(
			\iota \cdot p \langle \iota, u \rangle
		\bigr)_3
		=
		c_{i \mid j, -k}\bigl(
			\iota \cdot p^3 \langle \iota, u \rangle
		\bigr)_3^{-1}
		=
		c_{i \mid j, -k}\bigl(
			\iota \cdot p \langle \iota, u \rangle
		\bigr)_3^{-1}.
	\]
	Finally, linearizing the identity
	\[
		c_{k \mid ij'}\bigl(
			u \cdot \langle u, v \rangle p^2
		\bigr)_3\,
		c_{k \mid ij}(v \cdot \rho(u) p)_3\,
		=
		c_{k \mid ij}\bigl(
			u \cdot \langle u, v \rangle p
		\bigr)_3\,
		c_{k \mid ij'}(v \cdot \rho(u) p^2)_3
	\]
	from lemma
	\ref{b3-5} by
	\( p \) we get
	\[
		c_{i \mid jk}(v \cdot \rho(u))_3\,
		=
		c_{i \mid jk}\bigl(
			u \cdot \langle u, v \rangle
		\bigr)_3,
	\]
	so
	\[
		c_{i \mid jk}(u \cdot \rho(v))_3
		=
		c_{i \mid jk}\bigl(
			\iota \cdot \langle \iota, u \rangle \rho(v)
		\bigr)_3
		=
		c_{i \mid jk}\bigl(
			v
			\cdot
			\langle \iota, u \rangle
			\langle \iota, v \rangle
		\bigr)_3
		=
		c_{i \mid jk}\bigl(
			u \cdot \langle \iota, v \rangle^2
		\bigr)_3
	\]
	and the similar relation with
	\( \langle v, w \rangle \) instead of
	\( \rho(v) \) is its linearization.
\end{proof}

\begin{lemma}
	\label{b3-7}
	Again suppose that
	\( R \) is associative. Then
	\begin{align*}
		c_{i \mid jk}\bigl(
			u \cdot (p^2 - p) (q^2 - q)
		\bigr)_2
		&=
		1,
		&
		c_{i \mid jk}(u)_2
		&=
		c_{w(i) \mid w(j) w(k)}(u)_2
		\text{ for }
		w \in \mathrm W(\mathsf D_3),
		\\
		c_{i \mid jk}(u)_2\,
		c_{i \mid jk}\bigl(
			\iota \cdot \langle \iota, u \rangle
		\bigr)_2
		&=
		c_{-i \mid jk}(u)_2,
		&
		c_{i \mid jk}(\iota \cdot p)_2
		&=
		c_{w(i) \mid w(j) w(k)}(\iota \cdot p)_2
		\text{ for }
		w \in \mathrm W(\mathsf B_3),
		\\
		c_{i \mid jk}\bigl(
			\iota \cdot \langle \iota, u \rangle (p^2 - p)
		\bigr)_2
		&=
		1,
		&
		c_{i \mid jk}\bigl( u \cdot (p^2 - p) \bigr)_2
		&=
		c_{w(i) \mid w(j) w(k)}\bigl
			(u \cdot (p^2 - p)
		\bigr)_2
		\text{ for }
		w \in \mathrm W(\mathsf B_3).
	\end{align*}
	Also,
	\[
		b_{ijj'k}(p, \lambda^{\pm 1} q)
		=
		b_{ijj'k}(p, q)\,
		c_{k \mid ij}\bigl(
			\iota \cdot (p^2 - p) q
		\bigr)_2,
		\quad
		b_{ijj'k}(p \lambda^{\pm 1} q, r)
		=
		b_{ijj'k}(p q, r)\,
		c_{k \mid ij}\bigl(
			\iota \cdot (p^2 q^2 - p q) r
		\bigr)_2.
	\]
\end{lemma}
\begin{proof}
	Recall the relations
	\begin{align*}
		\tag{
			\( \blacklozenge \)
		}
		\label{b3-7-pi-r}
		b_{ijj'k}\bigl( p, \langle u, v \rangle \bigr)
		&=
		c_{k \mid ij}\bigl(
			u \cdot \langle u, v \rangle p
		\bigr)_2\,
		c_{k \mid ij}(v \cdot \rho(u) p)_2\,
		c_{k \mid ij'}\bigl(
			u \cdot \langle u, v \rangle p^2
		\bigr)_2\,
		c_{k \mid ij'}(v \cdot \rho(u) p^2)_2,
		\\
		\tag{
			\( \S \)
		}
		\label{b3-7-rho}
		b_{i, -j, j, k}(\lambda_j^* \rho(u), p)
		&=
		c_{k \mid ij}(u \cdot \rho(u) p)_2,
		\\
		\tag{
			\( \surd \)
		}
		\label{b3-7-pi-l}
		b_{ijj'k}\bigl( \langle u, v \rangle, p \bigr)
		&=
		c_{k \mid i, -j}\bigl(
			u \cdot \langle u, v \rangle p
		\bigr)_2\,
		c_{k \mid i, -j}(v \cdot \rho(u) p)_2\,
		c_{k \mid ij'}\bigl(
			v \cdot \langle u, v \rangle p
		\bigr)_2\,
		c_{k \mid ij'}(u \cdot \rho(v) p)_2
	\end{align*}
	from lemma \ref{b3-5}. From (\ref{b3-7-pi-r}) and lemma \ref{b3-2} we get
	\begin{align*}
		b_{ijj'k}\bigl(
			p,
			\langle u \cdot q, v \rangle
		\bigr)
		&=
		b_{ijj'k}\bigl(
			p,
			\langle u, v \cdot q^* \rangle
		\bigr),
		\\
		c_{k \mid ij}\bigl(
			u \cdot \langle u, v \rangle p (q^2 - q)
		\bigr)_2\,
		c_{k \mid ij}\bigl(
			v \cdot \rho(u) p (q^2 - q)
		\bigr)_2
		&=
		c_{k \mid ij'}\bigl(
			u \cdot \langle u, v \rangle p^2 (q^2 - q)
		\bigr)_2\,
		c_{k \mid ij'}\bigl(
			v \cdot \rho(u) p^2 (q^2 - q)
		\bigr)_2,
		\\
		c_{k \mid ij}\bigl(
			\iota \cdot p (q^2 - q) r
		\bigr)_2
		&=
		c_{k \mid ij'}\bigl(
			\iota \cdot p^2 (q^2 - q) r
		\bigr)_2,
		\\
		c_{k \mid ij}\bigl(
			\iota \cdot (p^2 - p) (q^2 - q) r
		\bigr)_2
		&=
		1,
		\\
		c_{k \mid ij}\bigl(
			\iota \cdot p (q^2 - q)
		\bigr)_2
		&=
		c_{k \mid i, -j}\bigl(
			\iota \cdot p (q^2 - q)
		\bigr)_2,
		\\
		c_{k \mid ij}\bigl(
			v \cdot p (q^2 - q)
		\bigr)_2
		&=
		c_{k \mid ij'}\bigl(
			v \cdot p^2 (q^2 - q)
		\bigr)_2,
		\\
		c_{k \mid ij}\bigl(
			u \cdot (p^2 - p) (q^2 - q)
		\bigr)_2
		&=
		1,
		\\
		c_{k \mid ij}\bigl( u \cdot (p^2 - p) \bigr)_2
		&=
		c_{k \mid i, -j}\bigl( u \cdot (p^2 - p) \bigr)_2.
	\end{align*}
	Simplifying
	\(
		b_{i, -j, j, k}\bigl(
			\lambda_j^* \rho(u),
			\langle v, w \rangle
		\bigr)
	\) in two ways via (\ref{b3-7-pi-r}) and (\ref{b3-7-rho}) we get
	\begin{align*}
		c_{k \mid ij}\bigl(
			v \cdot \langle v, w \rangle \rho(u)^2
		\bigr)_2\,
		c_{k \mid ij}(w \cdot \rho(v) \rho(u)^2)_2\,
		c_{k \mid ij}\bigl(
			u \cdot \rho(u) \langle v, w \rangle
		\bigr)_2
		&=
		c_{k \mid i, -j}\bigl(
			v \cdot \langle v, w \rangle \rho(u)
		\bigr)_2\,
		c_{k \mid i, -j}(w \cdot \rho(v) \rho(u))_2,
		\\
		c_{k \mid ij}(\iota \cdot p \rho(u)^2)_2
		&=
		c_{k \mid i, -j}(\iota \cdot p \rho(u))_2,
		\\
		c_{k \mid ij}(\iota \cdot p)_2
		&=
		c_{k \mid i, -j}(\iota \cdot p)_2,
		\\
		c_{k \mid ij}\bigl(
			\iota \cdot p (\rho(u)^2 - \rho(u))
		\bigr)_2
		&=
		1,
		\\
		c_{k \mid ij}(w \cdot p^2 q^2)_2\,
		c_{k \mid ij}\bigl(
			\iota \cdot \langle \iota, w \rangle p q^3
		\bigr)_2
		&=
		c_{k \mid i, -j}(w \cdot p^2 q^2)_2,
		\\
		c_{k \mid i, -j}(u)_2
		&=
		c_{k \mid ij}(u)_2\,
		c_{k \mid ij}\bigl(
			\iota \cdot \langle \iota, u \rangle
		\bigr)_2,
		\\
		c_{k \mid ij}\bigl(
			\iota \cdot \langle \iota, u \rangle (p^2 - p)
		\bigr)_2
		&=
		1.
	\end{align*}
	Now the symmetry relations follows from lemma \ref{b3-2}. The last two identities follow from by substitution
	\( u = \iota \cdot q \),
	\( v = \iota \cdot r \) in (\ref{b3-7-pi-r}) and (\ref{b3-7-pi-l}).
\end{proof}

\begin{lemma}
	\label{b3-8}
	Still suppose that
	\( R \) is associative. Then there are unique
	\(
		b_4
		\colon
		I_4 \to \schur(\stlin(\mathsf B_3, R, \Delta))
	\) and
	\(
		b_\eps
		\colon
		R_{2 \eps}
		\to
		\schur(\stlin(\mathsf B_3, R, \Delta))
	\) such that
	\[
		b_{ijjk}(p, q)
		=
		b_4\bigl( (p + p^*) q \bigr)\,
		b_\eps\bigl( (p^2 - p) q \bigr),
		\quad
		b_{i, -j, j, k}(p, q)
		=
		b_{ijjk}(p, q)\,
		c_{i \mid jk}(\iota \cdot p q)_2,
		\quad
		b_\eps(p^2) = 1.
	\]
	Moreover,
	\(
		c_{i \mid jk}\bigl(\iota \cdot (p^2 - p) q \bigr)
		=
		1
	\).
\end{lemma}
\begin{proof}
	Recall the Leibniz rule
	\[
		b_{ijj''k}(p q, r)
		=
		b_{ijj'k}(p, q r)\,
		b_{ij'j''k}(q, r p^*)
	\]
	from lemma \ref{b3-5}. First taking
	\( j = j' = j'' \),
	\( p = q = \eps \colon \mathcal E \)
	and then taking
	\( j' = j'' \) and
	\( q = \eps^2 \) for
	\( \eps \colon \mathcal E \) we get
	\[
		b_{ijjk}(\eps^2, p) = 1,
		\quad
		b_{ijj'k}(p \eps^2, q)
		=
		b_{ijj'k}(p, \eps^2 q).
	\]
	Next look at the identity
	\[
		\tag{
			\( \blacksquare \)
		}
		\label{b3-8-quad}
		b_{-i, j', j, k}(q r, \lambda_i p s)\,
		b_{-i, -j', j, k}(p r, \lambda_{-i} q s)
		=
		b_{-j, i', i, k}(q, \lambda_j p r^2 s)\,
		b_{-j, -i', i, k}(p, \lambda_{-j} q r^2 s)
	\]
	from lemma \ref{b3-5} (we simplified this identity using lemma
	\ref{b3-7}). Substituting
	\( i = i' \),
	\( j = j' \),
	\( q = \eps^2 \),
	\( r = \eta^2 \) for
	\( \eps, \eta \colon \mathcal E \) we get
	\begin{align*}
		b_{-i, -j, j, k}(p, \lambda_{-i} \eta^2 s)
		&=
		b_{-j, -i, i, k}(p, \lambda_{-j} \eta^4 s),
		\\
		b_{-i, -j, j, k}(p, \lambda_{-i} \eps^2 \eta^2 s)
		&=
		b_{-j, -i, i, k}(p, \lambda_{-j} \eps^4 \eta^2 s)
		=
		b_{-i, -j, j, k}(p, \lambda_{-i} \eps^4 \eta^4 s),
		\\
		b_{-i, -j, j, k}(p, \eps^2 s)
		&=
		b_{-i, -j, j, k}(p, \eps^4 s),
		\\
		b_{-i, -j, j, k}(p, \lambda_{-i} s)
		&=
		b_{-j, -i, i, k}(p, \lambda_{-j} s).
	\end{align*}
	Taking
	\( i = i' \),
	\( j = j' \),
	\( r = \eps^2 \) we get
	\begin{align*}
		b_{-i, j, j, k}(q, \lambda_i \eps^2 s)
		&=
		b_{-j, i, i, k}(q, \lambda_j \eps^4 s),
		\\
		b_{-i, j, j, k}(q, \lambda_i \eps^2 \eta^2 s)
		&=
		b_{-j, i, i, k}(q, \lambda_j \eps^4 \eta^2 s)
		=
		b_{-i, j, j, k}(q, \lambda_i \eps^4 \eta^4 s),
		\\
		b_{-i, j, j, k}(p, \eps^2 q)
		&=
		b_{-i, j, j, k}(p, \eps^4 q),
		\\
		b_{-i, j, j, k}(p, \lambda_i q)
		&=
		b_{-j, i, i, k}(p, \lambda_j q).
	\end{align*}
	Now let
	\( i = i' \),
	\( j = j' \),
	\( q = \eps^2 \), so
	\begin{align*}
		b_{-i, -j, j, k}(p r, s)
		&=
		b_{-i, j, j, k}(r, \lambda p s)\,
		b_{-i, -j, j, k}(p, r^2 s).
	\end{align*}
	Then the result of a more general substitution
	\( i = i' \),
	\( j = j' \) can be simplified to get
	\begin{align*}
		b_{ijjk}(q r, s)
		&=
		b_{ijjk}(q, r^2 s)\,
		b_{ijjk}(r, q s),
		\\
		b_{ijjk}(r, (q - q^*) s)
		&=
		b_{ijjk}(q, (r^2 - r) s).
	\end{align*}
	Next take
	\( i' = -i \),
	\( j = j' \),
	\( q = \eps^4 \) and apply the identity
	\[
		b_{i, -j, j, k}(\lambda_j^* \rho(u), p)
		=
		c_{k \mid ij}(u \cdot \rho(u) p)_2
	\]
	from lemma \ref{b3-5}, so
	\[
		b_{i, -j, j, k}(p, q)
		=
		b_{i, j, j, k}(p, q)\,
		c_{i \mid jk}(\iota \cdot p q).
	\]
	Then
	\begin{align*}
		b_{ijjk}(p, q r)\,
		b_{ijjk}(q, p^* r)
		=
		b_{ijjk}(p q, r)
		&=
		b_{i, j, -j, k}(p, q r)\,
		b_{ijjk}(q, p^* r)\,
		c_{i \mid jk}(\iota \cdot p q r)_2,
		\\
		b_{ijjk}(p, q)\,
		&=
		b_{i, j, -j, k}(p, q)\,
		c_{i \mid jk}(\iota \cdot p q)_2,
		\\
		b_{i, -j, -j, k}(p, q)
		&=
		b_{ijjk}(p, q).
	\end{align*}
	
	Note that the identity
	\(
		b_{ijj'k}(p, q)
		=
		b_{i, -j', -j, k}(p^*, \lambda q)
	\) from lemma \ref{b3-5} means that
	\[
		b_{ijjk}(p^*, q)
		=
		b_{ijjk}(p, q)\,
		c_{i \mid jk}(\iota \cdot (p^2 - p) q)_2.
	\]
	Combining this with
	\(
		b_{ijjk}(p, \lambda_i q)
		=
		b_{jiik}(p, \lambda_j q)
	\) and the identity
	\(
		b_{ijj'k}(p, q)
		=
		b_{-k, -j', -j, -i}(p^*, \lambda_i \lambda_k q)
	\) from the lemma \ref{b3-5} we see that
	\[
		b_{=}(p, q)
		=
		b_{ijjk}(p, \lambda_i q)
		=
		b_{i, -j, j, k}(p, \lambda_i q)\,
		c_{i \mid jk}(\iota \cdot p q)_2
	\]
	is independent of the indices.
	
	Now take the longest identity from lemma \ref{b3-5}, after simplification it takes the form
	\[
		b_{-i, j', j, k}(q r, \lambda_i p^* s)\,
		b_{-i', j', j, k}(q r, \lambda_{i'} p p^* q s)
		=
		b_{-j, i, i', k}(p q, \lambda_j p q r^2 s)\,
		b_{-j', i, i', k}(p q, \lambda_{j'} p r s)
	\]
	because the remaining four factors
	\[
		c_{k \mid ij}(u \cdot p^2 q^2 r^2 s)_2\,
		c_{k \mid ij'}(u \cdot p^2 q r s)_2\,
		c_{k \mid i'j}(u \cdot p q r^2 s)_2\,
		c_{k \mid -i', -j'}(u \cdot p q^2 r s)_2
	\]
	cancel out by lemma \ref{b3-8}. Expressing everything in terms of
	\( b_{=}({-}, {=}) \) we get
	\begin{align*}
		\tag{
			\( \blacktriangle \)
		}
		\label{b3-8-i2}
		c_{i \mid jk}\bigl(
			\iota \cdot (p^2 - p) q
		\bigr)_2 &= 1,
		\\
		b_{=}(q r, p^* s)\,
		b_{=}(q r, p p^* q s)
		&=
		b_{=}(p q, p q r^2 s)\,
		b_{=}(p q, p r s).
	\end{align*}
	Expand the last identity with
	\( q = \eps^2 \).
	\begin{align*}
		b_{=}(r, p^* s)\,
		b_{=}(r, p p^* s)
		&=
		b_{=}\bigl( p, p (r^2 - r) s \bigr)
		=
		b_{=}\bigl( r, p (p - p^*) s \bigr)
		\\
		b_{=}(p, q^* r)
		&=
		b_{=}(p, q^2 r).
	\end{align*}
	The last identity together with the two forms of the Leibniz rule mean that
	\( b_{=}(p q, r) = b_{=}(q p, r) \) and, more generally,
	\[
		b_{=}(p q r s, t) = b_{=}(p r q s, t).
	\]
	
	Now let
	\[
		b_4(p, q) = b_{=}(p^2, q),
		\quad
		b_\eps(p, q) = b_{=}(p^2 - p, q),
	\]
	so
	\( b_{=}(p, q) = b_4(p, q)\, b_\eps(p, q) \). The morphism
	\( b_4 \) is well defined on
	\( R_4 \times R_4 \) because
	\begin{align*}
		b_4\bigl( p (q^2 - q^*), r \bigr)
		&=
		b_{=}\bigl( p^2 (q^2 - q^*), (q^2 - q^*) r \bigr)\,
		b_{=}\bigl(
			q^2 - q^*,
			p^{* 2} (q^2 - q^*)^* r
		\bigr)
		=
		1,
		\\
		b_4\bigl(
			p,
			\langle u, v \rangle
		\bigr)
		&=
		b_{=}\bigl(
			p^2,
			\langle u, v \rangle
		\bigr)
		=
		1,
		\\
		b_4\bigl( p, \rho(\iota \cdot \rho(u)) q \bigr)
		&=
		b_{=}(p^2, \rho(u)^2 q)
		=
		b_{=}(p^2, \rho(u)^* q)
		=
		b_4(p, \rho(u) q),
		\\
		b_4(\rho(u), p)
		&=
		b_{=}\bigl(
			\rho(u),
			p (\rho(u) + \rho(u)^*)
		\bigr)
		=
		1,
		\\
		b_4\bigl( \langle u, v \rangle, q \bigr) &= 1,
		\\
		b_4\bigl( p \rho(\iota \cdot \rho(u)), q \bigr)
		&=
		b_4\bigl( p, \rho(\iota \cdot \rho(u)) q \bigr)
		=
		b_4(p, \rho(u) q)
		=
		b_4(p \rho(u), q).
	\end{align*}
	We claim that there is a well defined morphism
	\[
		b_4
		\colon
		I_4 \to \schur(\stlin(\mathsf B_3, R, \Delta)),
		\quad
		b_4(p, q) = b_4((p + p^*) q).
	\]
	Indeed,
	\(
		[\![ (p + p^*) q ]\!]
		\colon
		R_4 \times R_4 \to I_4
	\) is an epimorphism and if
	\( (p + p^*) q = (r + r^*) s \colon R_4 \), then
	\begin{align*}
		b_4(p, q)
		&=
		b_{=}\bigl( p, (p + p^*) q \bigr)
		=
		b_{=}\bigl( p, (p + p^*)^2 q \bigr)
		=
		b_4\bigl( p, (p + p^*) q \bigr),
		\\
		b_4\bigl( p, (q + q^*) r \bigr)
		&=
		b_{=}\bigl( p^2, (q^2 + q^{* 2}) r \bigr)
		=
		b_{=}\bigl( q^2, (p^2 + p^{* 2}) r \bigr)
		=
		b_4\bigl( q, (p + p^*) r \bigr),
		\\
		b_4(p, q)
		&=
		b_4(p, (p + p^*) q)
		=
		b_4(p, (r + r^*) s)
		=
		b_4(r, (p + p^*) (r + r^*) s) \\
		&=
		b_4(r, (p + p^*) q)
		=
		b_4(r, (r + r^*) s)
		=
		b_4(r, s).
	\end{align*}
	Since any
	\( p, q \colon I_4 \) are of the form
	\( p = p' (r + r^*) \),
	\( q = q' (r + r^*) \), the morphism
	\(
		b_4
		\colon
		I_4 \to \schur(\stlin(\mathsf B_3, R, \Delta))
	\) is a group homomorphism.
	
	Next let us deal with
	\( b_\eps(p, q) \). This morphism factors through
	\( R_{2 \eps} \times R_{2 \eps} \) because
	\begin{align*}
		b_\eps\bigl( p, (q - q^*) r \bigr)
		&=
		b_{=}\bigl( p^2 - p, (q - q^*) r \bigr)
		=
		b_{=}\bigl( q, (p + p^* + p^2 + p^{* 2}) r \bigr)
		=
		1,
		\\
		b_\eps\bigl( p (q - q^*), r \bigr)
		&=
		b_{=}\bigl( p^2 (q^2 - q^{* 2}), r \bigr)\,
		b_{=}\bigl( p (q - q^*), r \bigr) \\
		&=
		b_{=}\bigl( p^2, (q^2 - q^{* 2}) r \bigr)\,
		b_{=}(q^2 - q^{* 2}, p^{* 2} r)\,
		b_{=}\bigl( p, (q - q^*) r \bigr)\,
		b_{=}(q - q^*, p^* r)
		=
		1,
		\\
		b_\eps\bigl( p (q^2 - q) (r^2 - r), s \bigr)
		&=
		b_{=}\bigl(
			p (q^2 - q^*) (r^2 - r^*),
			\bigl(
				p (q^2 - q^*) (r^2 - r^*)
				+
				p^* (q^{* 2} - q) (r^{* 2} - r)
			\bigr)
			s
		\bigr) \\
		&\cdot
		b_{=}\bigl( p (q^2 - q^*), (r^2 - r^*) s \bigr)\,
		b_{=}\bigl( r^2 - r^*, p^* (q^{* 2} - q) s \bigr)
		=
		1.
	\end{align*}
	Recall that
	\( R_{2 \eps} = I_2 \rtimes R_2 \), where
	\( I_2 = \langle p^2 - p \rangle \leqt R_{2 \eps} \) and
	\( R_2 \) is the largest Boolean factor-ring of
	\( R_{2 \eps} \). There is a well defined homomorphism
	\[
		b_\eps
		\colon
		R_{2 \eps}
		\to
		\schur(\stlin(\mathsf B_3, R, \Delta)),\,
		b_\eps(p, q) = b_\eps\bigl( (p^2 - p) q \bigr),\,
		b_\eps(p^2) = 1.
	\]
	Indeed, the equality
	\( (p^2 - p) q = (r^2 - r) s \colon R_{2 \eps} \) implies
	\begin{align*}
		b_\eps(p^2, q)
		&=
		b_{=}(p^4 - p^2, q)
		=
		b_{=}\bigl(
			p,
			(
				p
				+
				p^*
				+
				p^3
				+
				p^2 p^*
				+
				p p^{* 2}
				+
				p^{* 3}
			)
			q
		\bigr)\,
		=
		1,
		\\
		b_\eps\bigl( (p^2 - p) q, r \bigr)
		&=
		b_{=}(p^2 q + p q + p^4 q^2 + p^2 q^2, r)
		=
		b_{=}(p^4, q^* r)\,
		b_{=}\bigl( p^2, (q + q^*) r \bigr)\,
		b_{=}(p, q r) \\
		&\quad \cdot
		b_{=}\bigl( q, (p + p^*) r \bigr)\,
		b_{=}\bigl( q^2, (p + p^*) r \bigr)
		=
		b_{=}(p, (1 + p + p^*) q r)
		=
		b_\eps(p^2 - p, q r),
		\\
		b_\eps(p, q r)
		&=
		b_{=}(p^2 - p, q r)
		=
		b_{=}(p^2 - p, q^{* 2} r)
		=
		b_\eps(p, q^2 r),
		\\
		b_\eps(p, q)
		&=
		b_\eps\bigl( (p^2 - p) q, q \bigr)
		=
		b_\eps\bigl( (r^2 - r) s, q \bigr)
		=
		b_\eps\bigl( (r^2 - r) s q, s \bigr)
		=
		b_\eps\bigl( (p^2 - p) q, s \bigr) \\
		&=
		b_\eps\bigl( (r^2 - r) s, s \bigr)
		=
		b_\eps(r, s),
	\end{align*}
	and any
	\( p, q \colon I_\eps \) are of the type
	\( p = ({p'}^2 - p') r \),
	\( q = ({q'}^2 - q') r \).
	
	Finally, the uniqueness claim follows from
	\( b_4((p + p^*) q) = b_{ijjk}(p, (p + p^*) q) \).
\end{proof}

\begin{lemma}
	\label{b3-9}
	If
	\( R \) is associative, then
	\( b_\eps(p) \) factors through
	\( R_{2 \eps \delta} \) (with the additional property
	\( b_\eps(p^2) = 1 \)) and
	\(
		d(u)
		=
		c_{1 \mid 23}(u)_2\,
		b_\eps\bigl( \langle \iota, u \rangle \bigr)
	\) factors through
	\( \Delta_{2 \mathrm b} \).
\end{lemma}
\begin{proof}
	By lemmas \ref{b3-7} and \ref{b3-8} we may assume that
	\( R = R_{2 \eps} \). Let
	\[
		d(u)
		=
		c_{1 \mid 23}(u)_2\,
		b_\eps\bigl( \langle \iota, u \rangle \bigr).
	\]
	We have the identities
	\begin{align*}
		b_\eps(p^2) &= 1,
		\\
		d(\iota \cdot p^2 q) &= d(\iota \cdot p q),
		\\
		\tag{
			\( \sphericalangle \)
		}
		\label{b3-9-pi-eps}
		b_\eps\bigl(
			\bigl(
				\langle u, v \rangle
				+
				\langle \iota, u \rangle
				\langle u, v \rangle
				+
				\langle \iota, v \rangle
				\rho(u)
			\bigr)
			(p^2 - p)
		\bigr)
		&=
		d\bigl(
			u \cdot \langle u, v \rangle (p^2 - p)
			\dotplus
			v \cdot \rho(u) (p^2 - p)
		\bigr),
		\\
		b_\eps\bigl(
			\rho(u) p^2
			+
			\langle \iota, u \rangle \rho(u) p
		\bigr)
		&=
		d(
			\iota \cdot \rho(u) p
			\dotplus
			u \cdot \rho(u) p
		),
		\\
		\tag{
			\( \blacktriangledown \)
		}
		\label{b3-9-pi-idem}
		b_\eps\bigl(
			\langle u, v \rangle p^2
			+
			\langle \iota, u \rangle \rho(v) p
			&+
			\langle \iota, v \rangle \rho(u) p
			+
			\langle \iota, u \rangle
			\langle u, v \rangle
			p
			+
			\langle \iota, v \rangle
			\langle u, v \rangle
			p
		\bigr) \\
		=
		d\bigl(
			\iota \cdot \langle u, v \rangle p
			&\dotplus
			u \cdot \rho(v) p
			\dotplus
			v \cdot \rho(u) p
			\dotplus
			u \cdot \langle u, v \rangle p
			\dotplus
			v \cdot \langle u, v \rangle p
		\bigr)
	\end{align*}
	from lemmas \ref{b3-5} (they already appeared as (\ref{b3-7-pi-r}), (\ref{b3-7-rho}), (\ref{b3-7-pi-l}) in lemma \ref{b3-7}), \ref{b3-7}, and \ref{b3-8}. Note that (\ref{b3-9-pi-eps}) with
	\( u = \iota \cdot q^2 \) reduces to
	\( d(v \cdot p^2) = d(v \cdot p) \).
	
	Next take
	\( v = \iota \cdot q^2 \) in (\ref{b3-9-pi-idem}) and get
	\[
		d(u)
		=
		d(\iota \cdot \rho(u))\,
		d\bigl( u \cdot \langle \iota, u \rangle \bigr).
	\]
	Linearizing this we further obtain
	\[
		d\bigl( u \cdot \langle \iota, v \rangle \bigr)\,
		d\bigl( v \cdot \langle \iota, u \rangle \bigr)
		=
		d\bigl( \iota \cdot \langle u, v \rangle \bigr).
	\]
	Also,
	\begin{align*}
		d\bigl(
			\iota
			\cdot
			\rho(u)
			\langle \iota, u \rangle
		\bigr)
		&=
		d\bigl( u \cdot \langle \iota, u \rangle \bigr)\,
		d\bigl( u \cdot \langle \iota, u \rangle^3 \bigr)
		=
		1,
		\\
		d\bigl(
			u
			\cdot
			\rho(u)
			\langle \iota, u \rangle
		\bigr)
		&=
		b_\eps\bigl(
			\rho(u) \langle \iota, u \rangle^2
			+
			\rho(u) \langle \iota, u \rangle^2
		\bigr)
		=
		1,
		\\
		d(u \cdot \rho(u) \dotplus \iota \cdot \rho(u))
		&=
		d\bigl(
			u \cdot \rho(u) \langle \iota, u \rangle
		\bigr)
		=
		1.
	\end{align*}
	The last identity admits the linearizations
	\begin{align*}
		d\bigl(
			u \cdot \rho(v)
			\dotplus
			v \cdot \rho(u)
			\dotplus
			u \cdot \langle u, v \rangle
			\dotplus
			v \cdot \langle u, v \rangle
			\dotplus
			\iota \cdot \langle u, v \rangle
		\bigr)
		&=
		1,
		\\
		d\bigl(
			u \cdot \langle v, w \rangle
			\dotplus
			v \cdot \langle w, u \rangle
			\dotplus
			w \cdot \langle u, v \rangle
		\bigr)
		&=
		1.
	\end{align*}
	We still have to check that
	\[
		d\bigl( w \cdot \langle u, v \rangle \bigr)
		=
		d\bigl(
			w
			\cdot
			\langle u, v \rangle
			\langle \iota, u \rangle
		\bigr)\,
		d\bigl(
			w \cdot \langle \iota, v \rangle \rho(u)
		\bigr).
	\]
	Indeed, for
	\( w = \iota \cdot p \) we have
	\[
		d\bigl(
			\iota
			\cdot
			\langle \iota, v \rangle
			\rho(u)
			p
		\bigr)
		=
		d\bigl( u \cdot \langle \iota, v \rangle p \bigr)\,
		d\bigl(
			u
			\cdot
			\langle \iota, u \rangle
			\langle \iota, v \rangle
			p
		\bigr)
		=
		d\bigl( \iota \cdot \langle u, v \rangle p \bigr)\,
		d\bigl(
			\iota
			\cdot
			\langle u, v \rangle
			\langle \iota, u \rangle
			p
		\bigr).
	\]
	In general case let
	\[
		w''
		=
		u
		\cdot
		\bigl(1 + \langle \iota, v \rangle \bigr)
		\bigl( 1 + \langle \iota, w \rangle \bigr)
		\dotplus
		v \cdot \bigl( 1 + \langle \iota, w \rangle \bigr)
		\dotplus
		w,
		\quad
		w' = w'' \cdot \langle \iota, w'' \rangle,
	\]
	so
	\( w' \equiv w' \cdot \langle \iota, w' \rangle \) and
	\(
		\langle \iota, \mu \rangle
		\langle \iota, w' \rangle
		\equiv
		\langle \iota, \mu \rangle
	\) modulo the ideal
	\( \langle p^2 - p \rangle \) for
	\( \mu \in \{ u, v, w \} \). Then
	\begin{align*}
		d\bigl( w' \cdot \langle u, v \rangle p \bigr)
		&=
		d\bigl( u \cdot \langle w', v \rangle p \bigr)\,
		d\bigl( v \cdot \langle w', u \rangle p \bigr)
		=
		d\bigl(
			w'
			\cdot
			\bigl(
				\langle \iota, u \rangle
				\langle w', v \rangle
				+
				\langle \iota, v \rangle
				\langle w', u \rangle
			\bigr)
			p
		\bigr),
		\\
		d(w' \cdot \rho(u) p)
		&=
		d(u \cdot \rho(w') \langle \iota, w' \rangle p)\,
		d\bigl( w' \cdot \langle w', u \rangle p \bigr)\,
		d\bigl(
			u
			\cdot
			\langle w', u \rangle
			\langle \iota, w' \rangle
			p
		\bigr)\,
		d(\iota \cdot \langle w', u \rangle p) \\
		&=
		d(w' \cdot \rho(w') \langle \iota, u \rangle p)\,
		d\bigl( w' \cdot \langle w', u \rangle p \bigr)\,
		d\bigl(
			w'
			\cdot
			\langle w', u \rangle
			\langle \iota, u \rangle
			p
		\bigr),
		\\
		d\bigl(
			w
			\cdot
			\bigl(
				\langle u, v \rangle
				&+
				\langle u, v \rangle
				\langle \iota, u \rangle
				+
				\langle \iota, v \rangle \rho(u)
			\bigr)
		\bigr) \\
		&=
		d\bigl(
			w
			\cdot
			\langle \iota, w \rangle
			\bigl(
				\langle u, v \rangle
				+
				\langle u, v \rangle
				\langle \iota, u \rangle
				+
				\langle \iota, v \rangle \rho(u)
			\bigr)
		\bigr) \\
		&=
		d\bigl(
			w'
			\cdot
			\langle \iota, w \rangle
			\bigl(
				\langle u, v \rangle
				+
				\langle u, v \rangle
				\langle \iota, u \rangle
				+
				\langle \iota, v \rangle \rho(u)
			\bigr)
		\bigr) \\
		&=
		d\bigl(
			w'
			\cdot
			\langle \iota, v \rangle
			\langle \iota, w \rangle
			\bigl(
				\langle w', u \rangle
				+
				\langle \iota, u \rangle
				\langle w', u \rangle
				+
				\rho(w') \langle \iota, u \rangle
				+
				\langle w', u \rangle
				+
				\langle \iota, u \rangle
				\langle w', u \rangle
			\bigr)
		\bigr)
		=
		1.
	\end{align*}
	
	It follows that the homomorphism
	\( b_\eps \) satisfies the identities from lemma \ref{b3-8} and
	\begin{align*}
		b_\eps\bigl(
			\langle u, v \rangle^2
			p
			+
			\langle u, v \rangle^2
			\langle \iota, u \rangle^2
			p
			+
			\langle \iota, v \rangle^2
			\rho(u)^2
			p
		\bigr)
		=
		1,
		\quad
		b_\eps\bigl(
			\rho(u) \langle \iota, u \rangle p
		\bigr)
		=
		b_\eps\bigl(
			\rho(u)^2 p + \rho(u) p
		\bigr).
	\end{align*}
	The linearizations of the right identity on
	\( u \) are
	\begin{align*}
		b_\eps\bigl(
			\rho(u) \langle \iota, v \rangle p
			+
			\rho(v) \langle \iota, u \rangle p
			+
			\langle u, v \rangle \langle \iota, u \rangle p
			+
			\langle u, v \rangle \langle \iota, v \rangle p
		\bigr)
		&=
		b_\eps\bigl(
			\bigl(
				\langle u, v \rangle^2
				+
				\langle u, v \rangle
			\bigr)
			p
		\bigr),
		\\
		b_\eps\bigl(
			\langle u, v \rangle \langle \iota, w \rangle
			\dotplus
			\langle v, w \rangle \langle \iota, u \rangle
			\dotplus
			\langle w, u \rangle \langle \iota, v \rangle
		\bigr)
		&=
		1.
	\end{align*}
	Finally, we have
	\begin{align*}
		b_\eps\bigl(
			\rho\bigl(
				u \cdot \langle \iota, v \rangle
				&\dotplus
				v \cdot \langle \iota, u \rangle
				\dotplus
				\iota \cdot \langle u, v \rangle
			\bigr)
			p
		\bigr)
		=
		b_\eps\bigl(
			\langle \iota, v \rangle^2 \rho(u) p
			+
			\langle \iota, u \rangle^2 \rho(v) p
			+
			\langle \iota, u \rangle
			\langle \iota, v \rangle
			\langle u, v \rangle
			p
			+
			\langle u, v \rangle^2 p
		\bigr) \\
		&=
		b_\eps\bigl(
			\langle \iota, v \rangle^2 \rho(u) p
			+
			\langle \iota, u \rangle^2 \rho(v) p
			+
			\langle \iota, u \rangle X p
			+
			\langle u, v \rangle^2 p
		\bigr) \\
		&=
		b_\eps\bigl(
			\langle \iota, v \rangle^2 \rho(u) p
			+
			\langle \iota, u \rangle
			\langle \iota, v \rangle
			\rho(u)
			p
			+
			\langle \iota, u \rangle^2
			\langle u, v \rangle^2
			p
			+
			\langle u, v \rangle^2 p
		\bigr) \\
		&=
		b_\eps\bigl(
			\bigl(
				\langle \iota, v \rangle^2
				-
				\langle \iota, v \rangle
			\bigr)
			(\rho(u)^2 - \rho(u))
			p
		\bigr)
		=
		1,
	\end{align*}
	where
	\(
		X
		=
		\langle \iota, u \rangle
		\langle u, v \rangle
		+
		\langle \iota, u \rangle \rho(v)
		+
		\langle \iota, v \rangle \rho(u)
		+
		\langle u, v \rangle^2
		+
		\langle u, v \rangle
	\).
\end{proof}

\begin{prop}
	\label{b3-schur}
	Let
	\( (R, \Delta) \) be a weakly unital
	\( \mathsf B_3 \)-ring in an infinitary pretopos. Then the Schur multiplier
	\( \schur(\stlin(\mathsf B_3, R, \Delta)) \) is the product of the images of
	\begin{align*}
		c_{1 \mid 23}({-})_3
		&\colon
		\Delta_3
		\to
		\schur(\stlin(\mathsf B_3, R, \Delta)),
		&
		b_4
		&\colon
		I_4
		\to
		\schur(\stlin(\mathsf B_3, R, \Delta)),
		\\
		b_\eps
		&\colon
		I_{2 \eps \delta}
		\to
		\schur(\stlin(\mathsf B_3, R, \Delta)),
		&
		d
		\colon
		&\Delta_{2 \mathrm b}
		\to
		\schur(\stlin(\mathsf B_3, R, \Delta)),
	\end{align*}
	where
	\[
		y_{123}(p q, r)\, y_{123}(p, q r)
		=
		b_4\bigl( (q + q^*) p^* r \bigr)\,
		b_\eps\bigl( (q^2 - q) p^* r \bigr),
		\quad
		d(u)
		=
		c_{1 \mid 23}(u)_2\,
		b_\eps\bigl(
			\langle \iota, u \rangle^2
			-
			\langle \iota, u \rangle
		\bigr)
	\]
	for
	\( [p, q, r] = 0 \).
\end{prop}
\begin{proof}
	We have to check that the homomorphisms from the statement are well defined because
	the ring
	\( R \) is assumed to be associative in lemmas \ref{b3-5}--\ref{b3-9}. For
	\( c_{i \mid jk}({-})_3 \) this follows from lemmas \ref{wunit-artin} and \ref{b3-6}. Let
	\[
		b_{ijj'k}(p, \eps q)
		=
		y_{ijk}(\eps, p q)^{-1}\,
		y_{ij'k}(\eps p, q)
	\]
	for
	\( \eps \colon \mathcal E \), this is well defined by lemmas \ref{wunit-artin} and \ref{b3-5}. By the same lemmas
	\(
		y_{ij'k}(p q, r)
		=
		y_{ijk}(p, q r)\,
		b_{ijj'k}(q, p^* r)
	\) for
	\( [p, q, r] = 1 \). Moreover,
	\( b_{ijj'k}({-}, {=}) \) is
	\( 2 \)-torsion and a homomorphism on each variable.
	
	Linearize the identities
	\( b_{ijjk}(p, q^2) = b_{ijjk}(p, q^*) \) and
	\( b_{ijjk}(p, q^3) = b_{ijjk}(p, q q^*) \) from lemmas \ref{wunit-artin} and \ref{b3-8}.
	\begin{align*}
		b_{ijjk}(p, q r) &= b_{ijjk}(p, r q),
		\\
		b_{ijjk}(p, q^2 r + q r^2)
		&=
		b_{ijjk}(p, q r^* + r q^*),
		\\
		b_{ijjk}(p, (q r) s)
		&=
		b_{ijjk}(p, (r q) s).
	\end{align*}
	Now recall the identity
	\[
		\tag{%
			\( \# \)%
		}
		\label{b3-schur-ass}
		[p q, r] + [q r, p] + [r p, q]
		=
		[p, q, r] + [q, r, p] + [r, p, q]
		=
		3 [p, q, r]
	\]
	valid in all alternative rings. By easy induction
	\( b_{ijjk}({-}, {=}) \) factors through the maximal associative factor-ring on the second argument.
	
	Next let
	\[
		b_4(p, q) = b_{1223}(p^2, q),
		\quad
		b_\eps(p, q) = b_{1223}(p^2 - p, q)
	\]
	as in the proof of lemma \ref{b3-8}. By this lemma (and lemma \ref{wunit-artin}) we have
	\[
		b_4(p^2, q) = b_4(p, q),
		\quad
		b_4(p^3, q) = 1,
		\quad
		b_\eps(p, q \eps)
		=
		b_\eps\bigl( (p^2 - p) q, \eps \bigr),
		b_\eps(p q, \eps)
		=
		b_\eps(q p, \eps)
		=
		b_\eps(p^* q, \eps)
	\]
	for
	\( \eps \colon \mathcal E \). Linearize the identities for
	\( b_4({-}, {=}) \).
	\begin{align*}
		b_4(p q, r) = b_4(q p, r),
		b_4(p q^2, r) = b_4(p^2 q, r),
		b_4((p q) r, s) = b_4((q p) r, s),
	\end{align*}
	so
	\( b_4({-}, {=}) \) factors through the maximal associative factor-ring on the first argument. For
	\( b_\eps({-}, {=}) \) note that
	\(
		b_\eps\bigl( (p q) r, \eps \bigr)
		=
		b_\eps\bigl( p (q r), \eps \bigr)
	\) by (\ref{b3-schur-ass}), so
	\[
		b_\eps\bigl( (p q) r, \eps \bigr)
		=
		b_\eps\bigl( p^* (q r), \eps \bigr)
		=
		b_\eps\bigl( (p^* q) r, \eps \bigr)
		=
		b_\eps\bigl( (q^* p) r, \eps \bigr)
		=
		b_\eps\bigl( q (p r), \eps \bigr)
		=
		b_\eps\bigl( (q p) r, \eps \bigr)
	\]
	and we conclude by induction as above.
	
	Now it is easy to see using lemmas \ref{wunit-artin}, \ref{b3-8}, and \ref{b3-9} that
	\( c_{i \mid jk}({-})_2 \),
	\( b_4({-}, {=}) \), and
	\( b_\eps({-}, {=}) \) are well defined on
	\( (R_{2 *}, \Delta_{2 *}) \). Since this
	\( \mathsf B_3 \)-factor-ring is also weakly unital, we conclude by lemmas \ref{b3-8} and \ref{b3-9}.
	
	It remains to prove that the morphisms from the statement generate the Schur multiplier. Indeed, take the factor-group
	\[
		G
		=
		\widetilde \stlin(\mathsf B_3, R, \Delta)
		/
		\bigl(
			c_{1 \mid 23}(\Delta_3)_3\,
			b_4(I_4)\,
			b_\eps(I_{2 \eps \delta})\,
			d(\Delta_{2 \mathrm b})
		\bigr).
	\]
	In this factor-group we have
	\( y_{ij'k}(p q, r) = y_{ijk}(p, q r) \) for
	\( [p, q, r] = 0 \), so there exists a unique homomorphism
	\( y_{ik} \colon R \to G \) such that
	\( y_{ijk}(p, q) = y_{ik}(p q) \) and
	\(
		y_{i0j}(u, v)
		=
		y_{ij}(-\lambda_{-i}^* \langle u, v \rangle)
	\) (the second identity follows from lemma \ref{b3-4}). Let
	\(
		y_{(i) j}(u, p)
		=
		z_{ij}(u, p)\,
		y_{-i, j}(\lambda_i^* \rho(u) p)^{-1}
	\), so
	\(
		\kappa(y_{(i) j}(u, p))
		=
		x_j(\dotminus u \cdot (-p))
	\) and
	\begin{align*}
		y_{(i) j}(u \dotplus v, p)
		&=
		y_{(i) j}(v, p)\,
		y_{(i) j}(u, p)
		&
		y_{(-i) j}(\phi(p), q) &= y_{-j, j}(q^* p q)^{-1},
		\\
		y_{(i) j}(u, p + q)
		&=
		y_{(i) j}(u, p)\,
		y_{-j, j}(-p^* \rho(u) q)\,
		y_{(i) j}(u, q)
		&
		y_{(j) k}(u \cdot p, q) &= y_{(i) k}(u, p q)
	\end{align*}
	by lemmas \ref{b3-3} and \ref{b3-4}. It follows that there exists a unique homomorphism
	\( y_j \colon \Delta \to G \) such that
	\( y_{(i) j}(u, p) = y_j(\dotminus u \cdot (-p)) \) and
	\( y_{-j, j}(p) = y_j(\phi(p)) \). It is easy to see that
	\( y_{ij}(p) \) and
	\( y_i(u) \) satisfy all Steinberg relations, so the homomorphism
	\(
		\kappa
		\colon
		G \to \stlin(\mathsf B_3, R, \Delta)
	\) admits a section and
	\( G \cong \stlin(\mathsf B_3, R, \Delta) \).
\end{proof}

\subsection{Other doubly laced cases}

\begin{prop}
	\label{b-schur}
	Let
	\( (R, \Delta) \) be a weakly unital
	\( \mathsf B_\ell \)-ring in an infinitary pretopos,
	\( \ell \geq 4 \). Then the Steinberg group
	\( \stlin(\mathsf B_\ell, R, \Delta) \) is centrally closed.
\end{prop}
\begin{proof}
	Note that
	\[
		c_{i \mid jk}(u, p)
		=
		\bigl \langle x_i(u), x_{jk}(p) \bigr \rangle
		=
		1,
		\quad
		c_{ij \mid kl}(p, q)
		=
		\bigl \langle x_{ij}(p), x_{kl}(q) \bigr \rangle
		=
		1
	\]
	by (\ref{HW}) applied to the triples
	\[
		x_i(u),
		x_{jl}(p),
		x_{lk}(q);
		\quad
		x_{ij}(p),
		x_{-k}(u),
		x_{-k, l}(q).
	\]
	By proposition \ref{a3-schur} applied to root subsystems of type
	\( \mathsf A_3 \) there are unique homomorphisms
	\(
		y_{ij}
		\colon
		R
		\to
		\widetilde \stlin(\mathsf B_\ell, R, \Delta)
	\) such that
	\( \kappa(y_{ij}(p)) = x_{ij}(p) \) and
	\(
		\bigl\langle x_{ij}(p), x_{jk}(q) \bigr\rangle
		=
		y_{ik}(p q)
	\). But then proposition \ref{b3-schur} applied to root subsystems of type
	\( \mathsf B_3 \) means that there are unique homomorphisms
	\(
		y_i
		\colon
		\Delta
		\to
		\widetilde \stlin(\mathsf B_\ell, R, \Delta)
	\) such that
	\( y_{ij} \) and
	\( y_i \) determine a section of
	\( \kappa \).
\end{proof}

It remains to deal with the root system
\( \mathsf F_4 \). It is convenient to distinguish its long and short roots, so all root subsystems of the spherical type
\( \mathsf B_3 \) have the crystallographic type
\( \mathsf B_3 \) or
\( \mathsf C_3 \), and below we use only the crystallographic type. Let
\( (R, S) \) be a weakly unital
\( \mathsf F_4 \)-ring. By (\ref{HW}) applied to
\( x_\alpha(p) \),
\( x_\beta(q) \),
\( x_\gamma(u) \) for
\( |\alpha| = |\beta| \neq |\gamma| \),
\( \alpha \perp \gamma \perp \beta \),
\( \angle(\alpha, \beta) = \frac {2 \pi} 3 \) we get
\[
	\bigl\langle
		x_{\alpha + \beta}(p),
		x_\gamma(u)
	\bigr\rangle
	=
	1.
\]
In other words,
\( c_{i \mid jk}(u) = 1 \) for all root subsystems of types
\( \mathsf B_3 \) and
\( \mathsf C_3 \).

Now let
\( x_{ij} \colon R \to \stlin(\mathsf F_4, R, S) \) and
\( x_i \colon S \to \stlin(\mathsf F_4, R, S) \) be the root homomorphisms for the standard root subsystem
\( \mathsf B_3 \subseteq \mathsf F_4 \), and
\( x'_{ij} \colon S \to \stlin(\mathsf F_4, R, S) \) and
\( x'_i \colon R \to \stlin(\mathsf F_4, R, S) \) be the root homomorphisms for the standard root subsystem
\( \mathsf C_3 \subseteq \mathsf F_4 \), see \cite[\S 8]{root-graded} for details. We have
\begin{align*}
	x_1(u) &= x'_{1 2}(u),
	&
	x_{-1}(u) &= x'_{-1, -2}(\lambda u),
	&
	x_{1 2}(p) &= x'_1(p),
	&
	x_{-1, -2}(p) &= x'_{-1}(\lambda p),
	\\
	x_2(u) &= x'_{-1, 2}(\lambda u),
	&
	x_{-2}(u) &= x'_{1, -2}(\lambda u),
	&
	x_{-1, 2}(p)
	&=
	x'_2(\lambda p),
	&
	x_{1, -2}(p)
	&=
	x'_{-2}(\lambda p).
\end{align*}

By lemma \ref{b3-3} there are unique homomorphisms
\(
	y_{-i, i}
	\colon
	R \to \widetilde \stlin(\mathsf F_4, R, S)
\) for
\( 1 \leq |i| \leq 3 \) such that
\begin{align*}
	\kappa(y_{-i, i}(p)) &= x_i(\phi(p)),
	&
	y_{-i, i}((p q) r) &= y_{-i, i}(p (q r)),
	\\
	\bigl\langle x_{-i, j}(p), x_{ji}(q) \bigr\rangle
	&=
	y_{-i, i}(\lambda_i p q),
	&
	y_{-i, i}(p + \lambda p^*) = 1,
	\\
	1 = \bigl\langle x_i(u), x_i(v) \bigr\rangle
	&=
	y_{-i, i}(\phi(u v^*))
\end{align*}
and similarly for
\(
	y'_{-i, i}
	\colon
	S \to \widetilde \stlin(\mathsf F_4, R, S)
\). In particular,
\begin{align*}
	y_{1, 0, 2}(u, v) &= y'_{-1, 1}(v u^*),
	&
	y_{-1, 0, -2}(u, v) &= y'_{1, -1}(v u^*),
	\\
	y_{1, 0, -2}(u, v) &= y'_{2, -2}(v^* u),
	&
	y_{-1, 0, 2}(u, v) &= y'_{-2, 2}(v^* u),
	\\
	y_{-1, 1}(q p^*) &= y'_{1, 0, 2}(p, q),
	&
	y_{1, -1}(q p^*) &= y'_{-1, 0, -2}(p, q),
	\\
	y_{2, -2}(q^* p) &= y'_{1, 0, -2}(p, q),
	&
	y_{-2, 2}(q^* p) &= y'_{-1, 0, 2}(p, q).
\end{align*}
Also,
\begin{align*}
	z_{12}(u, p) &= z'_{12}(p, u)^{-1},
	&
	z_{1, -2}(u, p)
	&=
	z'_{-2, -1}(\lambda p, -u^*)^{-1},
	\\
	z_{-1, 2}(u, p)
	&=
	z'_{21}(\lambda p, -\lambda u^*)^{-1},
	&
	z_{-1, -2}(u, p)
	&=
	z'_{-1, -2}(\lambda p, \lambda u)^{-1},
	\\
	z_{21}(u, p)
	&=
	z'_{-1, 2}(-\lambda p^*, \lambda u)^{-1},
	&
	z_{2, -1}(u, p) &= z'_{-2, 1}(-p^*, -u^*)^{-1},
	\\
	z_{-2, 1}(u, p) &= z'_{2, -1}(-p^*, -u^*)^{-1},
	&
	z_{-2, -1}(u, p) &= z'_{1, -2}(-p^*, \lambda u)^{-1}.
\end{align*}
Next note that
\( y_{ij}(p, q) = y_{ikj}(p, q) \) is independent of the sign of
\( k \) by lemma \ref{b3-8} and similarly for
\( y'_{ij}(u, v) = y'_{ikj}(u, v) \). Lemma \ref{b3-4} implies
\begin{align*}
	y_{12}(\phi(u), p) &= y'_{-1, 1}(u^* \rho(p)),
	&
	y_{1, -2}(\phi(u), p) &= y'_{2, -2}(u \rho(p)),
	\\
	y_{-1, 2}(\phi(u), p) &= y'_{-2, 2}(u \rho(p)),
	&
	y_{-1, -2}(\phi(u), p) &= y'_{1, -1}(u^* \rho(p)),
	\\
	y_{-1, -2}(p, \phi(u)) &= y'_{1, -1}(\rho(p) u),
	&
	y_{1, -2}(p, \phi(u)) &= y'_{2, -2}(\rho(p) u^*),
	\\
	y_{-1, 2}(p, \phi(u)) &= y'_{-2, 2}(\rho(p) u^*),
	&
	y_{12}(p, \phi(u)) &= y'_{-1, 1}(\rho(p) u),
	\\
	y_{-1, 1}(p^* \rho(u)) &= y'_{12}(\phi(p), u),
	&
	y_{2, -2}(p \rho(u)) &= y'_{1, -2}(\phi(p), u),
	\\
	y_{-2, 2}(p \rho(u)) &= y'_{-1, 2}(\phi(p), u),
	&
	y_{1, -1}(p^* \rho(u)) &= y'_{-1, -2}(\phi(p), u),
	\\
	y_{1, -1}(\rho(u) p) &= y'_{-1, -2}(u, \phi(p)),
	&
	y_{2, -2}(\rho(u) p^*) &= y'_{1, -2}(u, \phi(p)),
	\\
	y_{-2, 2}(\rho(u) p^*) &= y'_{-1, 2}(u, \phi(p)),
	&
	y_{-1, 1}(\rho(u) p) &= y'_{12}(u, \phi(p)),
\end{align*}
so
\( y_{-i, i}(p q) = y_{-i, i}(q p) \). Below we use homomorphisms
\begin{align*}
	b_4
	&\colon
	I_4 \to \schur(\stlin(\mathsf F_4, R, S)),
	&
	b_\eps
	&\colon
	R_{2 \eps \delta}
	\to
	\schur(\stlin(\mathsf F_4, R, S)),
	\\
	b'_4
	&\colon
	J_{4'} \to \schur(\stlin(\mathsf F_4, R, S)),
	&
	b'_\eps
	&\colon
	S'_{2 \eps \delta'}
	\to
	\schur(\stlin(\mathsf F_4, R, S)),
\end{align*}
where
\( J_{4'} \leqt S_{4'} \) and
\( S'_{2 \eps \delta'} \) are the symmetric objects to
\( I_4 \) and
\( R_{2 \eps \delta} \).

\begin{lemma}
	\label{f4-1}
	There are unique homomorphisms
	\[
		e
		\colon
		S_{44} \to \schur(\stlin(\mathsf F_4, R, S)),
		\quad
		e'
		\colon
		R_{44} \to \schur(\stlin(\mathsf F_4, R, S))
	\]
	such that
	\[
		z_{12}(u, p)\,
		y_{-1, 2}(\rho(u), p)^{-1}\,
		y'_{-1, 2}(\rho(p), u)
		=
		e(u \rho(p))\, e'(\rho(u) (p^3 + p))
		=
		e(\rho(p) (u^3 + u))\, e'(p \rho(u)).
	\]
	Also,
	\[
		b_\eps(p) = b'_\eps(u) = 1,
		\quad
		b_4(p (q^2 - q)) = e'(p (q^2 - q)),
		\quad
		b'_4(u (v^2 - v)) = e(u (v^2 - v)),
		\quad
		e(\rho(p)) = e'(p^3).
	\]
\end{lemma}
\begin{proof}
	Let
	\begin{align*}
		e(u, p)
		&=
		z_{12}(u, p)\,
		y_{-1, 2}(\rho(u), p)^{-1}\,
		y'_{-1, 2}(\rho(p), u)
		\colon
		\schur(\stlin(\mathsf F_4, R, S)),
		\\
		e'(p, u)
		&=
		z'_{12}(p, u)\,
		y'_{-1, 2}(\rho(p), u)^{-1}\,
		y_{-1, 2}(\rho(u), p)
		=
		e(u, p)^{-1}.
	\end{align*}
	By lemma \ref{b3-4} we have
	\begin{align*}
		e(u + v, p) &= e(u, p)\, e(v, p),
		\\
		e(u, p + q) &= e(u, p)\, e(u, q),
		\\
		e(\phi(p), q)
		&=
		b_4\bigl( (p^2 + p) q \bigr)\,
		b_\eps\bigl( (p^2 + p) q \bigr),
		\\
		z_{32}(u, p q)\,
		y_{-3, 2}(\rho(u) p, q)^{-1}
		&=
		y'_{-1, 2}\bigl( \rho(q), u \rho(p) \bigr)^{-1}\,
		e(u \rho(p), q)\,
		b_4\bigl( (p^3 + p^2) q \rho(u) \bigr)\,
		b_\eps\bigl( (p^3 + p^2) q \rho(u) \bigr),
		\\
		z_{32}(u, \eps^2 q)
		&=
		y'_{-1, 2}\bigl( \rho(q), u \rho(\eps^2) \bigr)^{-1}\,
		y_{-3, 2}(\rho(u) \eps^2, q)\,
		e(u \rho(\eps^2), q)
		\text{ for }
		\eps \colon \mathcal E,
		\\
		e(u \rho(p), q)
		&=
		e(u, p q)\,
		b'_\eps\bigl(
			\bigl( \rho(p)^2 - \rho(p) \bigr) \rho(q) u
		\bigr)\,
		b_4\bigl( (p^3 + p) q \rho(u) \bigr)\,
		b_\eps\bigl( (p^3 + p) q \rho(u) \bigr).
	\end{align*}
	Then there is a unique homomorphism
	\(
		e \colon S \to \schur(\stlin(\mathsf F_4, R, S))
	\) such that
	\[
		e(u, p)
		=
		e(u \rho(p))\,
		b'_\eps\bigl(
			\bigl( \rho(p)^2 - \rho(p) \bigr) u
		\bigr)\,
		b_4\bigl( (p^3 + p) \rho(u) \bigr)\,
		b_\eps\bigl( (p^3 + p) \rho(u) \bigr),
	\]
	so
	\begin{align*}
		e(\phi(p q q^*))
		&=
		b_\eps\bigl( (p^2 - p) q \bigr)\,
		b'_\eps\bigl(
			\bigl( \rho(q)^2 + \rho(q) \bigr) \phi(p)
		\bigr)\,
		b_4\bigl( q^3 (p + p^2) \bigr),
		\\
		\tag{%
			\( \Omega \)%
		}
		\label{f4-1-phi}
		e(\phi(p)) &= b_\eps(p^2 + p)\, b_4(p^2 + p).
	\end{align*}
	By symmetry, there is a unique homomorphism
	\(
		e' \colon R \to \schur(\stlin(\mathsf F_4, R, S))
	\) such that
	\[
		e(u, p)^{-1}
		=
		e'(p \rho(u))\,
		b_\eps\bigl(
			\bigl( \rho(u)^2 - \rho(u) \bigr) p
		\bigr)\,
		b'_4\bigl( (u^3 + u) \rho(p) \bigr)\,
		b'_\eps\bigl( (u^3 + u) \rho(p) \bigr).
	\]
	But this means that
	\begin{align*}
		e(u \rho(p))
		&=
		e'(-p \rho(u))\,
		b_\eps\bigl( \rho(u)^2 p + \rho(u) p^3 \bigr)\,
		b_4\bigl( \rho(u) (p^3 + p) \bigr)\,
		b'_\eps\bigl( \rho(p)^2 u + \rho(p) u^3 \bigr)\,
		b'_4\bigl( \rho(p) (u^3 + u) \bigr),
		\\
		e(u v v^*)
		&=
		e(u v)\,
		b_\eps\bigl(
			\rho(u) \bigl( \rho(v)^3 + \rho(v)^2 \bigr)
		\bigr)\,
		b'_\eps\bigl( u (v^3 + v) \bigr)\,
		b'_4\bigl( u (v^3 + v) \bigr),
		\\
		e(2 u) &= 1,
		\\
		e(\lambda u) &= e(u),
		\\
		e'(p)
		&=
		e(\rho(p))\,
		b_\eps(p^3 + p)\,
		b_4(p^3 + p)\,
		b'_\eps(\rho(p)),
		\\
		e(u v^*)
		&=
		e(u v)\,
		b'_4\bigl( u (v^2 + v) \bigr).
	\end{align*}
	
	Now combine the linearization of the identity for
	\( e'(p) \) with (\ref{f4-1-phi}), so
	\[
		b'_\eps(\phi(p)) = 1.
	\]
	Using this and the dual identity we linearize the expression for
	\( e(u v v^*) \).
	\begin{align*}
		e\bigl( u (v w* + w v*) \bigr)
		&=
		b'_4\bigl( u (v^2 w + v w^2) \bigr)\,
		b'_\eps\bigl( u (v^2 w + v w^2) \bigr),
		\\
		b'_\eps\bigl( u (v^2 + v) \bigr)
		&=
		1,
		\\
		b'_\eps(u) &= b_\eps(p) = 1,
		\\
		e(u v^2) &= e(u v)\, b'_4(u (v^2 + v)),
		\\
		e(u (v w)) &= e(u (w v)).
	\end{align*}
	The relation (\ref{f4-1-phi}) and its dual take the form
	\[
		e(\phi(p)) = b_4(p^2 + p),
		\quad
		e(u + u^*) = b'_4(u^2 + u).
	\]
	
	As in the proof of proposition \ref{b3-schur} we see that
	\( e \) factors through
	\( S_{44} \). It follows that
	\( b'_4(u (v^2 + v)) = e(u (v^2 + v)) \), so we do not need the generators
	\( b'_4 \). Similarly,
	\( e' \) factors through
	\( R_{44} \) and allows us to eliminate
	\( b_4 \).
\end{proof}

\begin{prop}
	\label{f4-schur}
	Let
	\( (R, S) \) be a weakly unital
	\( \mathsf F_4 \)-ring in an infinitary pretopos. Then the Schur multiplier
	\( \schur(\stlin(\mathsf F_4, R, S)) \) is the product of the images of
	\[
		e
		\colon
		S_{44} \to \schur(\stlin(\mathsf F_4, R, S)),
		\quad
		e'
		\colon
		R_{44} \to \schur(\stlin(\mathsf F_4, R, S))
	\]
	satisfying the relation
	\( e(\rho(p)) = e'(p^3) \) (equivalently,
	\( e(u^3) = e'(\rho(u)) \)).
\end{prop}
\begin{proof}
	It remains to check that these homomorphisms generate the whole Schur multiplier. Let
	\( G \) be the factor-group of
	\( \widetilde \stlin(\mathsf F_4, R, S) \) by the product of their images.
	
	For any root subsystem
	\( \Psi \subseteq \mathsf F_4 \) of crystallographic type
	\( \mathsf B_3 \) there is a well defined homomorphism
	\( b_4 \colon I_4 \to G \), determined up to composition with the involution. If two such subsystems have a common root subsystem of type
	\( \mathsf A_2 \), then the corresponding homomorphisms coincide or differ by the involution. Also, by construction this homomorphism is trivial for the standard root subsystem
	\( \Psi_0 \) of crystallographic type
	\( \mathsf B_3 \). Now note that every
	\( \Psi \) contains a root
	\( \alpha = \pm \e_i \pm \e_j \) with
	\( i, j \neq 4 \). This root has two decompositions into a sum of long roots, one inside
	\( \Psi \) and another inside
	\( \Psi_0 \). But it is easy to see that every two such decompositions lie in a common root subsystem
	\( \Psi_1 \) of crystallographic type
	\( \mathsf B_3 \). So both
	\( \Psi_0 \cap \Psi_1 \) and
	\( \Psi \cap \Psi_1 \) contain root subsystems of type
	\( \mathsf A_2 \) and the homomorphism
	\( b_4 \) associated with
	\( \Psi \) is trivial. Similarly, Schur multipliers of root subsystems of crystallographic type
	\( \mathsf C_3 \) have trivial images in
	\( G \).
	
	It follows that there are well-defined homomorphisms
	\( y_\alpha \colon R \to G \) for long
	\( \alpha \) and
	\( y_\alpha \colon S \to G \) for short
	\( \alpha \) satisfying all Steinberg relations except possibly the ones with two roots of distinct lengths at the angle
	\( \frac {3 \pi} 4 \). We already know that one of these non-trivial relations holds.
	
	By the proof of lemma \ref{f4-1} (or directly by lemma \ref{b3-4})
	\[
		z_{32}(u, p q)
		=
		y_{-3, 2}(\rho(u) p, q)\,
		y'_{-1, 2}\bigl( \rho(q), u \rho(p) \bigr)^{-1}
		\colon
		G,
	\]
	so the required relation also holds for another pair of roots. Now consider the bipartite graph of all roots, where edges are the pairs of roots at the angle
	\( \frac \pi 4 \). If the relation holds for an edge, then we may assume that this edge corresponds to
	\( z_{12}(u, p) \), so the relation holds for all adjacent edges with the same short vertex. By symmetry, the relation holds for all edges in a connected component of this graph. But it is easy to see that the graph is connected.
\end{proof}

\section{Non-triviality of Schur multipliers}

\subsection{Root graded case}

Recall that Schur multipliers of finite simple groups of Lie type are known. We list all non-trivial Schur multipliers of elementary subgroups of simply laced reductive groups of isotropic rank
\( \geq 3 \) over finite fields following \cite{griess, steinberg-2}. In the table below
\( \Phi \) denotes an ordinary or twisted root system and
\( q \) is the order of the base field (for twisted groups it is a perfect square). It is known that the simply laced elementary group coincides with the Steinberg group and has
\( p \)-torsion Schur multiplier, its center is
\( p' \)-torsion, where
\( p \) is the characteristic of the base field. So the Schur multiplier of the Steinberg group is the
\( p \)-torsion of the Schur multiplier of the corresponding simple group
\( \Phi(q) \) of Lie type, i.e. the adjoint elementary group.

\begin{center}
	\begin{tabular}{|c||c|c|c|c|c|c|c|c|}
		\hline
		
		\( \Phi \)
		&
		\( \mathsf A_3 \)
		&
		\( \mathsf B_3 \)
		&
		\( \mathsf B_3 \)
		&
		\( \mathsf C_3 \)
		&
		\( \mathsf D_4 \)
		&
		\( \mathsf F_4 \)
		&
		\( \up 2 {\mathsf A_5} \)
		&
		\( \up 2 {\mathsf E_6} \)
		\\ \hline
		
		\( q \)
		&
		\( 2 \)
		&
		\( 2 \)
		&
		\( 3 \)
		&
		\( 2 \)
		&
		\( 2 \)
		&
		\( 2 \)
		&
		\( 4 \)
		&
		\( 4 \)
		\\ \hline
		
		\( \schur(\Phi(q))_p \)
		&
		\( \mathrm C_2 \)
		&
		\( \mathrm C_2 \)
		&
		\( \mathrm C_3 \)
		&
		\( \mathrm C_2 \)
		&
		\( \mathrm C_2 \times \mathrm C_2 \)
		&
		\( \mathrm C_2 \)
		&
		\( \mathrm C_2 \times \mathrm C_2 \)
		&
		\( \mathrm C_2 \times \mathrm C_2 \)
		\\ \hline
		
	\end{tabular}
\end{center}

The last two columns correspond to
\( \stlin(\mathsf B_3, \mathbb F_4, \mathbb F_2) \) and
\(
	\stlin(\mathsf F_4, \mathbb F_2, \mathbb F_4)
	\cong
	\stlin(\mathsf F_4, \mathbb F_4, \mathbb F_2)
\), also there is an isomorphism
\(
	\stlin(
		\mathsf B_3,
		\mathbb F_2,
		\mathbb F_2 \times \mathbb F_2
	)
	\cong
	\stlin(\mathsf D_4, \mathbb F_2)
\). Moreover, it is known \cite[theorems (2.6), (2.8), (2.10)]{schur-mult} that
\[
	\schur(\stlin(\mathsf A_3, \mathbb F_2[\eps]))
	\cong
	\schur(\stlin(\mathsf B_3, \mathbb F_2[\eps]))
	\cong
	\mathrm C_2 \times \mathrm C_2,
	\quad
	\schur(\stlin(\mathsf D_4, \mathbb F_2[\eps]))
	\cong
	\mathrm C_2 \times \mathrm C_2 \times \mathrm C_2,
\]
where
\( \mathbb F_2[\eps] \cong \mathbb F_2[X] / (X^2) \) is the algebra of dual numbers over
\( \mathbb F_2 \).

\begin{theorem}
	\label{schur-graded}
	Let
	\( \Phi \) be an irreducible spherical root system of rank at least
	\( 3 \). Let also
	\( A \) be a unital
	\( \Phi \)-ring. Then the Steinberg group
	\( \stlin(\Phi, A) \) is centrally closed unless
	\(
		\Phi
		\in
		\{
			\mathsf A_3,
			\mathsf B_3,
			\mathsf D_4,
			\mathsf F_4
		\}
	\), in these exceptional cases the Schur multuplier is given by generators and relations from propositions
	\ref{a3-schur}--\ref{f4-schur} (i.e. there are no other relations).
\end{theorem}
\begin{proof}
	First of all recall that
	\[
		\stlin(\mathsf H_3, A)
		\cong
		\stlin(\mathsf D_6, K),
		\quad
		\stlin(\mathsf H_4, A)
		\cong
		\stlin(\mathsf E_8, K)
	\]
	by
	\cite{h-graded} for a suitable commutative unital ring
	\( K \), so these groups are centrally closed. So from now on we can assume that
	\( \Phi \) is one of the four exceptional root systems.
	
	Without loss of generality
	\( A \) is finitely generated and lies in one of the distinguished varieties, so it is finite. Propositions from the previous section give us a surjective homomorphism
	\( X \to \schur(\stlin(\Phi, A)) \), where
	\( X \) is a certain explicit abelian group. In order to check that no
	\( 0 \neq x \in X \) maps to
	\( 1 \), it suffices to find a further quotient of
	\( A \) with known Schur multiplier of the Steinberg group and with non-trivial image of
	\( x \). By the cited results this is easy in all cases except possibly
	\( \Phi = \mathsf B_3 \) and the generators
	\(
		b_\eps
		\colon
		I_{2 \eps \delta}
		\to
		\schur(\stlin(\mathsf B_3, R, \Delta))
	\).
	
	Let
	\(
		(R, \Delta)
		=
		(R_{2 \eps \delta}, \Delta_{2 \eps \delta})
	\) be a unital
	\( \mathsf B_3 \)-ring from the corresponding variety, where
	\( R \) is indecomposable into direct product. Then
	\( R = V \rtimes \mathbb F_2 \) for some finite vector field
	\( V \) over
	\( \mathbb F_2 \). If there is
	\( v \in \Delta \) such that
	\( \langle \iota, v \rangle = 1 \), then we can replace
	\( v \) by
	\( v \dotplus \iota \cdot p \) in such a way that
	\( \rho(v) \) becomes
	\( 0 \) and the axioms imply that
	\( \Delta = \iota \cdot R \dotoplus v \cdot R \) is completely determined. In this case
	\(
		\stlin(\mathsf B_3, R, \Delta)
		\cong
		\stlin(\mathsf D_4, R)
	\), so the Schur multiplier is known. So from now on we may assume that
	\( \langle \iota, u \rangle^2 = 0 \) for all
	\( u \in \Delta \), then
	\( \langle u, v \rangle^2 = 0 \) for all
	\( u, v \in \Delta \). Also, for every
	\( u \) either
	\( \rho(u) = 0 \) or
	\( \rho(u) = 1 + \langle \iota, u \rangle \).
	
	Let
	\( \Delta_0 = \{ u \in \Delta \mid \rho(u) = 0 \} \). It follows that
	\( \Delta_0 \) is a subgroup,
	\( \langle u, v \rangle = 0 \) for all
	\( u, v \in \Delta_0 \), and
	\(
		\Delta
		=
		\Delta_0 \sqcup (\Delta_0 \dotplus \iota)
	\). The operations on
	\( \Delta \) are completely determined by the maps
	\[
		g \colon V \to \Delta_0,\,
		g(p) = \iota \cdot p,
		\quad
		f \colon \Delta_0 \to V,\,
		f(u) = \langle \iota, u \rangle,
	\]
	and the
	\( R \)-module structure on
	\( \Delta_0 \). The map
	\( f \) is
	\( R \)-linear and
	\( f(g(p)) = 0 \). Replacing
	\( \Delta \) by
	\( \Delta / \Ker(f) \) we can assume that
	\( f \) is injective and
	\( \Delta_0 \cdot V = g(V) = \dot 0 \). Moreover, by taking a suitable factor-algebra of
	\( (R, \Delta) \) and enlarging
	\( \Delta_0 \) if necessary we can assume that
	\( R = \mathbb F_2[\eps] \) and
	\(
		\Delta
		=
		\iota \cdot \mathbb F_2
		\dotoplus
		v \cdot \eps \mathbb F_2
	\), where
	\( \eps^2 = 0 \),
	\( \rho(v) = 0 \), and
	\( \langle \iota, v \rangle = 1 \). This
	\( \mathsf B_3 \)-ring is a factor-algebra of
	\[
		(
			R,
			\iota \cdot R
			\dotoplus
			v \cdot \eps \mathbb F_2
		)
		\subseteq
		(R, \iota \cdot R \dotoplus v \cdot R).
	\]
	By lemma \ref{b3-4} there are unique homomorphisms
	\(
		y_i
		\colon
		\iota \cdot \eps \mathbb F_2
		\to
		\widetilde \stlin(
			\mathsf B_3,
			R,
			\iota \cdot R \dotoplus v \cdot R
		)
	\) such that
	\(
		\bigl\langle
			x_i(\iota \cdot p),
			x_{ij}(q)
		\bigr\rangle
		=
		y_j(\iota \cdot p q)
	\) for
	\( p \in \eps \mathbb F_2 \) and
	\( q \in R \). Now it suffices to check that
	\( b_\eps(\eps) \) does not lie in the normal subgroup
	\[
		N
		\leqt
		\Image\bigl(
			\widetilde \stlin(
				\mathsf B_3,
				R,
				\iota \cdot R \dotoplus v \cdot \eps \mathbb F_2
			)
			\to
			\widetilde \stlin(
				\mathsf B_3,
				R,
				\iota \cdot R \dotoplus v \cdot R
			)
			/
			d(
				\iota \cdot \mathbb F_2
				\dotoplus
				v \cdot \mathbb F_2
			)
		\bigr)
	\]
	generated by all
	\( y_i(\iota \cdot \eps) \).
	
	Recall that
	\(
		\stlin(
			\mathsf B_3,
			R,
			\iota \cdot R \dotoplus v \cdot R
		)
		\cong
		\stlin(\mathsf D_4, R)
	\) in such a way that the long roots
	\( \pm \e_i \pm \e_j \) of
	\( \mathsf B_3 \) correspond to themselves in
	\( \mathsf D_4 \) with the same parametrizations of root subgroups by
	\( R \) and the roots
	\( \pm \e_i \) correspond to the pairs
	\( (\pm \e_i - \e_4, \pm \e_i + \e_4) \). Actually, there are two such isomorphisms, choose any of them. Let us identify these groups. Combining notations from proposition \ref{d4-schur} and \ref{b3-schur} we see that
	\( b_\eps(p) = c_0(p^2 - p) \),
	\(
		c_{i \mid jk}(\iota \cdot p \dotplus v \cdot q)
		=
		c_0(p^2 + q^2 - q)\, c_\pm(q^2)
	\), and
	\(
		d(\iota \cdot p \dotplus v \cdot q)
		=
		c_0(p^2)\, c_\pm(q^2)
	\), where the signs depend on the indices and the choice of the isomorphism.
	
	Now we need the explicit construction of the factor-group
	\(
		G
		=
		\widetilde \stlin(\mathsf D_4, R)
		/
		\bigl( c_-(\mathbb F_2)\, c_+(\mathbb F_2) \bigr)
	\) from \cite[\S 3]{schur-mult}. Namely,
	\[
		G
		=
		M
		\rtimes
		\mathrm{Spin}(8, \mathbb F_2),
	\]
	where the second factor coincides with the simply connected elementary group
	\( \elem^{\mathrm{sc}}(\mathsf D_4, \mathbb F_2) \) and the Steinberg group
	\( \stlin(\mathsf D_4, \mathbb F_2) \). The first factor is generated by elements
	\( C \),
	\( n_\alpha \),
	\( d_\alpha \) for
	\( \alpha \in \mathsf D_4 \) with the relations
	\begin{align*}
		[n_\alpha, n_\beta] &= 1
		\text{ for }
		\alpha \neq -\beta,
		&
		[n_\alpha, n_{-\alpha}] &= C,
		&
		C^2 &= 1,
		\\
		d_\alpha &= d_\beta
		\text{ for }
		\angle(\alpha, \beta)
		\in
		\{ 0, \frac \pi 2, \pi \},
		&
		[n_\alpha, d_\beta] &= 1,
		&
		n_\alpha^2 &= 1,
		\\
		d_{\e_1 + \e_2} d_{\e_2 + \e_3}
		&=
		d_{\e_1 + \e_3} C,
		&
		[d_\alpha, d_\beta]
		&=
		C^{\frac{2 (\alpha, \beta)}{(\alpha, \alpha)}},
		&
		d_\alpha^2 &= C,
	\end{align*}
	so there is an exact sequence
	\(
		1
		\to
		\mathrm C_2
		\to
		M
		\to
		\mathrm C_2^{26}
		\to
		1
	\). The generators
	\( n_\alpha \) maps to
	\(
		t_\alpha(\eps)
		\in
		\mathrm{Spin}(8, \mathbb F_2[\eps])
	\). The action of
	\( \stlin(\mathsf D_4, \mathbb F_2) \) on
	\( M \) is given by
	\begin{align*}
		\up{x_\alpha(1)}{n_\beta}
		&=
		n_\beta n_{\alpha + \beta} C_{ki} C_{jl}
		\text{ for }
		\alpha = \sigma \e_i + \tau \e_j,
		\beta = -\tau \e_j + \rho \e_k;
		\\
		\up{x_\alpha(1)}{n_\beta} &= n_\beta C
		\text{ for }
		\alpha \perp \beta;
		\\
		\up{x_\alpha(1)}{n_\beta} &= n_\beta
		\text{ for }
		(\alpha, \beta) > 0;
		\\
		\up{x_\alpha(1)}{n_{-\alpha}}
		&=
		n_{-\alpha} d_\alpha n_\alpha;
		\\
		\up {x_\alpha(1)} C &= C;
		\\
		\up{x_\alpha(1)}{d_\gamma}
		&=
		d_\gamma
		n_\alpha^{
			2 \frac{(\alpha, \gamma)}{(\alpha, \alpha)}
		}
		C^{2 \frac{(\beta, \gamma)}{(\alpha, \alpha)}}
		\text{ for }
		\alpha = \sigma \e_i + \tau \e_j,
		\beta = \e_j + \e_k,
		i < j,
		k < l;
	\end{align*}
	where
	\( i j k l \) is a permutation of
	\( \{ 1, 2, 3, 4 \} \),
	\( C_{ij} = C \) for
	\( i < j \),
	\( C_{ij} = 1 \) for
	\( i > j \) (so in the last identity the root
	\( \beta \) is uniquely determined by
	\( \alpha \)), and
	\( \sigma \),
	\( \tau \),
	\( \rho \) are signs.
	
	It is easy to see that
	\(
		y_i(\iota \cdot \eps)
		=
		n_{\e_i + \e_4} n_{\e_i - \e_4}
	\), where
	\( \e_{-i} = -\e_i \). The relations above imply that
	\( N \cong \mathrm C_2^6 \) is an abelian group generated by
	\( y_i(\iota \cdot \eps) \) and it does not contain
	\( b_\eps(\eps) = C \).
\end{proof}

\subsection{Tits indices}

We are ready to compute Schur multipliers of globally isotropic Steinberg groups. Let
\( K \) be a unital commutative ring and
\( G \) be a reductive group scheme over
\( K \) with isotropic pinning
\( (T, \Phi) \). For simplicity (see the next subsection for justification) we also assume that the root datum of the split form of
\( G \) is constant, its root system
\( \widetilde \Phi \) is irreducible, and the map
\( \widetilde \Phi \to \Phi \sqcup \{ 0 \} \) comes from an irreducible Tits index as in \cite{diophantine}. The only irreducible Tits indices with the rank of
\( \Phi \) at least
\( 3 \) are the following.
\begin{itemize}
	
	\item
	\( \up 1 {\mathsf A_{n, r}^{(d)}} \) for
	\( d (r + 1) = n + 1 \) and
	\( n \geq r \geq 3 \) with
	\( \widetilde \Phi = \mathsf A_n \) and
	\( \Phi = \mathsf A_r \), where
	\( R \) is an Azumaya algebra over
	\( K \) of degree
	\( d \).
	
	\item
	\( \up 2 {\mathsf A_{n, r}^{(d)}} \) for
	\( d \mid n + 1 \),
	\( 2 r d \leq n + 1 \), and
	\( n \geq r \geq 3 \) with
	\( \widetilde \Phi = \mathsf A_n \) and
	\( \Phi = \mathsf B_r \), where
	\( R \) is an Azumaya algebra over a \'etale quadratic extension of
	\( K \) of degree
	\( d \). The morphism
	\(
		\langle {-}, {=} \rangle
		\colon
		\Delta \times \Delta \to R
	\) generates
	\( R \) as an ideal for
	\( 2 r d \leq n \). If
	\( d = 1 \) (i.e.
	\( R \) itself is an \'etale quadratic extension of
	\( K \)) and
	\( 2 r d = n + 1 \), then the involution is standard,
	\( \lambda = -1 \),
	\( \Delta = K \),
	\( \rho \colon \Delta \to R \) is the embedding
	\( K \subseteq R \),
	\( \phi(p) = p + p^* \), and
	\( u \cdot p = u p p^*\).
	
	\item
	\( \mathsf B_{n, r} \) for
	\( n \geq r \geq 3 \) with
	\( \widetilde \Phi = \mathsf B_n \) and
	\( \Phi = \mathsf B_r \), where
	\( R = K \) with trivial involution and
	\( \lambda = 1 \),
	\( \Delta = M \) is a projective right
	\( K \)-module of rank
	\( 2 n - 2 r + 1 \) with quasi-regular quadratic form
	\( \rho \colon M \to K \),
	\( \phi(p) = 0 \),
	\( m \cdot p = m p \).
	
	\item
	\( \mathsf C_{n, r}^{(d)} \) for
	\( d = 2^k \mid 2 n \),
	\( r d \leq n \),
	\( n = r \) in the case
	\( d = 1 \),
	\( n \geq r \geq 3\) with
	\( \widetilde \Phi = \mathsf B_n \) and
	\( \Phi = \mathsf B_r \), where
	\( R \) is an Azumaya algebra over
	\( K \) of degree
	\( d \). If
	\( d = 1 \) (i.e.
	\( R = K \)), then the involution is trivial,
	\( \lambda = -1 \), and there is a canonical subgroup
	\( \phi(K) \leq \Delta^0 \leq \Delta \) such that
	\( \Delta^0 \cdot K \subseteq \Delta^0 \),
	\( \langle \Delta^0, \Delta \rangle = 0 \), the
	\( K \)-module
	\( \Delta / \Delta^0 \) is finitely generated projective of rank
	\( 2 n - 2 r d \), and
	\(
		\langle {-}, {=} \rangle
		\colon
		\Delta / \Delta_0 \times \Delta / \Delta_0
		\to
		K
	\) is a symplectic form (i.e. a non-degenerate alternating form). Moreover,
	\( \Delta^0 \) is also a
	\( K \)-module,
	\( u \cdot p = u p^2 \) for
	\( u \in \Delta^0 \), and
	\( \rho|_{\Delta^0} \colon \Delta^0 \to K \) is an isomorphism of
	\( K \)-modules.
	
	\item
	\( \up 1 {\mathsf D_{n, n}^{(1)}} \) for
	\( n \geq 4 \) with
	\( \widetilde \Phi = \Phi = \mathsf D_n \) and
	\( \Phi = \mathsf B_r \), where
	\( R = K \).
	
	\item
	\( \up 1 {\mathsf D_{n, r}^{(d)}} \) and
	\( \up 2 {\mathsf D_{n, r}^{(d)}} \) for
	\( d = 2^k \mid 2 n \),
	\( r d \leq n \),
	\( n \geq 4 \),
	\( r \geq 3 \), and
	\( r d < n \) or
	\( d = 1 \) with
	\( \widetilde \Phi = \mathsf D_n \) and
	\( \Phi = \mathsf B_r \), where
	\( R \) is an Azumaya algebra over
	\( K \) of degree
	\( d \). If
	\( d = 1 \) (i.e.
	\( R = K \)), then the involution is trivial,
	\( \lambda = 1 \),
	\( \Delta = M \) is a projective right
	\( K \)-module of rank
	\( 2 n - 2 r \) with regular quadratic form
	\( \rho \colon M \to K \),
	\( \phi(p) = 0 \),
	\( m \cdot p = m p \).
	
	\item
	\( \mathsf E^{28}_{7, 3} \) with
	\( \widetilde \Phi = \mathsf E_7 \) and
	\( \Phi = \mathsf B_3 \), where
	\( R \) is an octonion algebra over
	\( K \) with the standard involution and
	\( \lambda = -1 \),
	\( \Delta = K \) with
	\( \iota = 1 \),
	\( \phi(p) = p + p^* \),
	\( \rho(u) = u \),
	\( \langle u, v \rangle = 0 \),
	\( u \cdot p = u p p^* \).
	
	\item
	\( \mathsf F^0_{4, 4} \) with
	\( \widetilde \Phi = \Phi = \mathsf F_4 \), where
	\( R = S = K \) with trivial involutions,
	\( \phi(p) = 0 \) and
	\( \rho(p) = p \) for
	\( p \in R \),
	\( \phi(u) = 2 u \) and
	\( \rho(u) = u^2 \) for
	\( u \in S \).
	
	\item
	\( \up 2 {\mathsf E^2_{6, 4}} \) with
	\( \widetilde \Phi = \mathsf E_6 \) and
	\( \Phi = \mathsf F_4 \), where
	\( R = K \) with trivial involution,
	\( S \) is a \'etale quadratic extension of
	\( K \) with the standard involution,
	\( \phi(p) = 0 \) and
	\( \rho(p) = p \) for
	\( p \in R \),
	\( \phi(u) = u + u^* \) and
	\( \rho(u) = u u^* \) for
	\( u \in S \).
	
	\item
	\( \mathsf E^9_{7, 4} \) with
	\( \widetilde \Phi = \mathsf E_7 \) and
	\( \Phi = \mathsf F_4 \), where
	\( R = K \) with trivial involution,
	\( S \) is a quaternion algebra over
	\( K \) with the standard involution,
	\( \phi(p) = 0 \) and
	\( \rho(p) = p \) for
	\( p \in R \),
	\( \phi(u) = u + u^* \) and
	\( \rho(u) = u u^* \) for
	\( u \in S \).
	
	\item
	\( \mathsf E^{28}_{8, 4} \) with
	\( \widetilde \Phi = \mathsf E_8 \) and
	\( \Phi = \mathsf F_4 \), where
	\( R = K \) with trivial involution,
	\( S \) is an octonion algebra over
	\( K \) with the standard involution,
	\( \phi(p) = 0 \) and
	\( \rho(p) = p \) for
	\( p \in R \),
	\( \phi(u) = u + u^* \) and
	\( \rho(u) = u u^* \) for
	\( u \in S \).
	
	\item
	\( \up 1 {\mathsf E^0_{6, 6}} \) with
	\( \widetilde \Phi = \Phi = \mathsf E_6 \) and
	\( R = K \).
	
	\item
	\( \mathsf E^0_{7, 7} \) with
	\( \widetilde \Phi = \Phi = \mathsf E_7 \) and
	\( R = K \).
	
	\item
	\( \mathsf E^0_{8, 8} \) with
	\( \widetilde \Phi = \Phi = \mathsf E_8 \) and
	\( R = K \).
	
\end{itemize}
For simplicity we write the spherical types of the root systems
\( \widetilde \Phi \) and
\( \Phi \). Definitions and lists of Tits indices are given in
\cite{index-loc, index-field}. The claims about
\( \Phi \)-rings in these cases can be proved as follows. They trivially hold in the split case, for classical cases by explicit calculations and for exceptional cases see
\cite[examples 3 and 6]{root-graded}. In the non-split case note that the above descriptions are preserved under descent and independent of the choices of Weyl elements (cf.
\cite[theorems 5 and 7]{root-graded}). See also
\cite{odd-petrov} and
\cite{twisted-forms} for descriptions of classical cases as odd unitary groups.

For a unital commutative ring
\( K \) let
\( K_2 = K / \langle 2, p^2 - p \rangle \),
\(
	K_{2 \eps}
	=
	K / \bigl\langle 2, (p^2 - p) (q^2 - q) \bigr\rangle
\), and
\( K_3 = K / \langle 3, p^3 - p \rangle \). If
\( R \) is an \'etale quadratic algebra over
\( K \), then
\( R_4 = R \otimes_K K_2 \) induces a decomposition
\(
	K_2
	=
	K_{2 \mathrm s} \times K_{2 \mathrm a}
\) into a ``\textbf split'' and ``\textbf anisotropic'' parts such that
\(
	R_4 \otimes_{K_2} K_{2 \mathrm s}
	\cong
	K_{2 \mathrm s} \times K_{2 \mathrm s}
\) and
\(
	R_4 \otimes_{K_2} K_{2 \mathrm a}
	\cong
	K_{2 \mathrm a} \otimes_{\mathbb F_2} \mathbb F_4
\) (though these isomorphisms are non-canonical). Let also
\(
	K_{2 \mathrm s \eps}
	=
	K_{2 \eps} \otimes_{K_2} K_{2 \mathrm s}
\) (over this ring
\( R \) still splits) and
\(
	K_{2 \mathrm s \eps}
	=
	K_{2 \eps} \otimes_{K_2} K_{2 \mathrm s}
\) (over this ring
\(
	R
	\cong
	K_{2 \mathrm s \eps} \otimes_{\mathbb F_2} \mathbb F_4
\)).

\begin{theorem}
	\label{schur-tits}
	Let
	\( G \) be a reductive group scheme over a unital commutative ring
	\( K \) with isotropic pinning
	\( (T, \Phi) \) such that the map
	\( \widetilde \Phi \to \Phi \sqcup \{ 0 \} \) comes from an irreducible Tits index. Then the Steinberg group
	\( \stlin_G(K) \) is centrally closed with the following exceptions.
	\begin{center}
		\begin{tabular}{|c||c|c|c|c|c|c|c|c|}
			\hline
			
			Tits index
			&
			\( \up 1 {\mathsf A^{(1)}_{3, 3}} \)
			&
			\( \up 2 {\mathsf A^{(1)}_{5, 3}} \)
			&
			\( \mathsf B_{3, 3} \)
			&
			\( \mathsf C^{(1)}_{3, 3} \)
			&
			\( \up 2 {\mathsf D_{4, 3}^{(1)}} \)
			&
			\( \up 1 {\mathsf D_{4, 4}^{(1)}} \)
			&
			\( \mathsf F^0_{4, 4} \)
			&
			\( \up 2 {\mathsf E^2_{6, 4}} \)
			\\ \hline
			
			\( \schur(\stlin_G(K)) \)
			&
			\( K_{2 \eps} \)
			&
			\( K_{2 \mathrm a} \times K_{2 \mathrm a} \)
			&
			\( K_3 \times K_{2 \eps} \)
			&
			\( K_2 \)
			&
			\(
				K_{2 \mathrm s}
				\times
				K_{2 \mathrm s \eps}
			\)
			&
			\( K_2 \times K_{2 \eps} \)
			&
			\( K_2 \)
			&
			\( K_{2 \mathrm a} \times K_{2 \mathrm a} \)
			\\ \hline
			
		\end{tabular}
	\end{center}
\end{theorem}
\begin{proof}
	Just apply theorem \ref{schur-graded} to the above list of irreducible Tits indices. Note that the Schur multiplier is trivial if
	\( R \) is a non-commutative Azumaya algebra or an octonion algebra.
\end{proof}

\subsection{Locally isotropic case}

Recall that
\( \R \in \mathbf P_K \) is the affine line over
\( K \) considered as a presheaf of rings. We need its factor-rings
\( \R_2 \colon R \mapsto R_2 \),
\( \R_3 \colon R \mapsto R_3 \),
\( \R_{2 \eps} \colon R \to R_{2 \eps} \),
and similarly for the subscripts
\( 2 \mathrm s \),
\( 2 \mathrm a \),
\( 2 \mathrm s \eps \),
\( 2 \mathrm a \eps \). The symbols
\( \mathrm s \) and
\( \mathrm a \) correspond to the loci where the group scheme
\( G \) splits or not.

\begin{theorem}
	\label{schur-loc}
	Let
	\( G \) be an adjoint semisimple group scheme over a unital commutative ring
	\( K \) such that its local isotropic rank is at least
	\( 3 \) and the root system
	\( \widetilde \Phi \) of its split form is constant and irreducible. Then the Steinberg group object
	\( \stlin_G(\R) \in \mathbf U_K \) is centrally closed with the following exceptions.
	\begin{center}
		\begin{tabular}{|c||c|c|c|c|c|c|c|}
			\hline
			
			\( \widetilde \Phi \)
			&
			\( \mathsf A_3 \)
			&
			\( \mathsf A_5 \)
			&
			\( \mathsf B_3 \)
			&
			\( \mathsf C_3 \)
			&
			\( \mathsf D_4 \)
			&
			\( \mathsf F_4 \)
			&
			\( \mathsf E_6 \)
			\\ \hline
			
			\( \schur(\stlin_G(\R)) \)
			&
			\( \R_{2 \eps} \)
			&
			\( \R_{2 \mathrm a} \times \R_{2 \mathrm a} \)
			&
			\( \R_3 \times \R_{2 \eps} \)
			&
			\(
				\R_2
			\)
			&
			\(
				\R_{2 \mathrm s}
				\times
				\R_{2 \mathrm s \eps}
			\)
			&
			\( \R_2 \)
			&
			\( \R_{2 \mathrm a} \times \R_{2 \mathrm a} \)
			\\ \hline
			
		\end{tabular}
	\end{center}
	The group
	\( \mathrm{Aut}(G)(K) \) acts non-trivially on this Schur multiplier in the case
	\( \mathsf B_3 \) by the spin norm (recall that
	\(
		\mathrm{SO}(7, \mathbb F_3)
		/
		\mathrm{EO}(7, \mathbb F_3)
		\cong
		\mathrm C_2
	\)), in the cases
	\( \mathsf A_5 \) and
	\( \mathsf E_6 \) by the permutation action of
	\( \mathrm{Out}(G)(\mathbb F_2) \cong \mathrm S_2 \), and in the case
	\( \mathsf D_4 \) by the standard representation of
	\(
		\mathrm{Out}(G)(\mathbb F_2)
		\cong
		\mathrm S_3
		\cong
		\mathrm{GL}(2, 2)
	\) on
	\( \R_{2 \mathrm s} \times \R_{2 \mathrm s} \).
\end{theorem}
\begin{proof}
	Denote the functor appearing in the low line of above table by
	\( ({-})_* \) (and
	\( ({-})_* = 0 \) for other root systems), i.e.
	\( \R_* \) is either a factor-ring of
	\( \R \) or a product of two such factor-rings. Take
	\( s \in K \) such that
	\( G_s \) has an isotropic pinning of rank at least
	\( 3 \) and the corresponding map between root systems comes from an irreducible Tits index. We claim that the epimorphism
	\[
		\bigl( \R^{(s^\infty)} \bigr)_*
		\to
		\schur\bigl(
			\stlin_G\bigl( \R^{(s^\infty)} \bigr)
		\bigr)
	\]
	from propositions \ref{a3-schur}--\ref{f4-schur} and its composition with
	\[
		\schur\bigl(
			\stlin_G\bigl( \R^{(s^\infty)} \bigr)
		\bigr)
		\to
		\schur\bigl(
			\stlin_G\bigl( \R_s \bigr)
		\bigr)
	\]
	are isomorphisms. Indeed, note that
	\[
		\bigl( \R^{(s^\infty)} \bigr)_*
		\cong
		\bigl( \R_* \bigr)^{(s^\infty)}
		\to
		(\R_*)_s
		\cong
		(\R_s)_*
	\]
	are isomorphisms, because both localization and co-localization in the middle just take the direct factor determined by an idempotent power of
	\( s \) (namely, by
	\( s \) itself or
	\( s^2 \)). Finally,
	\( (\R_s)_* \to \schur\bigl( \stlin_G(\R_s) \bigr) \) is an isomorphism of presheaves by theorem \ref{schur-tits}.
	
	The action of
	\( \mathrm{Aut}(G)(K_s) \) on the Schur multiplier
	\(
		\schur\bigl( \stlin_G(\R^{(s^\infty)}) \bigr)
		\cong
		\schur\bigl( \stlin_G(\R_s) \bigr)
	\) can be calculated using lemma \ref{c-ext-xmod} locally in Zariski topology. Recall that every element of
	\( G(K_s) \) locally decomposes into a product of root elements and an element of
	\( L(K_s) \), the actions of all factors can be calculated explicitly. Moreover, every element of
	\( \mathrm{Aut}(G)(K_s) \) decomposes into a product of element of
	\( G(K_s) \) and a standard outer automorphism if
	\( K_s \) is a finite field or
	\( \mathbb F_2[\eps] \).
	
	Now take a partition of unity
	\( K = \sum_{i = 1}^n K s_i \) such that
	\( G_{s_i} \) have isotropic pinnings of rank at least
	\( 3 \) and the maps from
	\( \widetilde \Phi \) come from irreducible Tits indices. We have a diagram consisting of homomorphisms
	\(
		(\R_*)^{((s_i s_j)^\infty)}
		\to
		(\R_*)^{(s_i^\infty)}
	\), its objects the Schur multipliers of Steinberg groups of colocalizations of
	\( \R \). Dually, there is a similar diagram
	\( (\R_*)_{s_i} \to (\R_*)_{s_i s_j} \) of the Schur multipliers of Steinberg groups of localizations of
	\( \R \). Both diagrams consist of finitely generated projective modules over
	\( \R_3 \times \R_{2 \eps} \) (the maps are linear by the description of the action of
	\( \mathrm{Aut}(G)(K_{s_i s_j}) \)) and the corresponding objects of these diagrams are canonically isomorphic. Clearly, the diagram for the localizations has the limit
	\( \R_* \) because the gluing homomophisms satisfy the cocycle condition (recall that the Schur mulipliers for
	\( \R_{s_i s_j s_k} \) are known) and all finitely generated projective modules over
	\( \R_3 \times \R_{2 \eps} \) decompose into direct sums of free modules over various direct factors of this ring. It follows that the diagram for colocalizations can be continued to a cocone with the vertex
	\( \R_* \), and this vertex is the colimit by \cite[lemma 5]{iso-st-k2}.

	The resulting homomorphism
	\( \R_* \to \schur(\stlin_G(\R)) \) is an epimorphism by lemma \ref{c-ext-xmod}. On the other hand, the composition
	\(
		\R_*
		\to
		\schur(\stlin_G(\R))
		\to
		\schur(\stlin_G(\R_{s_i}))
	\) coincides with the projection
	\( \R_* \to (\R_*)_{s_i} \) for all
	\( i \), so our homomorphism is a monomorphism.
\end{proof}

\section{Existence of locally isotropic Steinberg groups}

We need several obvious lemmas to reduce arbitrary locally isotropic reductive groups to the groups considered in theorem \ref{schur-tits}. In these lemmas
\( K \) is a unital commutative ring and
\( G \) is a reductive group scheme over
\( K \) of local isotropic rank at least
\( 3 \).

\begin{lemma}
	\label{ring-prod}
	If
	\( K = \prod_{i = 1}^n K_i \), then
	\(
		\stlin_G(\R)
		\to
		\prod_{i = 1}^n \stlin_G(\R_i)
	\) is an isomorphism, where
	\( \R_i(E) = E \otimes_K K_i \).
\end{lemma}

\begin{lemma}
	\label{isogeny}
	The canonical homomorphism
	\( \stlin_G(\R) \to \stlin_{G / \Cent(G)}(\R) \) is an isomorphism, where
	\( G / \Cent(G) \) is the scheme factor-group (a semisimple group scheme of adjoint type).
\end{lemma}

\begin{lemma}
	\label{group-prod}
	If
	\( G = \prod_{i = 1}^n G_i \), then
	\(
		\stlin_G(\R)
		\to
		\prod_{i = 1}^n \stlin_{G_i}(\R)
	\) is an isomorphism.
\end{lemma}

\begin{lemma}
	\label{weyl-restr-rk}
	Suppose that
	\( G = \mathrm R_{K' / K}(G') \) (a Weyl restriction), where
	\( K' / K \) is a finite \'etale extension and
	\( G' \) is an adjoint semisimple group scheme over
	\( K' \). Then the local isotropic rank of
	\( G \) is at most the local isotropic rank of
	\( G' \).
\end{lemma}
\begin{proof}
	Without loss of generality
	\( K \) is local and
	\( G' \) is indecomposable into a direct product. Choose a maximal isotropic pinning
	\( (T, \Phi) \) of
	\( G \). It is easy to see that the homomorphism
	\( T_{K'} \to G' \) of group schemes over
	\( K' \) has trivial kernel (because
	\( G' \) is indecomposable), so it is a closed embedding. The root system of the action of
	\( T_{K'} \) on the Lie algebra
	\( \mathfrak g' \) of
	\( G' \) is precisely
	\( \Phi \) again using that
	\( G' \) is indecomposable, and the rank of
	\( (T_{K'}, \Phi) \) is the same as of
	\( (T, \Phi) \).
\end{proof}

\begin{lemma}
	\label{weyl-restr}
	Suppose that
	\( G = \mathrm R_{K' / K}(G') \), where
	\( K' / K \) is a finite \'etale extension and
	\( G' \) is an adjoint semisimple group scheme over
	\( K' \). Then
	\( \stlin_G(\R) \cong \stlin_{G'}(\R') \), where
	\( \R'(E) = E \otimes_K K' \).
\end{lemma}

\begin{theorem}
	\label{compactness}
	Let
	\( K \) be a unital commutative ring and
	\( G \) be a reductive group scheme over
	\( K \) of local isotropic rank at least
	\( 3 \). Then the Steinberg group object
	\( \stlin_G(\R) \) lies in
	\( \Ex(\Ind(\mathbf P_K)) \) up to isomorphism, so the Steinberg group functor
	\[
		\stlin_G \colon \Ring_K \to \Group,\,
		E \mapsto \ev_E(\stlin_G(\R))
	\]
	is well defined.
\end{theorem}
\begin{proof}
	By lemma \ref{ring-prod} we can assume that the root datum of the split form of
	\( G \) is constant. Next apply lemma \ref{isogeny} to reduce to the case when
	\( G \) is adjoint semisimple. Further by lemma \ref{group-prod} we can assume that all components of the root system of the split form of
	\( G \) are of the same type. By \cite[XXIV, proposition 5.9]{sga3}
	\( G \cong \mathrm R_{K' / K}(G') \) is a Weyl restriction for a finite \'etale extension
	\( K' / K \) and some semisimple group scheme
	\( G' \) over
	\( K' \), still adjoint and of local isotropic rank at least
	\( 3 \) by lemma
	\ref{weyl-restr-rk}, so using lemma \ref{weyl-restr} we can assume that the root system of the split form of
	\( G \) is irreducible. We have
	\[
		\stlin_G(\R)
		=
		\widetilde \elem_G(\R)
		/
		\schur(\widetilde \stlin_G(\R)),
	\]
	where both
	\( \widetilde \elem_G(\R) \) and
	\( \schur(\widetilde \stlin_G(\R)) \) lie in
	\( \Ex(\Ind(\mathbf P_K)) \). Indeed, the first one is the universal central extension of the elementary subgroup, and
	\( \elem_G(\R) \) lies in
	\( \Ind(\mathbf P_K) \) by \cite[theorem 1]{iso-elem}. On the other hand, the Schur multiplier lies in
	\( \mathbf P_K \) by theorem \ref{schur-loc}.
\end{proof}

Theorem \ref{schur-loc} and the proof of theorem \ref{compactness} (namely, the lemmas from this section) allow us to explicitly compute the Schur multiplier
\( \schur(\stlin_G(K)) \) under assumptions of the last theorem.

\bibliographystyle{plain}
\bibliography{references}

\end{document}